\documentclass[12pt,a4paper,reqno]{article}
\usepackage{amsmath,amsxtra,amssymb,latexsym, amscd,amsthm,mathrsfs}
\usepackage{relsize}
\usepackage[bottom]{footmisc}
\usepackage{color}
\usepackage{graphics,graphicx}
\usepackage{tikz-cd}
\usetikzlibrary{positioning}
\usepackage{fancyhdr}
\usetikzlibrary{matrix,arrows,decorations.pathmorphing}
\usepackage[left=1.1in,right=1.1in,top=1in,bottom=1in]{geometry}    
\theoremstyle{plain}
\newtheorem{s}{Section}[section]
\newtheorem{corollary}[s]{Corollary}

\newtheorem{theorem}[s]{Theorem}
\newtheorem{lemma}[s]{Lemma}
\newtheorem{proposition}[s]{Proposition}
\theoremstyle{definition}
\newtheorem{example}[s]{Example}
\newtheorem{definition}[s]{Definition}
\theoremstyle{remark}
\newtheorem{remark}[s]{Remark}

\usepackage{titling}
\usepackage{lipsum}

\pretitle{\begin{center}\Huge\bfseries}
\posttitle{\par\end{center}\vskip 0.5em}
\preauthor{\begin{center}\Large}
\postauthor{\end{center}}
\predate{\par\large\centering}
\postdate{\par}

\title{Average size of 2-Selmer groups of Jacobians of hyperelliptic curves over function fields}
\author{Dao Van Thinh}
\date{\today}
\begin{document}

\maketitle
\thispagestyle{empty}

\begin{abstract}
In this paper, we are going to compute the average size of 2-Selmer groups of two families of hyperelliptic curves with marked points over function fields. The result will be obtained by a geometric method which could be considered as a generalization of the one that was used previously in \cite{HLN14} to obtain the average size of 2-Selmer groups of elliptic curves.
\end{abstract}

\tableofcontents

\section{Introduction}
Let $K$ be a global field (number fields or function fields), and $A$ be an Abelian variety over $K$. The abelian group of $K-$rational points $A(K)$ has been investigated widely, and there are a great number of interesting questions related to it Mathematicians are trying to answer. One of them is to find the rank of $A(K)$. By Mordell-Weil theorem, we know that the rank of $A(K)$ is finite. However, there is no known algorithm that always returns the rank of $A(K)$. Beside the Mordell group $A(K)$, we also have other important groups associated to $A$ such as Selmer groups and the Tate-Shafarevich group. More precisely, the Tate-Shafarevich group of $A/K$ is the group
$$TS(A/K) = ker \bigg( H^1(K,A) \xrightarrow{res} \prod_v H^1(K_v,A) \bigg),$$where on the right hand side, we take the product over all places $v$ of $K$, and $K_v$ denotes the completion of $K$ at the place $v$. In addition, for a given isogeny $\phi: A \rightarrow A'$ of abelian varieties over $K$, we define the $\phi-Selmer$ group to be
$$Sel^{\phi}(A/K)= ker \bigg( H^1(K,A[\phi]) \xrightarrow{res} \prod_v H^1(K_v,A) \bigg),$$where $A[\phi]= ker(\phi)$. Roughly speaking, the size of the Tate-Shafarevich group $TS(A/K)$ measures the failure of the local-to-global principle for principal homogeneous spaces for $A/K$, and the homogeneous spaces corresponding to elements of $Sel^{\phi}(A/K)$ possess $K_v-$rational points for every places $v$ of $K$ (see \cite{HS00} p.281). Notice that $Sel^{\phi}(A/K)$ is finite and its size is theoretically computable. In contract, there is no known algorithm that could determine the size of $TS(A(K))$. In particular, if K is a function field and $A$ is the Jacobian variety of a smooth curve $K$, it is conjectured by Tate that the size of $TS(A(K))$ is finite, and we can show that in this conjecture is equivalent to the Birch-Swinnerton-Dyer conjecture.

Three mentioned groups associated to $A$ are related by the following important sequence:
$$0 \rightarrow A'(K)/\phi (A(K)) \rightarrow Sel^{\phi}(A/K) \rightarrow TS(A/K)[\phi] \rightarrow 0. \hspace{2cm}(1)$$
If $\phi$ is the multiplication-by-$m$ map from $A$ to $A$, from $(1)$, the size of $Sel^{\phi}(A/K)$ will help us to deduce an upper bound for the rank of $A(K)$; and we also imply that the Tate-Shafarevich has finite $m-$torsion points. Because of its importance, a great number of study has been conducted in order to understand Selmer groups. In the next part, we will review one of the most significant works related to Selmer groups.

Recently, the average size of Selmer groups of Jacobians of hyperelliptic curves over the rational field $\mathbb{Q}$ was studied extensively by Bhargava, Gross, Arul Shankar, Ananth Shankar, and Xiaoheng Wang in \cite{BG13}, \cite{Sha16}, and \cite{SW13}. In these paper, the common strategy is to connect our hyperelliptic curves (their Weiertrass equations) to some representations that arise from Vinberg theory of $\theta-$groups. From that they are able to build a bijection between elements in Selmer groups and integral orbits of the associated representation that satisfy some special conditions. Then the last step is to applying the geometric-of-number technique to count integral orbits. This kind of argument previously was used by Bhargava and A. Shankar to compute the average size of $n-$Selmer group of elliptic curves over $\mathbb{Q}$, for $2 \leq n \leq 5$. We summarize these results in the following theorem:
\begin{theorem}(see \cite{BS13a}, \cite{BS13b}, \cite{BS13c}, \cite{BS13d}) If elliptic curves over $\mathbb{Q}$ are ordered by height, then
\begin{itemize}
\item[i)] The average size of 2-Selmer groups of elliptic curves is 3.
\item[ii)] The average size of 3-Selmer groups of elliptic curves is 4.
\item[iii)] The average size of 4-Selmer groups of elliptic curves is 7.
\item[iv)] The average size of 5-Selmer groups of elliptic curves is 6.
\end{itemize}
\end{theorem} 
Now we consider the same problem over function fields. One of the first papers in this direction is \cite{dJ02}, where the author showed that the average size of 3-Selmer groups of elliptic curves over $\mathbb{F}_q(t)$, with char$(\mathbb{F}_q) >3$, is $4 + \epsilon(q)$, where $\lim_{q \rightarrow \infty} \epsilon(q) =0$. The method used there is partially similar to the one in \cite{BS13b}: translate the problem into counting ternary cubic forms over $\mathbb{F}_q(t)$ problem. The other significant result in the function fields setting is of Jack Thorne (see \cite{Tho16}). By using the geometric-of-number argument, he was able to count integral orbits of a coregular representation over function fields, and then deduce that the average size of 2-Selmer groups of elliptic curves with two marked points equals $12$ (with some mild restrictions on the family of elliptic curves and the characteristic of $\mathbb{F}_q$, see \cite{Tho16} for details). For the family of all elliptic curves over function fields, the problem was investigated in \cite{HLN14}. Remarkably, the method, that was used in \cite{HLN14}, is more geometric than other mentioned papers. More precisely, they introduced two moduli spaces $\mathcal{M}$ and $\mathcal{A}$ whose points represent elements in 2-Selmer groups and elliptic curves respectively. Hence, the estimating average size problem becomes counting points on fibers of the Hitchin map $\mathcal{M} \rightarrow \mathcal{A}$. 

The aim of this article is to compute the average size of 2-Selmer groups of Jacobians of hyperelliptic curves in the function fields setting. From now on, we assume that our base field $K$ is the function field $\mathbb{F}_q(C)$, where $C$ is a smooth curve of genus $g$. In chapter $1$, we consider hyperelliptic curves of genus $n$ with a marked rational Weierstrass point, and the main result is:
\begin{theorem}
Suppose that $q> 4^{2n+1}$, then 
\begin{equation*}
\limsup\limits_{d \rightarrow \infty}\dfrac{\mathlarger{\sum}\limits_{\substack{\text{Hyperelliptic curves H over $\mathbb{F}_q(C)$} \\ Height(H) \leq d }} \frac{|Sel_2(H)|}{|Aut(H)|}}{\mathlarger{\sum}\limits_{\substack{\text{Hyperelliptic curves H over $\mathbb{F}_q(C)$} \\ Height(H) \leq d }} \frac{1}{|Aut(H)|}} \leq 3 +f(q),
\end{equation*}
where $\lim_{q \rightarrow \infty} f(q) = 0$.
\end{theorem}
Moreover, if we only consider hyperelliptic curves with square-free discriminant (i.e. transversal hyperelliptic curves, see the definition \ref{transversal}), then the rational function $f(q)$ of $q$ is gone:
\begin{theorem}
\begin{equation*}
\lim\limits_{d \rightarrow \infty}\dfrac{\mathlarger{\sum}\limits_{\substack{\text{Transversal hyperelliptic curves H} \\ Height(H) \leq d }} \frac{|Sel_2(H)|}{|Aut(H)|}}{\mathlarger{\sum}\limits_{\substack{\text{Transversal hyperelliptic curves H} \\ Height(H) \leq d }} \frac{1}{|Aut(H)|}} = 3,
\end{equation*}
\end{theorem}
In chapter 2, we will estimate the average size of 2-Selmer groups of hyperelliptic curves with two marked points: a rational Weiertrass point and a rational non-Weiertrass point. The main theorem in this chapter is:

\begin{theorem}
Suppose that $q > 16^{\frac{m^2(2m+1)}{2m-1}}$ and $p=char(\mathbb{F}_q) > 3$. Then we have that
\begin{eqnarray*}
&\limsup\limits_{deg(\mathcal{L}) \rightarrow \infty} \frac{\mathlarger{\sum}\limits_{\substack{\text{Hyperelliptic curves H} \\ Height(H) \leq d}} \frac{|Sel_2(H)|}{|Aut(H)|}}{\mathlarger{\sum}\limits_{\substack{\text{Hyperelliptic curves H} \\ Height(H) \leq d}} \frac{1}{|Aut(H)|}} \\
&\leq 4. \zeta_C((2m+1)^2). \prod\limits_{v \in |C|}\big(1+c_{2m-1}|k(v)|^{-2}+\dots+c_1|k(v)|^{-2m}-2|k(v)|^{(2m+1)^2}\big) \\ & \hspace{10cm}+2 + f(q),
\end{eqnarray*}where $\lim_{q \rightarrow \infty} f(q) =0$, and $c_i$ are constants which are only depended on $m$ and $p$.
\end{theorem}
And if we restrict to the family of transversal hyperelliptic curves, then we will obtain:
\begin{theorem}
If char($\mathbb{F}_q$) is "good", then 
$$\lim_{d \rightarrow \infty} \frac{\sum \limits_{\substack{\text{$(\mathcal{L}, \underline{a})$ is transversal} \\ deg(\mathcal{L} \leq d}} |Sel_2(H_{\underline{a}})|}{\sum \limits_{\substack{\text{$(\mathcal{L}, \underline{a})$ is transversal} \\  deg(\mathcal{L} \leq d}} 1} = 6.$$
\end{theorem}

Discussion of the method: over the rational field, Bhargava and Gross in \cite{BG13} made use of the connection between hyperelliptic curves and a representation in Vinberg theory. Based on that, they was able to construct a bijection between elements in 2-Selmer group and soluble integral orbits. Hence, the problem became the counting integral orbits problem, and it was solved by using geometric-of-number argument. In our case, i.e. over function field, due to Paul Levi (see \cite{Lev08}), the Vinberg theory of $\theta-$groups is available in positive characteristics case. Thus, we are able to built a connection between hyperelliptic curves and some Vinberg representations. Our next step is a generality of the geometric counting technique used in \cite{HLN14}. More precisely, by using the canonical isomorphism between the stabilizer scheme (associated to the Vinberg representation) and the 2-torsion subgroup of Jacobian of the universal hyperelliptic curve (see \ref{can1} and \ref{can2}), we could translate the main problem into a counting regular sections problem. Finally, by using the canonical reduction of principal bundles, the counting regular sections problem is transferred into counting sections of semi-stable vector bundles. This kind of bundles allow us to use Riemann-Roch theorem to estimate the dimension of the set of global sections. Notice that the canonical reduction theory for principal bundles is the generalization of the  Harder-Narashimhan semi-stable filtration of vector bundles, and using that canonical reduction, we also can estimate the size of automorphism groups of principal bundles, which play an important role in our weighted average. 

Discussion of the results: our results are consistent with the previous results in \cite{BG13} and \cite{Sha16}, and if we restrict to the transversal families, the obtained average sizes are equal to corresponding results over the rational field. Moreover, from the computation parts, we can give some intuitive meanings of the results as follows. Let assume we are restricting to the transversal families. Then in both cases, the main contributors are the case of Borel subgroups and the whole groups. For the case of the whole group, we obtain the Tamagawa number that equals 2 in the first case and equals 4 in the second case. In the case of Borel subgroups, we obtain 1 in the first case and 2 in he second case which are the numbers of Kostant sections in corresponding representations we are about to use. One more notice about some restrictions of $q$ is that since we are going to use some results in Vinberg theory over positive characteristics which is developed in \cite{Lev08} and \cite{Lev09}, we require the characteristics of $\mathbb{F}_q$ is good (in both case it is just $p>3$). And also in general case (theorem 1 and 3), $q$ is assumed to be large enough due to the comparison between 2-Selmer group and the set of torsors over $C$ of 2-torsion subgroup (see \ref{compare1} and \ref{compare2}).

\begin{remark}
\begin{itemize}
\item[i)] In this article, we only consider odd hyperelliptic curves, i.e. hyperelliptic curves with a marked rational Weierstrass point. For the case of even hyperelliptic curves, the method which is used in this article still works very well, and we are going to complete that case in our subsequent work. And notice that in the rational field setting, that problem has already been investigated by A. Shankar and X. Wang in \cite{SW13}.
\item[ii)] In \cite{Bha13}, Bhargava consider the family of general hyperelliptic curves, i.e. without marked points. This case is totally different from other cases that we mentioned before (including our cases) because of the fact that the representation Bhargava used there is a coregular representation but it does not come from Vinber theory of $\theta-$groups. Additionally, this representation does not possess any Kostant sections which play the essential role in our work (for example to define the regular locus). Because of this difficulty, we need to have some new ideas, and we hope to overcome this problem in the future.
\end{itemize}
\end{remark}

\subsection*{Acknowledgements}
I would like to express my deepest gratitude to Gan Wee Teck, my supervisor, for his enthusiastic guidance throughout the research, and also numerous useful conversations. I would also like to extend my appreciation to Chin Chee Whye, Ngo Bao Chau, and Luo Caihua for helpful discussions. Finally, I would like to thank National University of Singapore for its support. This work was done when the author was supported by NUS Research Scholarship.
\subsection*{Notation} Through out the article, we denote $C$ to be a smooth, complete, geometrically connected curve over $k=\mathbb{F}_q$ such that $C(k) \neq \emptyset$. Let $K= k(C)$ be the function field of $C$. For each closed point $c \in |C|$, we denote $K_v$ to be the completion of $K$ at $v$ and $\mathcal{O}_v$ to be its ring of integers. We also denote the genus of $C$ by $g$, and hyperelliptic curves that we are going to study have genus  $n>1$. 
\section{Hyperelliptic curves with a marked Weierstrass point}
In this part, we will consider the family of hyperelliptic curves of genus $n$ with a marked Weierstrass point. This kind of curves is the generalization of elliptic curves (when $n=1$), and they can also be defined by canonical equations, which are called Weierstrass equations. In the following section, we are going to review Weierstrass equations of hyperelliptic curves, and from that the heights of hyperelliptic curves will be defined.
\subsection{Weierstrass equation and height of hyperelliptic curve}
Let $H$ be a hyperelliptic curve over the function field $K=\mathbb{F}_q(C)$ with a marked $K-$rational Weierstrass point $O$. Assume that the characteristics of $\mathbb{F}_q$ and $2.(2n+1)$ are coprime. Then $H$ has an affine Weierstrass equation of the following form:
$$y^2= x^{2n+1}+a_2x^{2n-1}+ \dots + a_{2n}x +a_{2n+1},$$where $a_i \in K$ for all $i$ such that the discriminant $\Delta(a_2,\cdots,a_{2n+1}) \neq 0.$ The Weierstrass point $O$ is corresponding to $x = \infty.$ Moreover, the tuple $(a_2,a_3,\dots,a_{2n+1})$ is unique up to the following identification:
$$(a_2,a_3,\dots,a_{2n+1}) \equiv (\lambda^4.a_2, \lambda^6a_3,\dots, \lambda^{4n+2}.a_{2n+1}) \hspace{2cm} \lambda \in K^{\times}.$$ 
Fix the data $(a_2,a_3,\dots,a_{2n+1})$, we define the minimal integral model of $H$ as follows: for each point $v \in |C|$, we can choose an integer $n_v$ which is the smallest integer satisfying that: the tuple $(\varpi_v^{4n_v}a_2, \varpi_v^{6n_v}a_3, \cdots, \varpi_v^{4nn_v}a_{2n}, \varpi_v^{(4n+2)n_v}a_{2n+1})$ has coordinates in $\mathcal{O}_{K_v}$. Given $(n_v)_{v \in |C|}$, we define the invertible sheaf $\mathcal{L}_H \subset K$ whose sections over a Zariski open $U \subset C$ are given by $$\mathcal{L}_H(U) = K \cap \big(\prod_{v \in U}\varpi_v^{-n_v}\mathcal{O}_{K_v} \big).$$
Then it is easy to see that $a_i \in H^0(C, \mathcal{L}_H^{\otimes 2i})$ for all $i$. Furthermore, the stratum $(\mathcal{L}_H, \underline{a})$ is minimal in the sense that there is no proper subsheaf $\mathcal{M}$ of $\mathcal{L}_H$ such that $a_i \in H^0(C, \mathcal{M}^{\otimes 2i})$ for all $i$. Conversely, given a minimal strata $(\mathcal{L}, \underline{a})$ satisfying that $\Delta(\underline{a}) \neq 0$, we consider a subscheme of $\mathbb{P}^2(\mathcal{L}^{2n+1} \oplus \mathcal{L}^{2} \oplus \mathcal{O}_C)$ that is defined by 
$$Z^{2n-1}Y^2 = X^{2n+1} +a_2Z^2X^{2n-1} +\dots + a_{2n}Z^{2n}X + a_{2n+1}Z^{2n+1}.$$ This is a flat family of curves $\mathcal{H} \rightarrow C$ with integral geometric fibers, and the generic fiber $\mathcal{H}_K$ is a hyperelliptic curve over $K(C)$ with a marked rational Weierstrass point. Furthermore, the associated minimal data of $H$ is exactly $(\mathcal{L}, \underline{a})$. Hence we have just shown the surjectivity of the following map $\phi_{\mathcal{L}}$ with a given line bundle $\mathcal{L}$ over $C$:
$$\phi_{\mathcal{L}}: \{ \text{minimal tuples} \hspace{0.2cm}(\mathcal{L}, \underline{a}) \} \rightarrow \{ \text{Hyperelliptic curves $(H,O)$ such that $\mathcal{L}_H \cong \mathcal{L}$} \}.$$
Moreover, the sizes of fibers of $\phi_{\mathcal{L}}$ can be calculated as follows
\begin{proposition}
Given a line bundle $\mathcal{L}$ over $C$, the map $\phi_{\mathcal{L}}$ defined as above is surjective, and the preimage of $(H,O)$ is of size $\frac{|\mathbb{F}_q^{\times}|}{|Aut(H,O)|}$, here $Aut(H,O)$ denotes the subset of all elements in $Aut(H)$ which preserve the marked point $O$.
\end{proposition}
\begin{proof}
Suppose that $(H,O)$ is a hyperelliptic curve with the associated minimal data $(\mathcal{L}, \underline{a}).$ Since we fix the line bundle $\mathcal{L}$, the tuple of sections $\underline{a}$ is well-defined upto the following identification: 
$$\underline{a} \equiv \lambda. \underline{a} = (\lambda^4a_2, \dots, \lambda^{4n}a_{2n}, \lambda^{4n+2}a_{2n+1}), \hspace{1cm} \lambda \in \mathbb{F}_q^{\times}.$$
In the other words, there is a transitive action of $\mathbb{F}_q$ on the fiber $\phi_{\mathcal{L}}^{-1}(H)$. Furthermore, the stabilizer of any element in $\phi_{\mathcal{L}}^{-1}(H)$ is exactly $Aut(H,O)$. Hence, the size of $\phi_{\mathcal{L}}^{-1}(H,O)$ is $\frac{|\mathbb{F}_q^{\times}|}{|Aut(H,O)|}$. 
\end{proof}
\begin{definition}{\textbf{(Height of hyperelliptic curve)}} The height of the hyperelliptic curve $(H,O)$ is defined to be the degree of the associated line bundle $\mathcal{L}_H$. 
\end{definition}
 
 For each geometric point $v \in C$, we denote $C_v$ the completion of $C\otimes_k \bar{k}$ at $v$, $Spec(K_v)$ the generic point of $C_v$. So given the minimal integral model $\mathcal{H} \rightarrow C$, we see that $\mathcal{H}_{K_v}$ is defined over $\mathcal{O}_{K_v}$. 
\begin{definition}
Let $\mathcal{H} \rightarrow C$ be the minimal integral model of $(H,O)$. The hyperelliptic curve $(H,O)$ over $K$ is called to be $\textbf{regular}$ if for any geometric point $v$ of $C$, the completion $\mathcal{W}_{K_v}$ is regular over $Spec(\mathcal{O}_{K_v})$.
\end{definition}
The above condition of being regular is equivalent to that the minimal integral model is the minimal regular model. Now we define the notation of transversality.
\begin{definition} \label{transversal}  The hyperelliptic curve $(H, O)$ is called to be $\textbf{transversal}$ if the discriminant of its minimal integral model $\Delta(\underline{a}) \in H^0(\mathcal{L}^{(2n+1)(2n+2)})$ is square-free. 
\end{definition}
We have an important property of a transversal hyperelliptic curve:
\begin{proposition}
If $\alpha: C \rightarrow [S/\mathbb{G}_m]$ is transversal, then the corresponding hyperelliptic curve $H_\alpha \rightarrow C$ is the minimal regular model of its generic fiber. Hence the relative Jacobian $E_\alpha = Pic^0_{H_\alpha/C}$ is the global N\'{e}ron model of its generic fiber.
\label{transversal-->neron model}
\end{proposition}
\begin{proof}
The minimality of $H_\alpha$ is an easy consequence of the transversality. To prove that $H_\alpha$ is regular, we can firstly assume that $C=Spec(D)$, where $D$ is a discrete valuation ring with $m_D=(\pi).$ Then locally $H_\alpha$ is defined as $Spec(D[x,y]/(y^2-f(x))$, with $ord_\pi(\Delta(f)) \leq 1$. It is enough to prove that $Spec(D[x,y]/(y^2-f(x))$ is regular. In fact, take any maximal ideal $m_x \subset D[x,y]$ containing $y^2-f(x)$, then $m_x=(\pi,m_x')$ satisfying that the image $\overline{(m_x')} \subset (D/(\pi))[x,y]$ is maximal. Hence the ideal $\overline{(m_x')}$ is generated by $2$ elements, thus $m_x$ is generated by $3$ elements. Furthermore, since $dim(D[x,y])= dim(D) +2=3$, we deduce that the local ring $D[x,y]_{m_x}$ is regular. As a result, $Spec(D[x,y]/(y^2-f(x))$ is regular at $m_x$ if and only if $y^2-f(x) \notin m_x^2$. If there exists $m_x$ satisfying that $y^2-f(x) \in m_x^2$, then $m_x$ will be of the form $(\pi,y,p(x)$ where $\overline{p(x)} \in (D/(\pi)[x]$ is irreducible. This implies that 
$$f(x)=\pi^2f_1(x) + \pi p(x)f_2(x) + p(x)^2f_3(x),$$where $f_i(x) \in D[x]$. We also can choose an extension $D \subset D'$ such that $p(x)$ has a root $\beta$ in $D'$, and $D'$ is a DVR where the uniformizer is the image of $\pi$ (take $D' = D[x]/p(x)$ for example). Then $ord_D(\Delta(f)) = ord_{D'}(\Delta(f)) \leq 1$. We will finish the proof by showing that $ord_{D'}(\Delta(f))$ is strictly bigger than $1$. By considering $f(x)$ as a polynomial with coefficients in $D'$, and note that $\Delta(f(x)) = \Delta(f(x-a))$ for any $a \in D'$, we can assume that $p(0)=0$ and hence $f(x)$ is of the following form over $D'$:
$$f(x)= \pi^2f_1(x) + \pi xf_2(x) + x^2f_3(x),$$where $f_i(x) \in D'[x]$. In the other words,
$$f(x) = x^{2n+1}+a_{2n}x^{2n}+ \cdots + a_2x^2+a_1x+a_0,$$where $a_i \in D',$ $\pi |a_1, \pi^2|a_0$. By using the relation between the discriminant and the resultant of $(f,f')$:
\begin{equation}
\Delta(f) = (-1)^n R(f,f') = det\begin{pmatrix}
1 & a_{2n} &\cdots & 0& 0& 0 \\
0 & 1 &\cdots & 0& 0& 0 \\
\vdots & \vdots & & \vdots& \vdots& \vdots \\
0 & 0 &\cdots & a_1& a_0& 0 \\
0 & 0 &\cdots & a_2& a_1& a_0 \\
2n+1 & 2na_{2n} &\cdots & 0& 0& 0 \\
0 & 2n+1 &\cdots & 0& 0& 0 \\
\vdots & \vdots & & \vdots& \vdots& \vdots \\
0 & 0 &\cdots & 2a_2& a_1& 0 \\
0 & 0 &\cdots & 3a_3& 2a_2& a_1
\end{pmatrix}.
\end{equation}
From the properties of $a_0$ and $a_1$, we imply that $ord_{D'}(\Delta(f)) \geq 2$. \\
The second statement is a consequence of the first statement and a result of Raynaud in \cite{Ray70}. Notice that our hyperelliptic curves always have a rational point, hence they satisfy the assumption in \cite{Ray70}. 
\end{proof}
Moreover, later on we also need the following important property of a transversal hyperelliptic curve.
\begin{proposition}If $\alpha: C \rightarrow [S/\mathbb{G}_m]$ is transversal, then $H_\alpha \rightarrow C$ is semi-stable. 
\label{transversal-->semistable}
\end{proposition} 
\begin{proof}
For each closed point $s \in C$, we need to show that the fiber $H_{\alpha,s}$ is a semi-stable curve over $k(s)-$the residue field of $\mathcal{O}_C$ at $s$. By definition, this is equivalent to that over the algebraically closure $\overline{k(s)}$, any singular points of the curve $H_{\alpha,\overline{k(s)}}$ are ordinary double points. Notice that $H_{\alpha,\overline{k(s)}}$ is the projective curve defined by the following affine equation:
$$y^2= \bar{f}(x)= x^{2n+1}+\overline{b_2}x^{2n-1}+ \cdots + \overline{b_{2n}}x+\overline{b_{2n+1}},$$where $\overline{b_i} \in \overline{k(s)}$. Because $\Delta(f)$ is square-free, any roots in $\overline{k(s}$ of $\bar{f}(x)$ are of order at most $2$. In fact, otherwise we can assume that $\bar{f}(x)$ divides by $x^3$, i.e. $\pi_s$ divides $b_{2n-1},b_{2n},$ and $b_{2n+1}$. By looking at the expansion of the resultant $R(f,f')$, we can see that $ord_s(\Delta(f))= ord_s(R(f,f')) \geq 2$, a contradiction. Thus we have just proved that any roots in $\overline{k(s}$ of $\bar{f}(x)$ are of order at most $2$. By Example $10.3.4$ in \cite{Liu06}, we obtain the proposition.
\end{proof}
\begin{remark}
\label{E_S}
For each element $\underline{a} \in S$, by setting $f_{\underline{a}}(x)= x^{2n+1} + a_2x^{2n-1}+ \cdots + a_{2n}x+a_{2n+1}$ we obtain a projective curve $H_{\underline{a}}$ in $\mathbb{P}^3$ with the affine equation: $y^2=f_T(x)$. By varying $\underline{a}$, we obtain $H_S$ a flat family of integral projective curves over $S$. By the representability of the relative Picard functor, we obtain the scheme $Pic_{H_S/S}$ locally of finite type over $S$, and also the relative Jacobian $E_S=Pic^0_{H_S/S}$ over $S$
\end{remark}
\subsection{2-Selmer group and the first cohomology}
\label{2-selmer group and the first coho}
Recall that $C$ is a smooth, projective and geometrically connected curve over finite field $k$, we denote $K=k(C)$ the function field of $C$, $K_v$ its completion at a closed point $v\in |C|$, and $\bar{K}$ an algebraic closure field containing $K$.
 
For each morphism $\alpha: C \rightarrow [S/\mathbb{G}_m]$, we associate a flat family of genus $n$ curve $(W_\alpha, s_\alpha)$ over $C$. If we denote $E_\alpha$ the relative Jacobian of that family, then it can be seen that $E_\alpha = \alpha^*E_S$ where $E_S$ is defined in the remark \ref{E_S}. Now let we recall the definition of 2-Selmer group of the Jacobian $E$ of a hyperelliptic curve over $K$. 

Let $\phi :E \rightarrow E$ be the multiplication-by-2 map, then we have an exact sequence of $G(\bar{K}/K)$-module 
\begin{equation*}
0 \rightarrow E[2] \rightarrow E \xrightarrow{.2} E \rightarrow 0
\end{equation*}
Taking Galois cohomology gives the long exact sequence
\begin{equation*}
0 \rightarrow E[2](K) \rightarrow E(K) \xrightarrow{.2} E(K) \rightarrow H^1(G(\bar{K}/K), E[2]) \rightarrow  H^1(G(\bar{K}/K), E)\rightarrow \cdots.
\end{equation*}
From this we get a short exact sequence 
\begin{equation*}
0 \rightarrow E(K)/2E(K) \rightarrow H^1(G(\bar{K}/K), E[2]) \rightarrow  H^1(G(\bar{K}/K), E)[2] \rightarrow 0
\end{equation*}

Similarly, we consider the local picture: if $K_v$ is a completion, then we will get the local sequence just like the one above. Now we can glue them together to obtain the following commutative diagram \footnotesize
\[
\begin{tikzcd}
0 \arrow{r} & E(K)/2E(K) \arrow{r} \arrow{d} & H^1(K, E[2]) \arrow{r} \arrow{d}{\beta} & H^1(K, E)[2] \arrow{r} \arrow{d} & 0 \\
0 \arrow{r} & \prod_v E(K_v)/2E(K_v) \arrow{r} & \prod_v H^1(K_v, E[2]) \arrow{r}{\alpha} & \prod_v H^1(K_v, E)[2] \arrow{r} & 0
 \end{tikzcd}
\]
\normalsize
where $H^1(K, E)$ is the shorthand of $H^1(G(\bar{K}/K), E)$, similarly for $H^1(K_v, E)$.
\begin{definition}
The 2-Selmer group of $E$ is defined as the kernel of the decomposition $\alpha \circ \beta$, where $\alpha$ and $\beta$ are in the above diagram
\begin{equation*}
Sel_2(E) = ker(H^1(K, E[2]) \rightarrow \prod_v H^1(K_v, E[2])
\end{equation*}
\end{definition}
Let $E_{\alpha,K}$ be the generic fiber of $E_\alpha$, then we can define the $2-$Selmer group of $E_{\alpha,K}$ and denote it by $Sel_2(E_\alpha)$. \\
The group of isomorphism classes of $E_\alpha[2]-$torsors over $C$ can be classified by the \'{e}tale cohomology group $H^1(C, E_\alpha[2])$. By restriction to the generic fiber of $C$, we obtain a homomorphism
\begin{equation}
H^1(C,E_\alpha[2]) \rightarrow H^1(K, E_\alpha[2]).
\end{equation} 
The proof of the following Proposition is identical with Proposition 4.3.2 in \cite{HLN14}.
\begin{proposition}
The homomorphism $(1.2)$ factors through the 2-Selmer group $Sel_2(E_\alpha).$ 
\end{proposition}
Now we can establish the connection between the 2-Selmer group and the group of $E_\alpha[2]-$torsor over $C$. Firstly, we consider the transversal case:
\begin{proposition} If $\alpha: C \rightarrow [S/\mathbb{G}_m]$ is transversal then 
\begin{eqnarray*}
|Sel_2(E_\alpha)| = |H^1(C, E_\alpha[2])|
\end{eqnarray*}
\label{selmer equal torsor in transversal case}
\end{proposition}
\begin{proof}
we are going to prove that the restriction map
$$\rho_\alpha: H^1(C,E_\alpha[2]) \rightarrow H^1(K, E_\alpha[2])$$is injective and the image is $Sel_2(E_\alpha)$. \\
Injectivity: We need to show that any $E_\alpha[2]-$torsors over $C$ is uniquely determined by its generic fiber. Thus it is enough to consider the question over the formal disc $C_v$, where $C_v$ is the completion of $C$ at a geometric point $v$. Let denote $k(v)$ the residue field at $v$, $Spec(K_v)$ the generic point of $C_v$, and $I_v\subset Gal(K_v^s/K_v)$ the inertia group. Since $\alpha$ is transversal, Proposition \ref{transversal-->neron model} and \ref{transversal-->semistable} tell us that $H_\alpha /C_v$ is semi-stable and it is the N\'{e}ron model of its generic fiber. This implies that the special fiber of the Jacobian $E_\alpha$ at $v$ is an extension of an abelian variety of dimension $a_v$ by a torus of dimension $t_v$, hence the rank of $\mathbb{T}_2(E_{\alpha,v})=t_v+2a_v$. On the other hand, by Proposition 2.2.5 in \cite{SGA7}, we know that $\mathbb{T}_2(E_\alpha(K_v^s))^{I_v}$ is nothing but the generic fiber of the finite part $\mathbb{T}_2(E_\alpha/\mathcal{O}_{K_v})^f$ of $\mathbb{T}_2(E_\alpha/\mathcal{O}_{K_v})$, and also the special fiber of the finite \'{e}tale group scheme $\mathbb{T}_2(E_\alpha/\mathcal{O}_{K_v})^f$ is isomorphic to $\mathbb{T}_2(E_{\alpha,v})$. Hence the rank of $\mathbb{T}_2(E_\alpha(K_v^s))^{I_v}$ is also $t_v+2a_v$. We conclude that 
$$E_\alpha[2](k(v))=E_\alpha[2](K_v^s)^{I_v}$$.
We also have that $H^1(\mathcal{O}_{K_v},E_\alpha[2])= H^1(k(v), E_\alpha[2](k(v)))$. By inflation-restriction, we conclude that the map
$H^1(\mathcal{O}_{K_v},E_\alpha[2]) \rightarrow H^1(K_v, E_\alpha[2])$ is injective. \\
Now we will show that the image of $\rho_\alpha$ is exactly the group $Sel_2(E_\alpha)$. The proof below is almost identical to the proof of Proposition 4.3.4 in \cite{HLN14}. Indeed, for each $T_K$ an $E_\alpha[2]-$torsor over $K$ whose isomorphism class lies in $Sel_2(E_\alpha)$, we need to show that $T_K$ can be extended as an $E_\alpha[2]-$torsor over $C$. This is in fact a local problem hence we reduce to the case of formal disc $C_v$. From the $Selmer$ condition, the class of $T_K$ in $H^1(K_v,E_\alpha[2])$ lies in the image of $E_\alpha(K_v)/2E_\alpha(K_v).$ Hence the torsion $T_{K_v}$ fits in a cartesian diagram:
\[
\begin{tikzcd}
T_{K_v} \arrow{r} \arrow{d} & E_{\alpha,K_v} \arrow{d}{.2} \\
Spec(K_v) \arrow{r}{x} & E_{\alpha,K_v}
 \end{tikzcd}
\]
Since $E_{\alpha,C_v}$ is the N\'{e}ron model of $E_{\alpha,K_v}$, we can extend $x$ to a $C_v-$point $\tilde{x}$ of $E_{\alpha,C_v}.$ Now our promising $T_{\alpha,C_v}$ should be the fiber product of $C_v$ and $E_{\alpha,C_v}$ over $E_{\alpha,C_v}$ in the following diagram:
 \[
\begin{tikzcd}
T_{\alpha,C_v} \arrow{r} \arrow{d} & E_{\alpha,C_v} \arrow{d}{.2} \\
C_v \arrow{r}{\tilde{x}} & E_{\alpha,C_v}
 \end{tikzcd}
\]
The proof of the surjectivity of $\rho_\alpha$ is completed.
\end{proof}

\begin{proposition}
Let $\alpha: C \rightarrow [S/\mathbb{G}_m]$ and suppose that the generic fiber of $H_\alpha$ is a hyperelliptic curve. Then 
\begin{eqnarray*}
|Sel_2(E_\alpha)| \leq |H^1(C, E_\alpha[2])|, & \text{when}\hspace{0.2cm} E_\alpha[2](K) = 0, \\
|Sel_2(E_\alpha)| \leq 2^{2n-1}|H^1(C, E_\alpha[2])|, & \text{otherwise.}
\end{eqnarray*}
\label{compare1}
\end{proposition}
\begin{proof}
Note that $E_\alpha[2]$ is defined by the roots of a polynomial $g(x)$ whose coefficients are belong to $H^0(C, \mathcal{L}^{\otimes 2i})$, hence by Integral Root Theorem, we see that if $E_\alpha[2](K) \neq 0$, i.e $g(x)$ has a solution in $K$, then $g(x)$ will have a solution in $H^0(C, \mathcal{L}^{\otimes 2})$. In the other words, $E_\alpha[2](K) \neq 0$ implies that $E_\alpha[2](C) \neq 0$. The converse is easy, hence we have just shown that $E_\alpha[2](K) \neq 0$ if and only if $E_\alpha[2](C) \neq 0$.
Let $\mathcal{E}$ is the N\'{e}ron model of the generic fiber of $E_\alpha$ over $C$. Given an $E_\alpha[2]-$torsor  $T_K$ over $K$ whose isomorphism class lies in $Sel_2(E_\alpha)$, by using similar argument as in the previous Proposition, we can prove that  $T_K$ can be lifted to a torsor of $\mathcal{E}[2]$ over $C$. \\
From the above observation we get the following inequality:
\begin{equation}
|Sel_2(E_\alpha)| \leq |H^1(C, \mathcal{E}[2])|.
\end{equation}
Now we consider the short exact sequence of group schemes over $C$:
\begin{equation*}
0 \longrightarrow E_\alpha[2] \longrightarrow \mathcal{E}[2] \longrightarrow Q \longrightarrow 0,
\end{equation*}
where $Q$ is a skyscraper sheaf, then we will obtain a long exact sequence:
\begin{equation*}
0 \rightarrow H^0(E_\alpha[2]) \rightarrow H^0(\mathcal{E}[2]) \rightarrow H^0(Q) \rightarrow H^1(E_\alpha[2]) \rightarrow H^1(\mathcal{E}[2]) \rightarrow H^1(Q) \rightarrow A \rightarrow 0,
\end{equation*}
where $A= H^2(E_\alpha[2]) \rightarrow H^2(\mathcal{E}[2])$. By using the fact that $|H^0(Q)|=|H^1(Q)|$ since $Q$ is a skyscraper sheaf, we have the following expressions:
\begin{eqnarray*}
 |H^0(E_\alpha[2])|.|H^1(\mathcal{E}[2])|. |A| &=& |H^0(\mathcal{E}[2])|.|H^1(E_\alpha[2])| \\
\Longrightarrow  |H^0(E_\alpha[2])|.|H^1(\mathcal{E}[2])| &\leq & |H^0(\mathcal{E}[2])|.|H^1(E_\alpha[2])|\\
\Longrightarrow \dfrac{|H^1(\mathcal{E}[2])|}{|H^1(E_\alpha[2])|} &\leq & \dfrac{|H^0(\mathcal{E}[2])|}{|H^0(E_\alpha[2])|}.
\end{eqnarray*}
If $E_\alpha[2](K) =0$, then $|H^0(C, E_\alpha[2])| = 1 = |E_\alpha[2](K)| = |\mathcal{E}[2](C)|$, hence combining with inequality (4) we obtain that $|Sel_2(E_\alpha)| \leq |H^1(C, E_\alpha[2])|.$\\
If $E_\alpha[2](K) \neq \{0\}$ then by "lifting" property of N\'{e}ron model we have that $|\mathcal{E}[2](C)|= |E_\alpha[2](K)| \leq 2^{2n} \leq 2^{2n-1}.|E_\alpha[2](C)|.$ Hence the second inequality is obtained.

\end{proof}

In the estimation of the average size of 2-Selmer groups, the above second inequality can be ignored if we assume that the characteristics of our base field is large enough. Precisely, we have the following lemma:
\begin{lemma}
If $q > 4^{n(2n+1)}$ then the contribution of the case $E[2](C) \neq 0$ to the average is zero. In the other words, we have the following limit: 
\begin{equation*}
\limsup_{deg(\mathcal{L}) \rightarrow \infty}\dfrac{\mathlarger{\sum}\limits_{\substack{\alpha \in [S/\mathbb{G}_m](C) \\ \mathcal{L}(H_\alpha) \cong \mathcal{L} \\ E_{\alpha[2](K) \neq \{ 0\}}}} |H^1(C,E_{\alpha}[2])|}{H^0(C, \mathcal{L}^4 \oplus \mathcal{L}^6 \oplus \cdots \oplus \mathcal{L}^{4n+2})} = 0
\end{equation*}
\label{large char contributes 0}
\end{lemma}
\begin{proof}
Let $H_\alpha$ be an universal hyperelliptic curve over $C$ defined by $(\mathcal{L}, a_i)_{2\leq i \leq 2n+1}$, then the smooth locus $C'$ of the map $H \rightarrow C$ is determined by the condition $\Delta(c_i) \neq 0$, notice that $\Delta \in H^0(C,\mathcal{L}^{4n(2n+1)})$. Denote $E_\alpha$ the corresponding Jacobian of $H_\alpha$, then by the smoothness of $H_\alpha$ over $C'$, any $K_v-$points of $E_\alpha$ can be extended as  $C_v'-$points (same notation as in the previous proposition). Using the similar argument as the above proof, we imply that any elements in the 2-Selmer group of $E_\alpha$ can be lifted to $E_\alpha[2]-$torsors over $C'$. Consequently, we get
\begin{equation*}
|Sel_2(E_\alpha)| \leq |H^1(C', E_\alpha[2])|.
\end{equation*}
When $E_\alpha[2](C) \neq 0$, there exists a section $c \in H^0(C, \mathcal{L}^{\otimes 2})$ such that the $(x,z)-$ equation defining $H_\alpha$ can be factorized as
\begin{equation*}
x^{2n+1}+a_2x^{2n-1}z^2+ \cdots + a_{2n+1}z^{2n+1} = (x-cz)(x^{2n}+cx^{2n-1}z+b_2x^{2n-2}z^2+\cdots + b_{2n}z^{2n}).
\end{equation*}
It means that $\{a_i\}_{2\leq i \leq 2n+1}$ can be expressed in terms of $c$ and $\{b_j \}_{2\leq J \leq 2n}$, where $b_j \in H^0(C, \mathcal{L}^{2j}.$ If $d= deg(\mathcal{L})$ is large enough, then by using Riemann-Roch theorem, the number of multiple sets $\{a_i\}$ will be bounded by $q^{2n(2n+1)d+2n(1-g)}$.\\
Now we consider the following interpretation for $E_\alpha[2]-$tosors: any $E_\alpha[2]-$tosors over $C'$ can be considered as tame \'{e}tale covers of $C'$ of degree $2^{2n}$. Hence there is a natural map:
\begin{equation*}
\phi : H^1(C',E_\alpha[2]) \rightarrow \{\text{tame \'{e}tale covers of $C'$ of degree $4^n$} \}.
\end{equation*}
The number of points where $H_\alpha$ is singular $|C-C'|$ is bounded by the degree of $\Delta(H_\alpha)$, so $|C-C'| \leq 4n(2n+1)d$. As a consequence, the number of topological generators of $\pi_1^{tame}(C')$ is less than $2g+4n(2n+1)d$. The size of $H^1(C',E_\alpha[2])$ can be estimated if we understand the fiber of $\phi$. Let $M$ is a degree $2^2n$ \'{e}tale cover of $C'$, then giving $M$ the structure of $E_\alpha[2]-$torsor is equivalent to giving an action map:
\begin{equation*}
\psi: E_\alpha[2] \times_{C'} M \longrightarrow M
\end{equation*}
which is compatible with the structure maps to $C'$ and satisfies that the following natural map 
\begin{eqnarray*}
E_\alpha[2] \times_{C'} M & \longrightarrow & M \times_{C'} M \\
(g,m) & \mapsto & (g.m,m)
\end{eqnarray*}
is isomorphic. \\ 
 Since everything is proper and flat over $C'$, the map $\psi$ is totally defined by the generic map $\psi_K: (E_\alpha[2]\times_{C'} M)_K \rightarrow M_K$. As $K-$vector spaces, $dim(M_K) = 2^{2n}$ and $dim(E_\alpha[2]\times_{C'} M)_K = 2^{4n}$, hence the number of maps giving $M$ the structure of a $E_\alpha[2]-$tosors is bounded by $2^{6n}$, so is the fiber of $\phi$. \\
Now we obtain that the average in the case $E_\alpha[2](C) \neq 0$ is bounded by:
\begin{equation*}
\dfrac{q^{2^{6n}}.4^{2gn+4n^2(2n+1)d}.q^{2n(2n+1)d+2n(1-g)}}{q^{2n(2n+3)d+2n(1-g)}} = \dfrac{m.4^{4n^2(2n+1)d}}{q^{4nd}},
\end{equation*}
where $m$ is a constant independent to $d$. This goes to zero as $d$ goes to infinity if $q^{4n} > 4^{4n^2(2n+1)}$, or equivalently $q > 4^{n(2n+1)}$. The lemma is completed.
\end{proof}
From the above observations, we see that basically the size of 2-Selmer group of the Jacobian $E_\alpha$ is bounded by $H^1(C, E_\alpha[2])$. The set of $E_\alpha[2]-$torsors has the following interpretation: 
\begin{itemize}
\item[1.] Recall that $S=Spec(K[a_2,\dots,a_{2n},a_{2n+1}]) \cong \mathbb{A}^{2n}.$ Then any tuple $(\mathcal{L}, \underline{a})$ can be seen as a $C-$point of the quotient stack $[S/\mathbb{G}_m]$, where the action of $\mathbb{G}_m$ on $S$ is given by $\lambda. (a_2,\dots,a_{2n},a_{2n+1}) = (\lambda^4a_2,\dots,\lambda^{4n}a_{2n},\lambda^{4n+2}a_{2n+1}).$ We set $\mathcal{A}= Hom(C, [S/\mathcal{G}_m])$, then $A(k)$ classifies isomorphism classes of tuples $(\mathcal{L}, \underline{a}).$
\item[2.] Since the universal Jacobian $E_S$ is a group scheme over $S$, there is a natural map of quotient stacks
$$[BE_S[2]/\mathbb{G}_m] \xrightarrow{\psi} [S/\mathbb{G}_m].$$
Given a morphism $\alpha: C \rightarrow [S/\mathbb{G}_m]$, as in the step 1, we obtain a family of curve $H_\alpha \rightarrow C$. Denote $E_\alpha= \alpha^*E_S$, then $E_\alpha$ is exactly the relative Jacobian of $H_\alpha$ over $C$. An isomorphism class of $E_\alpha[2]-$torsor over $C$ can be seen as a morphism $ \beta: C \rightarrow [BE_S[2]/\mathbb{G}_m]$ that fits in the following commutative diagram:
\[
\begin{tikzcd}
C \arrow{r}{\beta} \arrow{rd}{\alpha} & \text{$[BE_S[2]/\mathbb{G}_m]$} \arrow{d}{\psi} \\
 & \text{$[S/\mathbb{G}_m]$}
 \end{tikzcd}
\]
Hence if we set $\mathcal{M}= Hom(C, [BE_S[2]/\mathbb{G}_m])$ then we have a natural map $\mathcal{M} \rightarrow \mathcal{A}$ where the fiber $\mathcal{M}_\alpha$ over $\alpha \in \mathcal{A}(k)$ classifies isomorphism classes of $E_\alpha[2]-$torsors over $C$. 
\item[3.] Notice that the natural map $\mathcal{M} \rightarrow \mathcal{A}$ is compatible with maps to $Hom(C, B\mathbb{G}_m)$. 
\end{itemize}
Our main theorems is the corollary of the following:
\begin{theorem}
Suppose that $q > 4^{n(2n+1)}$. Then we have that
$$\limsup_{deg(\mathcal{L}) \rightarrow \infty} \frac{|\mathcal{M}_{\mathcal{L}}(k)|}{|\mathcal{A}_{\mathcal{L}}(k)|} \leq 3 + f(q),$$where $lim_{q \rightarrow \infty} f(q) = 0$.
\end{theorem}
Let $\mathcal{A}^{trans}(k)$ be the subset of transversal elements in $\mathcal{A}(k)$, and $\mathcal{M}^{trans}(k)$ be the preimage of $\mathcal{A}^{trans}(k)$ under the natural map $\mathcal{M} \rightarrow \mathcal{A}$. Then in transversal case, we have the following limit:
\begin{theorem}
$$\lim_{deg(\mathcal{L}) \rightarrow \infty} \frac{|\mathcal{M}^{trans}_{\mathcal{L}}(k)|}{|\mathcal{A}^{trans}_{\mathcal{L}}(k)|} =3.$$
\end{theorem}
The next section \ref{Vinberg representation 1} will provide another interpretation of $k-$rational points of the moduli space $\mathcal{M}$ which is easier to estimate the sizes. 
\subsection{Representation of the split odd special orthogonal group}
\label{Vinberg representation 1}
 Let $W$ be a non-degenerate, split orthogonal space over $k$, of dimension $2n+1 \geq 3$ and of discriminant $1$. Then we denote $G=SO(W).$ Such an orthogonal space over $k$ is unique up to isomorphism, and it has an ordered basis 
\begin{equation*}
\{e_1, e_2, \dots, e_n, u, f_n, \dots, f_2, f_1\}
\end{equation*}
with inner products given by
\begin{gather*}
\langle e_i, e_j \rangle = \langle f_i, f_j \rangle = \langle e_i, u \rangle = \langle f_i, u \rangle = 0, \\
 \langle e_i, f_j \rangle = \delta_{ij}, \\
 \langle u, u \rangle = 1.
\end{gather*}
Let $T \in End(W)$ be an endomorphism of the vector space $W$, we define the adjoint transformation $T^*$ uniquely by the formula
\begin{equation*}
\langle Tv, w \rangle = \langle v, T^*w \rangle.
\end{equation*}
It is easy to see that the matrix $M$ of $T^*$ with respect to the above basis is obtained from the matrix of $T$ by reflection about the anti-diagonal. We say that $T$ is self-adjoint if $T=T^*$. We define the group scheme $G = SO(W)$ over $k$ by
\begin{equation*}
G:= SO(W) = \{ g \in GL(W): g^*g=1,  det(g)=1 \}.
\end{equation*}
Since $char(k) \neq 2$, the group $G(k)$ gives the points of the split orthogonal group $SO_{2n+1}$ of the space $W$. 

In this paper we consider the following representation of $G$:
\begin{equation*}
V= \{T \in End(W) | T = T^*, Trace(T) = 0\}
\end{equation*}
with the action of $G$ on $V$ by conjugation $T \rightarrow gTg^{-1}=gTg^*$. The space $V$ has dimension $2n^2+3n$, and can be identified with the submodule $Sym^2_0(W)$ of $Sym^2(W)$ which is a complement to the line spanned by the defining quadratic form on $W$. We also note that the representation $V$ is an irreducible representation which arises in Vinberg theory (see \cite{Vin76}). For each $T \in V$, we write the characteristic polynomial $f(x)$ of $T$:
\begin{equation*}
f(x) = det(xI-T) = x^{2n+1} +c_2(T)x^{2n-1}+ \cdots + c_{2n}(T)x + c_{2n+1}(T)
\end{equation*}
with coefficients $c_m(T) \in k$. The $c_m$ gives algebraically independent polynomial invariants which generate the full ring of polynomial invariants on $V$ over $k$. We denote $S:= Spec(k[V]^G)= Spec(k[c_2,c_3,\dots, c_{2n+1}] \cong \mathbb{A}^{2n}$ and $\pi$ the natural projection map: 
\begin{eqnarray*}
\pi: V & \longrightarrow & S \\
T &\mapsto & (c_2,c_3, \dots, c_{2n+1})
\end{eqnarray*}
will be defined. 

The second representation of $G$ we also want to consider is the adjoint representation $\mathfrak{g} = \mathfrak{so}(W)$ given by the conjugation action of $G$ on its Lie algebra. Similarly, the coefficients of the characteristic polynomial form an algebraically independent basis of the algebra of polynomial invariants of $\mathfrak{g}$ (see \cite{Bou82} Ch 8, §8.3,§13.2, VI). Note that we have an isomorphism of irreducible representations of $G$
\begin{equation*}
\mathfrak{g} \cong \wedge^2(W)
\end{equation*} 
which is convenient for our later computation.	
	
\subsubsection{Regular locus and Kostant section}
\begin{definition} Let $k/\mathbb{F}_p$ be a field, an element $T \in W(k)$ is called to be regular if dim$(Stab_{G_k}(T)) =0$.  
\end{definition}
\begin{remark} \begin{itemize}
\item[i.]In \cite{SS70} the general notion of regular elements is given. In our case, the regularity of an element $T$ is equivalent to the condition that the characteristic polynomial and the minimal polynomial of $T$ are the same.
\item[ii.]In Vinberg theory over $\mathbb{F}_q$, we know that the regularity is a codimension 2 condition, i.e dim$(W^{nonreg}) \leq $dim$(W)-2$. (see Lemma 6.31 in \cite{Lev08})
\end{itemize}
\end{remark}
\begin{example} Let $N$ be the nilpotent matrix:
$$\begin{pmatrix}
0 & & &&\\
1 & 0 & &&\\
& 1 & 0 &&\\
& & \ddots & \ddots  &\\
& & & 1 & 0 \\
& & & & 1 & 0
\end{pmatrix}$$
$A$ be any upper triangular matrix of the same size $l$ as $N$. The the sum $T= N + A$ is regular in the sense that the characteristic polynomial and the minimal polynomial of $T$ are the same. Indeed, it is easily to see by decomposition that the set of vectors over $k$: $\{Id, T, \dots, T^{l-1}\}$ is independent. Hence $T$ is regular.
\end{example}
In the rest of this subsection, we will present a (Kostant) section of the quotient map $W \rightarrow W//G \cong S$ precisely. For each element $(a_2,a_3, \dots, a_{2n+1})$ in $S$, we need to find a canonical traceless matrix whose characteristic polynomial is $f(x) = x^{2n+1} + a_2x^{2n-1} + \dots + a_{2n+1}$. And that is
\[ T_f=\left( \begin{array}{cc}
A & B_f \\
C & D
 \end{array} \right)\]
where $A$ is a $(n+1)\times n$ matrix whose first row is zero and the others form an identity matrix, $C$ is the zero $n \times n$ matrix, $D$ is the $n \times (n+1)$ matrix whose last column is zero and the others form an identity matrix. The last matrix $B_f$ is the following tri-anti-diagonal matrix:
 
\[ \left( \begin{array}{ccccccc}
0 & 0 & 0 & \cdots & 0 & -\dfrac{1}{2}a_{2n} & -a_{2n+1} \\
0 & 0 & 0 & \cdots & -\dfrac{1}{2}a_{2n-2} & -a_{2n-1} & -\dfrac{1}{2}a_{2n} \\
0 & 0 & 0 & \cdots & -a_{2n-3} & -\dfrac{1}{2}a_{2n-2} & 0 \\
\cdots & \cdots & \cdots & \cdots & \cdots & \cdots & \cdots \\
0 & -\dfrac{1}{2}a_4 & -a_5 & \cdots & 0 & 0 & 0 \\
-\dfrac{1}{2}a_2 & -a_3 & -\dfrac{1}{2}a_4 & \cdots & 0 & 0 & 0 \\
0 & -\dfrac{1}{2}a_2 & 0 & \cdots & 0 & 0 & 0 \\
 \end{array} \right)\]
 After that we could define a map $S \times \mathbb{G}_m \rightarrow V \times G \times \mathbb{G}_m$ as follow:
 \begin{equation*}
 (f , c) \mapsto (T_f, diag(c^{2n},c^{2n-2}, \dots, c^2, 1, c^{-2}, \dots, c^{-2n}), c)
 \end{equation*}
 It is easy to check that the above map is compatible with the actions of $\mathbb{G}_m$ on $S$ and $G \times  \mathbb{G}_m$ on $V$. Hence it induces a section on quotient stack that we also call the Kostant section:
 \begin{equation*}
 \kappa : [S/\mathbb{G}_m] \rightarrow [V/G\times \mathbb{G}_m]
 \end{equation*}
\begin{remark} Over algebraically closed field $\bar{k}$, any $G(\bar{k})-$orbits in $V^{reg}$ intersect $\kappa(S)$ at exactly one point. To see this, firstly, let recall some notations in Vinberg theory: denote $H=SL(W)$ over $k$, $\theta$ is an involution of $H$ given by $\theta(g) = (g*)^{-1}$. Our considering group $G=SO(W)$ is the fixed subgroup $H^{\theta}$, and the representation $V$ is the non-trivial eigenvector space for $\theta$ on the Lie algebra $\mathfrak{h}$:
$$\mathfrak{h}(1)= \{ T \in End(W): T^* =T,\, \text{Trace}(T)=0\}.$$Now if we set $G^*:= (H_{ad})^{\theta}$, then $G^*(\bar{k}) = \{g \in H(\bar{k}) | \, \theta(g)g^{-1} \in Z(\bar{k})\} /Z(\bar{k})$, where $Z$ is the center of $H$. We could easily see that $G^*(\bar{k})$ can be identified with $G(\bar{k})$.  The result (Lemma 0.20 in \cite{Lev08}) tell us the stated property of the Kostant section. 

In the next section, we are going to see the connection between 2-torsion subgroup of Jacobian of a hyperelliptic curve and the stabilizer group of the representation $(V,G)$ in section 1.1. 
\end{remark}

\subsection{Stabilizer group and 2-torsion subgroup} 
Firstly, we consider the stabilizer group scheme of the action of $G$ on $V=Sym^2_0(W)$. Recall that $V^{reg}$ is an open subset of $V$ consisting of all regular elements. We write $I \rightarrow V^{reg}$ for the stabilizer scheme, defined as the equalizer of the following diagram:
\[
\begin{tikzcd}
G\times V^{reg} \arrow[shift left=4]{rr}{(g,v)\mapsto g.v}
  \arrow[shift right=3,swap]{rr}{(g,v)\mapsto v}&& V^{reg}
 \end{tikzcd}
\]
It can be checked on geometric fibers that $I$ is a commutative and quasi-finite group scheme over $V^{reg}.$ Additionally, because $G(\bar{k})$ acts transitively on $V_f^{reg}(\bar{k})$ (Proposition 1.1 in \cite{Wan1}), the map $G \times \kappa \longrightarrow V^{reg}$, where $\kappa$ is the Kostant section, is \'{e}tale and surjective. With that covering, since $I$ is abelian, we now can descend $I$ to an abelian group scheme over $S$, and we denote it by $I_S$. In fact, we need to check that $(I,V^{reg} \rightarrow S)$ is a descent data. For each pair $(a,b) \in V^{reg}\times_S V^{reg}$, i.e. $\pi(a) = \pi(b) =s \in S$, there exists $g \in G(\bar{k})$ such that $a=g.b$. This induces an isomorphism of stabilizers $I_a \cong I_b$ over $\bar{k}$ given by sending $h \in I_a$ to $g^{-1}.h.g$. Since $I$ is commutative, these isomorphisms are not depended on the choice of $g$, and hence they are canonical and defined over $k$. The cocycle condition will follow from the canonical property of the above constructed isomorphisms. 
\begin{proposition} We have the following isomorphism of quotient stacks over $S$:
$$[V^{reg}/G] \cong [BI_S],$$where $[BI_S]$ is the classifying stack of $I_S-$torsors over $S$.
\end{proposition}
\begin{proof}
Firstly, we need to see how we can see the quotient stack $[V^{reg}/G]$ which is defined over $k$, as a stack over $S$. By definition, given a scheme $T$ over $k$, an element $Q$ in $[V^{reg}/G](T)$ is a $G-$torsor over $T$ that fit in the following commutative diagram:
\[
\begin{tikzcd}
Q \arrow{r}{\alpha} \arrow{d} & V^{reg} \arrow{d} \\
T \arrow{r} & Spec(k)
 \end{tikzcd}
\]where the map $\alpha$ is $G-$equivariant. Since $Q$ is a $G-$torsor over $T$, there exists a covering $\{T_i \rightarrow T \}$ of $T$ such that the restriction $Q_{T_i} \rightarrow T_i$ is the trivial torsor : $T_i \times G \rightarrow T_i$. Since $\alpha$ is $G-$equivariant, over $T_i$, it induces a natural map $T_i = Q/G \rightarrow V^{reg}/G =S$. Hence we obtain a natural map $T \rightarrow S$, and it fits in the following commutative diagram
 \[
\begin{tikzcd}
Q \arrow{r}{\alpha} \arrow{d} & V^{reg} \arrow{d} \\
T \arrow{r} & S
 \end{tikzcd}
\]This means that we can consider the quotient stack $[V^{reg}/G]$ as a stack over $S$. We define an equivalent relation on the product $G \times_k S:$ $(g_1,s_1) \equiv (g_2,s_2)$ if $s_1=s_2$ and $g_1=g_2.e$ for some $e \in I_{s_1}$. Since the action map $G \times_k S \rightarrow V^{reg}$ is surjective, the quotient $(G \times_k S)/\equiv$ which is denoted by $G \times^{I_S}S$, is isomorphic to $V^{reg}$. Now we will construct two morphisms from $[V^{reg}/G]$ to $[BI_S]$ and vice versa such that they are inverses to each other.

The first map $\psi: [V^{reg}/G] \rightarrow [BI_S]$: Given a scheme $T$ over $S$, an element $Q \in [V^{reg}/G](T)$, and a commutative diagram as above. we take $P$ to be the fiber product over $V^{reg}$:
\[
\begin{tikzcd}
Q \arrow{r}{\alpha}  & V^{reg} \\
P \arrow{r} \arrow{u} & S \arrow{u}{\kappa}
 \end{tikzcd}
\]
Now if we denote $\beta: P \rightarrow T$ be the composition of the following maps: $P \rightarrow Q \rightarrow T$, we obtain a commutative diagram:
 \[
\begin{tikzcd}
P \arrow{r}{\gamma} \arrow{d}{\beta} & S \arrow{d}{id} \\
T \arrow{r} & S 
 \end{tikzcd}
\]Now we will show that the above diagram will define an element in $[BI_S](T).$ It is easy to check that the map $P \xrightarrow{\gamma} S$ is $I_S-$equivariant ($I_S$ acts trivially on $S$), hence we only need to check that $P$ is a $I_S-$torsor over $T$. Over a point $t \in T$ over $s \in S$, the fiber $P_t$ is the set of elements in $Q_t$ that is mapped to $\kappa(s) \in V^{reg}$ via the map $\alpha$. Set-theoretically, $Q_t$ is $G$ without the group structure. Since $\alpha$ is $G-$equivariant, we deduce that $Q_t \cap \alpha^{-1}(\kappa(s))$ is a $I_s-$torsor. Thus, $P$ is a $I_S$ torsor over $T$.

The second map $\phi : [BI_S] \rightarrow [V^{reg/G}]:$ Given a scheme $T$ over $S$, any element $P\in [BI_S]$ fits the following commutative diagram:
  \[
\begin{tikzcd}
P \arrow{r}{\gamma} \arrow{d}{\beta} & S \arrow{d}{id} \\
T \arrow{r} & S 
 \end{tikzcd}
\]
By considering the action of $I_S$ on $G$, we have already defined the quotient $S \times^{I_S} G \cong V^{reg}$. Similarly, we define $P \times^{I_S} G$ as the quotient of $P \times_k G$ by the following equivalent relation: $(x,g) \sim (x,g.e)$ for any $x \in P, g \in G,$ and $e \in  I_{\gamma(x)}$. The action of $G$ on $P \times^{I_S} G$ is given by: $g . \overline{(x,g')} = \overline{(x,gg')}.$ Then the map that naturally induced from $\gamma$:
\begin{align*}
\gamma \times Id:  P \times^{I_S} G & \rightarrow S \times^{I_S} G = V^{reg} \\
  \overline{(x,g)} & \mapsto \overline{(\gamma(x),g)}
\end{align*}is well defined and $G-$equivariant. To sum up, we have just constructed the following commutative diagram:
  \[
\begin{tikzcd}
P \times^{I_S} G \arrow{r} \arrow{d}{pr_1} & V^{reg} \arrow{d}{\pi} \\
P \arrow{r}{\gamma} \arrow{d}{\beta} & S \arrow{d}{id} \\
T \arrow{r} & S 
 \end{tikzcd}
\]It is easy to see that the fiber of $P \times^{I_S} G \xrightarrow{\beta \circ pr_1} T$ over $t \in T$ that is mapped to $s \in S$  is $P_t \times G/I_s \cong I_s \times G/I_s$. Thus, $P \times^{I_S} G$ is a $G-$torsor over $T$. We have just constructed an element in $[V^{reg}/G](T)$.

The proof is completed by noting that two map we have just constructed are inverses to each other.
\end{proof}
Recall that for each element $T \in \kappa$, $f_T(x)$ the characteristic polynomial of $T$, we consider the projective curve in $\mathbb{P}^3$ with the affine equation: $y^2=f_T(x)$. As a result, we obtain $H_S$ a flat family of integral projective curves over $S$. By the representability of the relative Picard functor, we obtain the scheme $Pic_{H_S/S}$ locally of finite type over $S$, and also the relative Jacobian $E_S=Pic^0_{H_S/S}$ over $S$. The relation between two group schemes the stabilizer and the Jacobian is stated in the following proposition:
\begin{proposition} We have a canonical isomorphism of two \'etale group schemes over $S$:
\begin{equation*}
I_S \cong E_S[2]
\end{equation*}
\label{can1}
\end{proposition}
\begin{proof}(cf. Proposition 4.2.1 in \cite{HLN14} and Remark 1.6 in \cite{Wan1}) By setting $E_{V^{reg}}:= E_S \times_S V^{reg}$, it is enough to show that there is a canonical isomorphism 
$$I \cong E_{V^{reg}}[2]$$of group schemes over $V^{reg}$. Let $b$ be the bilinear form corresponding to the quadratic form $Q$ of the orthogonal space $W$ (see section 1.2). For each $T \in V^{reg}$, we consider following the intersection:
$$B_T= \{v \in W \,|\, b(v,v)=b(v,T(v))=0 \}.$$ Notice that $B_T$ is smooth if and only if $T$ is regular semi-simple, and in general, the singular locus of $B_T$ is the set of eigenvectors of $T$ whose eigenvalues are of order at least $2$ in the characteristic polynomial of $T$. Let $L_T$ be the variety of projective $n-1$ planes (where dim$(W)=2n+1$) contained in the smooth part of $B_T$ (this variety is denoted by $L^{f_T,T}_{\{0,0,\dots,0\}}$ in \cite{Wan1}). Vary $T$ we will obtain a scheme $L$ over $V^{reg}$. 

There are simply-transitive actions of $E_T[2]$ and $I_T$ on $L_T$ and they commute (see \cite{Wan2}). Hence, if we fix an element $X_0 \in L_T$, we can define a map
\begin{align*}
\psi: E_T[2] \rightarrow I_T
\end{align*} 
by taking $\psi([D])$, for any $[D] \in E_T[2]$, to be the unique element in $I_T$ sending $X_0$ to $X_0+[D]$. Commutativity of the two actions and commutativity of J[2] show that this map is independent on the choice of $X_0$. The isomorphism $\psi$ is $G-$equivariant, and it descends to an isomorphism $\psi_S: E_S[2] \rightarrow I_S$ over $S$.

 On geometric fibers, the above isomorphism can be seen as follows: given an element $T \in S(\bar{k})$, we have the factorization of $f_T$ in $\bar{k}$: $f_T(x) = \prod_{i=1}^{r+1}(x-\alpha_i)^{m_i}$. Denote $P_i \in E_T[2](\bar{k})$ the torsion points corresponding to $\alpha_i$, then $E_T[2](\bar{k})$ will be the abelian group of order $2^r$ generated by $P_i-\infty$ with a unique relation: $\sum_{i=1}^{r+1}m_i(P_i-\infty) =0$. On the other hand, we have the following isomorphism (see Proposition 1.1 in \cite{Wan1}):
\begin{equation*}
I_S(T) \cong \mu_2(\bar{k}[x]/f_T(x))^{\times}/(\pm 1)
\end{equation*}
In particular, $I_S(T)$ is an abelian 2-group of order $2^r$. The identification of $I_S(T)$ and $E_S[2](T)$ can be seen as follows: without loss of generality we can assume that $m_1,m_2,\dots, m_t$ are even and the rest is odd. For each $1\leq i \leq t$, since $gcd((x-\alpha_i)^{m_i}, f(x)/(x-\alpha_i)^{m_i})=1$ there exist two polynomials $h_i(x), g_i(x)$ such that $h_i(x).(x-\alpha_i)^{m_i}-g_i(x).f_T(x)/(x-\alpha_i)^{m_i}=2$. Now we set $p_i(x) =-1+(x-\alpha_i)^{m_i}.h_i(x)$ for any $1\leq i \leq t$. Similarly, in case $t<j\leq r+1$ we also obtain the polynomials $p_j(x)=-1+h_j(x)(x-\alpha_j)^{m_j}.\prod_{i=1}^{t}(x-\alpha_i)^{m_i}$ by considering $gcd((x-\alpha_j)^{m_j}.\prod_{i=1}^{t}(x-\alpha_i)^{m_i}, \prod_{l=t+1,l\neq j}^{r+1}(x-\alpha_l)^{m_l})=1$. Because the $r-t+1$ is odd, the number of odd numbers in $\{m_i\}_{i=\bar{1,t}}$ is odd, hence 
\begin{equation*}
\prod_{i=1}^{r+1}p_i(x)^m \equiv -1 (\text{mod} f_T(x)).
\end{equation*}
Additionally,
\begin{eqnarray*}
\mu_2(\bar{k}[x]/f_T(x))^{\times} &=& \bigg\{\prod_{i\in I}p_i(T) \bigg\}_{I\subset \{1,2,\dots, r+1\}},
\end{eqnarray*}
so the map sends $P_i-\infty$ to $p_i(T)$ is an isomorphism from $E_T[2]$ to $I_T$. 
\end{proof}
From the above construction, it is easy to see that the isomorphism $I_S \cong E_S[2]$ is $\mathbb{G}_m-$equivariant. Hence we have an isomorphism of quotient stacks $$[V^{reg}/G\times \mathbb{G}_m]\cong [BI_S/\mathbb{G}_m] \cong [BE_S[2]/\mathbb{G}_m].$$
The above observation give another interpretation of $\mathcal{M}(k)$ (see : from the isomorphisms $I_S \cong E_S[2]$ and $BI_S \cong [V^{reg}/G]$, we deduce that 
$$\mathcal{M} \cong Hom(C, [V^{reg}/G\times \mathbb{G}_m].$$
Consequently, $\mathcal{M}_{\mathcal{L}}(k)$ classifies tuples $(\mathcal{E}, s)$ where $\mathcal{E}$ is a principal $G-$bundle and $s$ is a global section of the vector bundle $(V^{reg}\times^G \mathcal{E}) \otimes \mathcal{L}$. In the next sections, we will try to estimate the size of $H^0(C,(V^{reg}\times^G \mathcal{E}) \otimes \mathcal{L})$ for a given $G-$bundle $\mathcal{E}$.
\subsection{Counting regular sections}
Traditionally, in order to estimate the number of global sections, we use Riemann-Roch theorem. In our case, we need to calculate the size of $\mathcal{M}_\mathcal{L}(k)$ which is the number of global sections of the associated vector bundle whose images are in the regular locus. Hence, before applying the Rimann-Roch theorem, we need to compute the density of regular sections.
\subsection{Density of sections in $V(\mathcal{E})$ that are in $V(\mathcal{E}^{reg})$}
Firstly, we can move to local problems by using the following result (Proposition 5.1.1 in \cite{HLN14}):
\begin{proposition} Let $C$ be a smooth projective curve over $\mathbb{F}_q$, $E$ a vector bundle over $C$ of rank $r$. Let $X \subset E$ be a locally closed $\mathbb{G}_m-$stable subscheme of codimension at least 2 whose fiber at every point $v \in C$, $X_v \subset E_v$ is also of codimension at least 2. Then the ratio
\begin{equation*}
\mu(X, \mathcal{L}) = \frac{|\{ s \in \Gamma(C, E \otimes \mathcal{L}): s \hspace{0.1cm}\text{avoids} \hspace{0.1cm} X \otimes \mathcal{L} \}|}{|\Gamma(C, E \otimes \mathcal{L})|}
\end{equation*}
as $deg(\mathcal{L}) \rightarrow \infty,$ tends to the limit
\begin{equation*}
\mu(X) := \lim_{deg(\mathcal{L} \rightarrow \infty} \mu(X, \mathcal{L}) = \prod_{v \in |C|} (1- \frac{c_v}{|k(v)|^r}),
\end{equation*}
where $c_v = |X_v(k(v))|,$ with $k(v)$ denoting the residue field at $v$.
\end{proposition}
\begin{proof}
See section 5.1 of \cite{HLN14}.
\end{proof}
Since "regularity" is codimension 2 condition, the above proposition is applied very well in our case. At any point $v \in C$, we need to compute the size of $V^{reg}(k_v)$. We can show that for any fixed polynomial $f(x)$, the number of regular orbits whose characteristic polynomial is $f(x)$ is equal to the order of the stabilizer $Stab(T)(k_v)$ of any point $T$ in the orbit. Hence the number of regular matrices $T$ with any fixed characteristic polynomial is equal to the order of the finite group $SO(W)(k_v)$ (see section 6.1 of \cite{BG13} for details). Moreover, since there is $q^{2n}$ choices for the coefficients of the characteristic polynomial $f(x)$, where $q:= |k_v|$, the size of $V^{reg}(k_v)$ is
\begin{equation}
q^{2n}. |SO(W)(q)| = q^{2n^2+3n}(1-q^{-2})(1-q^{-4}) \cdots (1-q^{-2n}).
\end{equation}
Combine with the fact that $dim(V) = 2n^2 + 3n$, we obtain the proof of the following proposition:
\begin{proposition}
For an arbitrary $G-$torsor $E$, 
\begin{equation*}
\lim_{deg(\mathcal(L)) \rightarrow \infty} \frac{|H^0(V^{reg}(E, \mathcal{L}))|}{|H^0(V(E, \mathcal{L}))|} = \zeta_C(2)^{-1}.\zeta_C(4)^{-1}\dots \zeta_C(2n)^{-1},
\end{equation*}
where $V(E, \mathcal{L}) = (E \times^G V) \otimes \mathcal{L}^{\otimes 2}$.
\label{general locus 1}
\end{proposition}

\begin{remark}
From the above argument we can see that if we consider some local property $P$ of polynomials over $K(C)$ (regularity, for example), and we also want to determine the density of regular vectors whose characteristic polynomials satisfy $P$, then the result should be $a . \zeta_C(2)^{-1}.\zeta_C(4)^{-1}\dots \zeta_C(2n)^{-1},$ where $a$ is the density of polynomials that satisfy $P$.
\label{remark on transversal poly}
\end{remark}
The transversality is not local, but we have the following result (see Proposition 5.1.6 in \cite{HLN14}): If $v$ is a place of $K$, define 
$$\alpha_v= \frac{|\{x \in S(\mathcal{O}_{K_v}/(\varpi_v^2))| \Delta(x) \equiv 0 \hspace{0.2cm}mod (\varpi_v^2) \}|}{|k(v)^{4n}|}$$and
$$\beta_v= \frac{|\{x \in V^{reg}(\mathcal{O}_{K_v}/(\varpi_v^2))| \Delta(x) \equiv 0 \hspace{0.2cm}mod (\varpi_v^2) \}|}{|k(v)^{4n^2+6n}|},$$then we can take the limit of the density of regular section in the transversal cases as follows:
\begin{proposition}We have the following equalities
\begin{itemize}
\item[1.] $$\lim_{deg(\mathcal{L})\rightarrow \infty} \frac{|\Gamma(C,\mathcal{L}^{\otimes 4}\oplus \mathcal{L}^{\otimes 6} \oplus \cdots \oplus \mathcal{L}^{\otimes 4n+2})^{sf}|}{|\Gamma(C,\mathcal{L}^{\otimes 4}\oplus \mathcal{L}^{\otimes 6} \oplus \cdots \oplus \mathcal{L}^{\otimes 4n+2})|} = \prod_{v \in |C|}(1-\alpha_v).$$
\item[2.] $$\lim_{deg(\mathcal{L})\rightarrow \infty} \frac{|\Gamma(C,V^{reg}(\mathcal{E},\mathcal{L}))^{sf}|}{|\Gamma(C,V^{reg}(\mathcal{E},\mathcal{L}))|}= \prod_{v \in |C|}(1-\beta_v)$$
\item[3.]$$\frac{\prod_{v \in |C|}(1-\beta_v)}{\prod_{v \in |C|}(1-\alpha_v)} = \zeta_C(2)^{-1}.\zeta_C(4)^{-1} \dots \zeta_C(2n)^{-1}$$
\end{itemize}
Here the upper script "sf" stands for "square free", i.e. $\Gamma()^{sf}$ is the set of sections whose invariants are transversal to the discriminant locus.
\label{transversal locus 1}
\end{proposition}
\begin{proof}
The first and the second statements are known results in \cite{HLN14} where they used the technique of Poonen \cite{Poo03}. For the third equality, it is enough to prove that for any point $v \in |C|$:
$$\frac{1-\beta_v}{1-\alpha_v} = (1-q^{-2})(1-q^{-4}\cdots(1-q^{-2n}),$$where $q=|k(v)|$. Let denote $R = k(v)[\epsilon]/(\epsilon^2)$, then by definition, $$1-\alpha_v= 1-\frac{\sum\limits_{a \in S^{non-transversal}(R)}1}{4n}.$$ 
Base on that, we can compute $1-\beta_v$ as follows: 
\begin{align*}
\beta_v &= \frac{\sum\limits_{T \in V^{nonreg}(R)}1 + \sum\limits_{T \in V^{reg}(R) ; \pi(T) \notin S^{sf}(R)}1}{q^{4n^2+6n}} \\
& =\frac{\sum\limits_{T \in V^{nonreg}(R)}1 + \sum\limits_{T \in V(R) ; \overline{T} \in V^{reg}(k(v)); \pi(T) \notin S^{sf}(R)}1}{q^{4n^2+6n}}.
\end{align*}
Given a non-transversal multi-set $a=(a_2,\dots,a_{2n+1}) \in S(R)$, we will find the number of $T \in V^{reg}(R)$ such that $\pi(T)=a$. Set $T=\overline{T} + \epsilon H$ and $a= \overline{a} + \epsilon b,$ where $\overline{T},H \in V(k(v))$ and $\overline{a},b=(b_2+\dots,b_{2n+1}) \in S(k(v))$, we firstly observe that there are $|G(k(v))|$ choices of $\overline{T}$ such that $\pi(\overline{T})= \overline{a}$. With a fixed $\overline{T}$, by considering $H$ and $b$ as elements in the tangent spaces of $V^{reg}$ and $S$, respectively, we can see that the tangent map:
$$d\pi: T_{\overline{T}}V^{reg} \rightarrow T_{\overline{a}}S $$will maps $H$ to $b$. Since $\pi: V^{reg} \rightarrow S$ is smooth, the number of choices of $H$ will be the size of the fiber of $d\pi$ at $b$, and it is equal to $q^{dim_{k(v)}(T_{\overline{T}}V^{reg}) - dim_{k(v)}(T_{\overline{a}}S)}=q^{2n^2+n}$. We obtain the following formula of $\beta_v$:
\begin{align*} 
\beta_v & =\frac{\sum\limits_{T \in V^{nonreg}(k(v))}q^{2n^2+3n} + \sum\limits_{a \in S^{non-transversal}(R)}|G(k(v)|.q^{2n^2+n}}{q^{4n^2+6n}} \\
&=\frac{\sum\limits_{T \in V^{nonreg}(k(v))}q^{2n^2+3n} + \sum\limits_{a \in S^{non-transversal}(R)}q^{4n^2+2n}(1-q^{-2})(1-q^{-4}) \cdots (1-q^{-2n})}{q^{4n^2+6n}} \\
&= \frac{\sum\limits_{T \in V^{nonreg}(k(v))}1 + \alpha_v.q^{2n^2+3n}(1-q^{-2})(1-q^{-4}) \cdots (1-q^{-2n})}{q^{2n^2+3n}}.
\end{align*}
This implies that
\begin{align*}
1-\beta_v &= \frac{|V^{reg}(k(v))|}{q^{2n^2+3n}}-\alpha_v.(1-q^{-2})(1-q^{-4}) \cdots (1-q^{-2n}) \\
&=(1-\alpha_v)(1-q^{-2})(1-q^{-4}) \cdots (1-q^{-2n}) \hspace{2cm}\text{By (1.6)}.
\end{align*}
The proof is completed.
\end{proof}
To summarize, we have just computed the density of regular sections among the set of global sections of our interested vector bundles. Note that in the statement of the main theorem 1, we are taking the weighted average with the weight is the size of automorphism group. Because of that, we also need to estimate the size of automorphism groups of principal $G-$bundle. That is the purpose of following subsection.  
\subsection{Principal $G-$bundles and semistability}
In this section, we will recall the notion and theory of semi-stable principal $G-$bundles which will play a role in later sections.

 Let $X$ be a smooth projective curve defined over $\mathbb{F}_q$, a principal $G-$bundle (or $G-$torsor) over $X$ is a variety $E$ equipped with a right action of $G$ and a $G-$invariant smooth projection $\pi :E \rightarrow X$ such that the map $E\times_X (X\times G) \rightarrow E \times_X E$ of fiber products over $X$ that maps $(y,(x,g)) \mapsto (y, yg)$ is an isomorphism, where $x \in X, \, y \in \pi^{-1}(x)$ and $g \in G$. In particular, $G$ acts freely on the right of $E$ with $X$ as the quotient. Isomorphism classes of $G-$bundles are classified by the non-abelian \'{e}tale cohomology set $H^1(X, G)$. For any quasi-projective scheme $F$ on which $G$ acts on the left, consider the action of $G$ on $E \times F$ defined by $g \circ (e, f) = (eg,g^{-1}f)$, where $g \in G, e \in E$ and $f \in F$. The quotient $(E \times F)/G$ for this action, which is a fiber bundle over $X$, will be denoted by $E(F)$. For instance, let $\rho: G \rightarrow G'$ be a group homomorphism of algebraic group, we consider the action of $G$ on $G'$ defined by left-multiplication, then the fiber $E(G')$ is a principal $G'-$bundle over $X$. 
\begin{definition}
A reduction $\sigma$ of the structure group of a principal $G-$bundle $E$ to a subgroup $P$ of $G$ is a section $\sigma: X \rightarrow E/P$ of the fiber bundle $E/P = E(G/P)$. Then $\sigma^*E$ is a $P-$bundle over $X$ and there is a natural isomorphism of $G-$bundles $\sigma^*E(G) \cong E$.
\end{definition}
Ramanathan's definition of stability for principal bundles can be phrased in many different ways (e.g. see Theorem 2.2. of \cite{FM98}, and also \cite{BS02}). The original version reads as follows:
\begin{definition}(\textbf{Semistability}) 
A principal $G-$bundle $E$ over $X$ is said to be semistable if for every reduction $\sigma: X \rightarrow E/P$ of the structure group to a parabolic $P$, we have
\begin{equation*}
\text{degree}(\sigma^*T_{E/P}) \geq 0
\end{equation*}
\end{definition}
Here, $T_{E/P},$ the tangent bundle over $E/P$ along the fiber of $E$, is defined to be the vector bundle $E(\mathfrak{g}/\mathfrak{p})= (E \times (\mathfrak{g}/\mathfrak{p}))/P$ over $E/P$ associated to the $P-$bundle $E \rightarrow E/P$ and the $P$ action on $\mathfrak{g}/\mathfrak{p}$ ($\mathfrak{p}$ is the Lie algebra of $P$) induced by the adjoint representation.

When $G=GL(n)$, this stability condition is consistent with the condition defined by Mumford: the vector bundle $E$ is called to be semi-stable if for any proper subbundle $F$ of $E$, we have the following inequality
\begin{equation*}
\frac{deg(F)}{rank(F)} \leq \frac{deg(E)}{rank(E)}.
\end{equation*}
 The ratio $\frac{deg(F)}{rank(F)}$ is called the slope of the vector bundle $F$ and it is denoted by $\mu(F)$. 
 
 The number of global sections of a semi-stable vector bundle can be estimated by using Riemann-Roch and Serre's duality, we summarize these results in the following proposition for convenience:
\begin{proposition}(Lemma 4.4 in \cite{Tho16}) 

Let $E$ be a semi-stable vector bundle of rank $r$ over a curve $X$, $g$ denotes the genus of $X$, and $h^0(...)$ the dimension of the space $H^0(...)$ of global sections. Then
\begin{description}
\item[a)]If $\mu(E) < 0$, then $h^0(X, E) =0$.
\item[b)]If $0 \leq \mu(E) \leq 2g-2$, then $h^0(C,E) \leq r(1+\mu(E)/2).$
\item[c)]If $\mu(E) > 2g-2$, then $h^0(C,E)=r(1-g+\mu(E)).$
\end{description}
\label{global sections of semistable}
\end{proposition}

\begin{definition} (Canonical reduction) Let $E$ be a principal $G-$bundle, a reduction of the structure group $(P, \sigma)$ of $E$ to a parabolic subgroup $P$ of $G$ is called $canonical$ if the following two conditions hold:
\begin{enumerate}
\item The Levi bundle $E_L$ associated, by extension of structure group, to $\sigma^*E$ for the projection $P \rightarrow L$ is semi-stable.

\item For every non-trivial character $\chi$ of $P$ which is a non-negative linear combination of simple roots with respect to some Borel subgroup contained in $P$, the line bundle $\chi_*\sigma^*E$ on $X$ has positive degree.
\end{enumerate}
\end{definition}
We also remark that in the case $G=GL(n)$, the canonical reduction corresponds precisely to the Harder-Narasimhan filtration of the associated vector bundle of rank $n$.

For any $G-$bundle $E$, there exists a unique canonical reduction, for the detail of the proof of this fact, the reader can look at \cite{BH04}.

Since counting global sections of semistable vector bundles is much the same as counting sections of line bundles (we just need to use Riemann- Roch theorem), we need to use Parabolic canonical reduction to reduce our problem to semistable vector bundles. Here are the details: Given a $G-$bundle $E$ and an representation $V$ of $G$, we consider the associated vector bundle $(E \times V)/G$ (we will denote it by $E \times^G V$). If $(P,\sigma)$ denotes the canonical reduction of $E$ (with associated $P-$bundle $E_P$), then one has the following isomorphism:
\begin{equation*}
E \times^G V \cong E_P \times^P V_P,
\end{equation*}
where $V_P$ is the restricted representation of $V$ to $P$. Suppose that we have a filtration of $P-$modules 
\begin{equation*}
0=V_0 \subset V_1 \subset V_2 \subset \dots \subset V_r = V_P,
\end{equation*}
such that $W_i:= V_i/V_{i-1}$ ($i= \overline{1,r}$) are irreducible as $P-$modules. If $U$ is the unipotent radical of $P$ and $L:= P/U$ is the Levi factor, then from the assumption on $V_i$ it follow that $U$ acts trivially on $W_i$. In other words, the actions of $P$ on $W_i$ factor through the quotient $L$. Additionally, since $E_L$ is a semistable $L-$bundle and $W_i$ is irreducible as $L-$module, $E_L \times^L W_i$ will be semistable vector bundles for all $i$ (Theorem 3.18 in \cite{RR84}). In conclusion, we have a $P-$equivariant filtration of the vector bundle $E \times^G V$ :
\begin{equation*}
0 = E_P \times^P V_0 \subset E_P \times^P V_1 \subset \dots \subset E_P \times^P V_r = E \times^G V  
\end{equation*}
such that each consecutive quotients $\mathscr{M}_i = E \times^P V_i / E \times^P V_{i-1}$ is semistable. 

Let us apply the above argument to the case $G = SO(W)$, where $W$ is a non-degenerate, split orthogonal space over $k$, of dimension $2n+1$ as discussed in Section 1.3. Suppose that a $G$-bundle $E$ has the canonical reduction $(P, \sigma)$ and the parabolic subgroup $P$ has the Levi quotient given by 
\begin{equation*}
L \cong GL(n_1) \times GL(n_2) \times \cdots \times GL(n_t) \times SO(2m+1).
\end{equation*}
This means that there exists a flag of isotropic subspaces $$0 = V_0 \subset V_1 \subset \cdots \subset V_{t} \subset V_{t}^{\bot} \subset \cdots \subset V_1^{\bot} \subset W,$$ 
where $dim(V_i/V_{i-1})=n_i$ for $1\leq i \leq t$, and $dim(V_{t}^{\bot}/V_t)= 2m+1$. From this, we obtain a filtration of the vector bundle $E \times^G W$: 
$$ 0=E_P \times^P V_0 \subset E_P \times^P V_1 \subset \cdots \subset E_P \times^P V_{t} \subset E_P \times^P V_t^{\bot} \subset \cdots \subset E_P \times^P V_1^{\bot} \subset E \times^G W $$
such that the quotient bundles $X_i = E_P \times^P V_i / E_P \times^P V_{i-1}$ for $1 \leq i \leq t$ and $X_{t+1} = E_P \times^P V_t^{\bot} / E_P \times^P V_t$ are semistable, and $E_P \times^P V_{i-1}^{\bot} / E_P \times^P V_{i}^{\bot}$ is dual to $X_i$. We denote those dual quotients by $X_i^*$ (note that $X_{t+1}$ is self-dual). If we denote the slope of vector bundle $X_i$ by $\mu_i$, then the "canonical" conditions imply that $\mu_1 > \mu_2 > \cdots > \mu_t > \mu_{t+1}=0.$ 

From the above filtration, we obtain a filtration for the vector bundle $E \times^G W^{\otimes 2}$:
$$0=E_P \times^P (V_0 \otimes W)  \subset E_P \times^P (V_1 \otimes W)  \subset \cdots \subset E_P \times^P (V_1^{\bot} \otimes W) \subset E \times^G (W \otimes W)$$and the quotient bundle $Y_i= \Big(E_P \times^P (V_i \otimes W) \Big) \Big/ \Big(E_P \times^P (V_{i-1} \otimes W)\Big) \cong X_i \otimes (E \times^G W)$, for $ 1 \leq i \leq t$, also has a filtration:
$$0 = X_i \otimes (E_P \times^P V_0) \subset X_i \otimes (E_p \times^P V_1) \subset \cdots \subset X_i \otimes (E_P \times^P V_1^{\bot}) \subset X_i\otimes (E \times^GW)$$whose quotient bundles are of the form: $X_i \otimes X_j \cong E_P \times^P \big((V_i/V_{i-1})\otimes (V_j /V_{j-1}) \big)$, $X_i \otimes X_j^*$, for $1 \leq j \leq t$,  and $X_i \otimes X_{t+1}$. Similarly, we also obtain a filtration for the quotient bundles  $Y_i'= \Big(E_P \times^P (V_{i-1}^{\bot} \otimes W) \Big) \Big/ \Big(E_P \times^P (V_i^{\bot} \otimes W)\Big) \cong X_i^* \otimes (E \times^G W)$, for $1 \leq i \leq t$, and $Y_{t+1}= \Big(E_P \times^P (V_t^{\bot} \otimes W) \Big) \Big/ \Big(E_P \times^P (V_t \otimes W)\Big) \cong X_{t+1} \otimes (E \times^G W)$.

From the above "square filtration" of $E \times^G W^{\otimes 2}$, by taking quotient modulo $E \times^G I$, where $I$ is the sub vector space of $W \otimes W$ spanned by $w \otimes w$ for all $w \in W$, we will obtain a filtration for the vector bundle $E \times^G \wedge^2(W)$
$$0=E_P \times^P \overline{V_0 \otimes W}  \subset E_P \times^P \overline{V_1 \otimes W}  \subset \cdots \subset E_P \times^P \overline{V_1^{\bot} \otimes W} \subset E \times^G \wedge^2(W),$$and each quotient bundle will have a filtration whose consecutive quotients are of the following forms: $\wedge^2(X_i)$, $\wedge^2(X_i^*)$, $X_i \otimes X_j$ for $i \neq j$, $X_i \otimes X_j^*$ for all $(i,j)$, and $X_i^* \otimes X_j^*$ for $i \neq j$. Notice that these forms are all semi-stable because the $L-$modules $\wedge^2(V_i/V_{i-1}), \wedge^2(V_{i-1}^{\bot}/V_i^{\bot}),$ for $1 \leq i \leq t+1$; $V_i/V_{i-1} \otimes V_j/V_{j-1}$, $V_{i-1}^{\bot}/V_i^{\bot} \otimes V_{j-1}^{\bot}/V_j^{\bot}$ for $ i \neq j$; and $V_i/V_{i-1} \otimes V_{j-1}^{\bot}/V_j^{\bot}$ for all $(i,j)$ are irreducible. This filtration will be helpful when we want to estimate the size of $Aut_G(E)(\mathbb{F}_q)$. 

Now we will see an estimation of sizes of automorphism groups of $G-$bundles. Given a principal $G-$bundle $E$ and its canonical reduction as above, we recall that the Lie algebra of $Aut_G(E)$ can be written as the space of global sections of the adjoint bundle. 
\begin{eqnarray*}
Lie(Aut_G(E)) &=& H^0(C, ad(E)) = H^0(C, E \times^G \mathfrak{g}),
\end{eqnarray*}
where $\mathfrak{g}=so(W) \cong \wedge^2(W)$ is the Lie algebra of $G.$ 
 
In case our base curve $C$ is an elliptic curve, the semistable filtration of $E \times^G \mathfrak{g}$ even splits, i.e. $E \times^G \mathfrak{g}$ is isomorphic to its associated graded bundle $\bigoplus_{i = 1}^a \mathcal{F}_i$, where $\{\mathcal{F}_i\}_{1\leq i \leq a} = \{\wedge^2(X_i)$, $\wedge^2(X_i^*)$, $X_i \otimes X_j$ for $i \neq j$, $X_i \otimes X_j^*$ for all $(i,j)$, and $X_i^* \otimes X_j^*$ for $i \neq j \}$. This phenomenon comes from the vanishing of higher cohomology of indecomposable vector bundles of positive degree over an elliptic curve and the additivity of the slope for tensor products of vector bundles. Furthermore, if we consider $\mathfrak{g}$ as $L-$module: $\mathfrak{g}= \bigoplus_{\nu \in M} \mathfrak{g}(\nu),$ where $M$ denote the additive subgroup of $\mathbb{Q}$ generated by the slopes $\mu_i$ of vector bundle $E \times^G V_i/V_{i-1}$, then the expression of $Lie(Aut_G(E))$ can be decomposed further:
\begin{eqnarray*}
Lie(Aut_G(E)) &=& H^0(C, ad(E)) = H^0(C, E_L \times^L \mathfrak{g}) \\
&=& H^0(C, ad(E_L)) \oplus \bigoplus_{\nu \in M_{>0}} H^0(C, E_L \times^L \mathfrak{g}(\nu)), 
\end{eqnarray*}
where $M_{>0}$ denotes the subset of positive elements in $M.$ The following proposition shows us how to calculate the automorphism group in case of elliptic curves (see $\S 2$ of \cite{HS01})
\begin{proposition}Assume that $C$ is an elliptic curve and $E$ is a principal $G-$bundle over $C$ with canonical reduction $(P,\sigma).$ Let $\rho$ be the half sum of the positive roots and $Aut_G(E)^+$ be the connected subgroup of $Aut_G(E)$ corresponding to the Lie subalgebra $\oplus_{\nu \in M_{>0}} H^0(C, E_L \times^L \mathfrak{g}(\nu))$ of $Lie(Aut_G(E))$.

 (a) We have equality for the identity components, $Aut_G(E)^{\circ}=Aut_P(E_P)^{\circ}.$ 
 
 (b) The group $Aut_G(E)^+$ is a unipotent normal subgroup of $Aut_G(E)$ of dimension $\langle 2\rho, \mu(E) \rangle ,$ and $Aut_G(E)^{\circ}$ is a semidirect product $Aut_L(E_L)^{\circ}  \ltimes Aut_G(E)^+.$ 
 
(c) $dim (Aut_G(E)) = dim (Aut_L(E_L)) + \langle 2\rho, \mu(E) \rangle.$
\end{proposition}

For a general curve $C$, our vector bundle $E \times^G \mathfrak{g}$ is not isomorphic to its associated graded bundle, but we have a filtration whose consecutive quotients are $E_L \times^L \mathfrak{g}(\nu), \nu \in M$, hence if we put some extra conditions on the slopes of $E_L \times^L \mathfrak{g}(\nu)$ for $\nu \in M_{>0}$ (e.g. the maximal slope is big enough), then the above result can be applied inductively. For the other cases, we still have some estimations (using Proposition \ref{global sections of semistable}). Here are the details: 

Let us consider the following filtration of the vector bundle $E \times^G W:$ 
$$ 0 = E_P \times^P V_0 \subset E_P \times^P V_1  \subset \dots \subset E_P \times^P V_{t+1} \subset E_P \times^P V_1^{\bot} \subset \dots \subset E \times^G W.$$
If we put an extra condition that $\mu_1 - \mu_2 > 2g-2,$ where $\mu_i$ denotes the slope of vector bundle $X_i = E_P \times^P V_i/E_P \times^PV_{i-1},$ then the following exact sequence of vector bundles :
\begin{equation}
0 \longrightarrow X_1 \longrightarrow E \times^G W \longrightarrow (E \times^G W)/X_1 \longrightarrow 0
\end{equation}
is split. In fact, the split-ness is equivalent to $H^1(C, ((E \times^G W) /X_1)^* \otimes X_1) = 0$. By Serre's duality, it is equivalent to show that $H^0(C, (E \times^G W) /X_1 \otimes X_1^* \otimes \omega_C) = 0,$ where $\omega_C$ is the cannonical sheaf of $C$. Firstly, we have a filtration of $(E \times^G W) /X_1 \otimes X_1^*$ whose consecutive quotients are semistable and the associated graded bundle is $(X_1^* \otimes X_1^*) \oplus \bigoplus_{i=2}^{t+1} ( (X_i \otimes X_1^*) \oplus (X_i^* \otimes X_1^*) ).$ Thus, the following relations are immediate:
\begin{eqnarray*}
 h^0(C, (E \times^G W) /X_1 \otimes X_1^* \otimes \omega_C) \leq h^0(C, gr((E \times^G W) /X_1 \otimes X_1^* \otimes \omega_C)), \\
 H^0(C, gr( (E \times^G W)/X_1 \otimes X_1^* \otimes \omega_C))= H^0(C, gr((E \times^G W)/X_1 \otimes X_1^*) \otimes \omega_C ),\\
 H^0(C, gr((E \times^G W) /X_1 \otimes X_1^*) \otimes \omega_C)= H^0( X_1^* \otimes X_1^* \otimes \omega_C) \oplus \hspace{3cm}\\  \bigoplus_{i=2}^{t+1} \Big( H^0(X_i \otimes X_1^* \otimes \omega_C) \oplus H^0(X_i^* \otimes X_1^*\otimes \omega_C) \Big).
\end{eqnarray*}
Each direct summand in the last expression is the set of global sections of semistable vector bundle of negative degree, hence by the properties of semistable vector bundles, they are all trivial. Thus, we have just proved that the exact sequence $(2)$ is split. Similarly, since $E \times^G W$ is self-dual, $X_1^*$ is also a direct summand of $E \times^GW$. In other words, we have the following decomposition:
$$E \times^G W = X_1 \oplus X_1^* \oplus Y,$$ where $Y \cong E \times^G (V_1^{\bot}/V_1)$. As a consequence, the vector bundle $E \times^G \wedge^2W$ can be decomposed as 
$$E \times^G \wedge^2W = \wedge^2X_1 \oplus \wedge^2X_1^* \oplus \wedge^2Y \oplus (X_1 \otimes Y) \oplus (X_1 \otimes X_1^*) \oplus (Y \otimes X_1^*). $$
This implies that
\begin{eqnarray*}
 Lie(Aut_G(E)) & \cong & H^0(C, ad(E)) \cong  H^0(C, E \times^G \wedge^2W) \\
 &=& H^0(\wedge^2X_1) \oplus H^0(X_1\otimes Y) \oplus H^0(\wedge^2Y) \oplus H^0(X_1\otimes X_1^*),
\end{eqnarray*} 
since semistable vector bundles of negative degree have no global sections. Now if we consider $G'= SO(V_1^{\bot}/V_1)$ as a subgroup of $G = SO(W)$ via the natural embedding:
\begin{eqnarray*}
\phi :  G' & \longrightarrow & G \\
 T & \longmapsto & diag(I_{n_1}, A, I_{n_1}),
\end{eqnarray*}
and denote the reduction of $E$ to $G'$ by $E_{G'}$. Then from the above description of $Lie(Aut_G(E)),$ we obtain the following equality:
\begin{eqnarray*}
\textrm{dim}Aut_G(E) =  \textrm{dim}Aut_{G'}(E_{G'}) + \textrm{dim}Aut(X_1) + h^0(\wedge^2X_1) + h^0(X_1 \otimes Y).
\end{eqnarray*}
In general, we can prove the following proposition:
\begin{proposition} a) Suppose that $\mu_i - \mu_{i+1} > 2g-2$ for some $1 \leq i \leq t$, then the exact sequence of vector bundles
\begin{equation}
0 \rightarrow E_P \times^P V_i \rightarrow E \times^G W \rightarrow (E \times^G W) /(E_P \times^P V_i) \rightarrow 0
\end{equation}
is split.\\
b) There exists a constant $c$ that only depends on the genus $g$ of the curve $C$, such that for any $G-$bundle $E$  E with canonical reduction to $P$, we have the following inequality:
\begin{eqnarray*}
|Aut_G(E)(\mathbb{F}_q| \geq  c. |Aut_{SO(2m+1)}(X_{t+1})|.\prod_{i=1}^t \big(|Aut_{GL(n_i)}(X_i)|. |H^0(\wedge^2X_i)| \big) . \\  \prod_{i=1}^{t-1} \prod_{j=i+1}^{t} \big(|H^0(X_i \otimes X_j)| . |H^0(X_i \otimes X_j^*)|\big).\prod_{i=1}^t |H^0(X_i \otimes X_{t+1})|,
\end{eqnarray*}
where $H^0(F)$ is the notation of $H^0(C,F)$-the set of global sections over $C$ of the vector bundle $F$. \\
c) In particular, if $rank(X_j) = 1$ for all $j$ and $\mu_j - \mu_{j+1} > 2g-2$ for all $j$, we obtain the following equalities:
\begin{eqnarray*}
E \times^G W \cong X_1 \oplus X_2 \oplus \cdots \oplus X_m \oplus \mathscr{O}_C \oplus X_m^* \oplus \cdots \oplus X_1^*, \\
|Aut_G(E)(\mathbb{F}_q)| = |Aut_L(E_L)(\mathbb{F}_q)| .  |Aut_G(E)^+(\mathbb{F}_q)| = (q-1)^m .  |Aut_G(E)^+(\mathbb{F}_q)|, 
\end{eqnarray*}
where the summands that appear in $Aut_G(E)^+$ are $E_L \otimes^L \wedge^2(\mathscr{O}_C)$, $E_L \otimes^L \wedge^2(X_i)$ for all $i$, $E_L \otimes^L ( W_0 \otimes X_i)$ for all $i$, $E_L \otimes^L (X_i \otimes X_j)$ and $E_L \otimes^L (X_i \otimes X_j^*)$ for all $i < j$.
\label{splitting when large unstability}
\end{proposition}
\begin{proof}
To show the statement $a)$, we will use a similar argument as the one we used to proved that the exact sequence $(2)$ is split in case $\mu_1 - \mu_2 > 2g-2.$ In fact, as above we consider the extension group $H^1\Big(C, \big((E \times^G W) /(E_P \times^P V_i)\big)^*\otimes (E_P \times^P V_i) \Big),$ and try to show that it vanishes. By the hypothesis $\mu_i - \mu_{i+1} > 2g-2,$ it is easy to see that the vector bundle $\big((E \times^G W) /(E_P \times^P V_i)\big)^*\otimes (E_P \times^P V_i)$ has a filtration of semistable vector bundles of the form $X_j \otimes Z$ where $1 \leq j \leq i$ and $Z$ belongs to the set $\big\{ X_j \textrm{ for } 1 \leq j \leq t+1; X_h^* \textrm{ for } i+1 \leq h \leq t+1 \big\},$ whose slopes are all strictly bigger than $2g-2$. As a consequence, the vector bundle $\big((E \times^G W) /(E_P \times^P V_i)\big)^*\otimes (E_P \times^P V_i) \otimes \omega_C$ has no global section. By using Serre's duality, we have completed the proof of $a)$.

 Before proving $b)$, note that we do not assume any conditions related to the slope of $X_i$ here. Furthermore, on the right hand side of the inequality in $b)$, the factors $|Aut_{H}(X_i)(\mathbb{F}_q)|$ can be included in the constant $c$, where $H$ is $GL(n_i)$ or $SO(2m+1)$. In fact, since the moduli space of semi-stable vector bundle of fixed degree $l$ and rank $r$ is of finite type over $\mathbb{F}_q$ (c.f. \cite{Mar81}), we only have a finite number of isomorphism classes of $(l,r)-$ semi-stable vector bundle. Consequently, sizes of automorphism groups of any $(l,r)-$semi-stable vector bundle are bounded above by a constant that depends only on $l, r$ and $C$. Notice that since we assume that the curve $C$ has a $\mathbb{F}_q-$rational point, any $(l,r)-$ semi-stable vector bundles can be translated to semi-stable vector bundles of type $(l',r)$ with $0 \leq l' < r$ by tensoring with some line bundles. By using the fact that $Aut_{GL(n_1)}(X_i) \cong Aut_{GL(n_i)}(X_i \otimes \mathcal{F})$ for any line bundle $F$, we deduce that sizes of automorphism groups $Aut_{H}(X_i)(\mathbb{F}_q)$ are universally bounded by a constant that depends only on $C$ and $n$. Now we consider some cases as follows:
 \begin{itemize}
 \item[Case 1:] If $\mu_i-\mu_{i+1} \leq 2g-2$ for all $1 \leq i \leq t$, then it is easy to produce the constant $c$ since every factors on the right hand side of the inequality in $b$ are universally bounded by \ref{global sections of semistable}. 
 \item[Case 2:] If there exists only one index $i$ such that $\mu_i-\mu_{i+1} > 2g-2$. By the results in $a)$, we have the following splitting exact sequence:
 \begin{equation}
 0 \rightarrow \mathcal{E}_P \times^P V_i \rightarrow \mathcal{E} \times^GW \rightarrow (\mathcal{E} \times^G W)/(\mathcal{E}_P \times^P V_i) \rightarrow 0.
  \end{equation}
  Hence, by denoting the vector bundles $\mathcal{E}_P \times^P V_i$ and $\mathcal{E}_P \times^P (V_i^{\bot}/V_i)$ by $Y_i$ and $Z$ respectively, we obtain the following decomposition: 
  $$\mathcal{E} \times^GW = Y_i \oplus Z \oplus Y_i^*.$$
  From the assumption that $\mu_i-\mu_{i+1} > 2g-2,$ we deduce that $H^0(C, Z \otimes Y_i^*) =0$ and $H^0(C, (Y_i^*)^{\otimes 2})=0$. Hence any elements of the automorphism group $Aut_{GL(W)}(\mathcal{E} \times^GW)(\mathbb{F}_q)$ are of the following form:
  $$ \begin{pmatrix}
  A & B &C \\
  0 & D & E \\
  0 & 0 & F
  \end{pmatrix},$$where $A,F \in Aut_{GL(V_i)}(Y_i)$; $D \in Aut_{GL(V_i^{\bot}/V_i)}(Z)$; $B, E \in H^0(C, Y_i \otimes Z)$; and $C \in H^0(C, Y_i^{\otimes 2})$. Observe that any automorphisms of $G-$bunde $\mathcal{E}$ can be seen automorphisms of $GL(W)-$bundle $\mathcal{E} \times^GW$ which preserve the orthogonal structure of $\mathcal{E}$. From the above description of $Aut_{GL(W)}(\mathcal{E} \times^GW)(\mathbb{F}_q)$, we obtain that: \small
  \begin{eqnarray*} Aut_G(\mathcal{E})(\mathbb{F}_q)  = \left\{ T= \begin{pmatrix}
  A & B &C \\
  0 & D & E \\
  0 & 0 & F
  \end{pmatrix} \in Aut_{GL(W)}(\mathcal{E} \times^GW)(\mathbb{F}_q) \left\| \begin{array}{ll} T.T^* = \text{Id} \\ det(T)=1 \end{array}\right.\right\} \\ 
  = \left\{ T= \begin{pmatrix}
  A & B &C \\
  0 & D & -DB^*(A^*)^{-1} \\
  0 & 0 & (A^*)^{-1} 
  \end{pmatrix} \left\| \begin{array}{llll} A \in Aut_{GL(V_i)}(Y_i)(\mathbb{F}_q) \\ D \in Aut_{SO(V_i^{\bot}/V_i)}(Z)(\mathbb{F}_q) \\
  B \in H^0(Y_i \otimes Z) \\
  C \in H^0(Y_i^{\otimes 2}) : AC^*+CA^* = -BB^*
  \end{array} \right. \right\}
  \end{eqnarray*} \normalsize
This allows us to compute the size of $Aut_G(\mathcal{E})(\mathbb{F}_q):$
\begin{align*}
Aut_G(\mathcal{E})(\mathbb{F}_q)| = |Aut_{GL(V_i)}(Y_i)(\mathbb{F}_q)|.|Aut_{SO(V_i^{\bot}/V_i)}(Z)(\mathbb{F}_q)|.|H^0(Y_i \otimes Z)|. \\ .|H^0(\wedge^2(Y_i))|.
\end{align*}
By using the condition $\mu_i - \mu_{i+1} > 2g-2$ and the semi-stable filtration of $\mathcal{E} \times^G W$, we imply that
 $$H^0(Y_i \otimes Z) = \bigoplus_{j=1}^{i} \bigg( H^0(X_j \otimes X_{t+1}) \oplus \bigoplus_{h=i+1}^t \big(H^0(X_j \otimes X_h) \oplus H^0(X_j \otimes X_h^*) \big) \bigg),$$ and 
 $$H^0(\wedge^2(Y_i)) = \bigoplus_{1 \leq l < h \leq i} H^0(X_l \otimes X_h) \bigoplus_{j=1}^i H^0( \wedge^2(X_j).)$$ The proof of $b)$ in this case is completed by noting that we can bound  $|H^0(X_l \otimes X_h^*)|$ for $1 \leq l < h \leq i$; $|H^0(X_j \otimes X_f)|.|H^0(X_j \otimes X_f^*)|$ for $i+1 \leq j,f \leq t+1$ by a constant that only depends on the genus $g$ and $G$ because of the assumption in this case: $\mu_j-\mu_{j+1} \leq 2g-2$ for all $j \neq i$. 
 \item[Case 3:] In the decreasing sequence $\mu_1 > \mu_2 > \cdots > \mu_t > \mu_{t+1}=0$, we assume that $t_1 >\cdots > t_f$ are all indexes from $1$ to $t+1$ satisfying that $\mu_{t_j}-\mu_{t_j+1} > 2g-2$. In this case we have the following decomposition of the vector bundle $\mathcal{E} \times^G W$ as follows:
 $$\mathcal{E} \times^G W = Y_{t_1} \oplus Y_{t_1}^* \oplus \cdots Y_{t_{f}} \oplus Y_{t_f}^* \oplus Z,$$where $Y_{t_j} = (\mathcal{E}_P \times^P V_{t_j}) / (\mathcal{E}_P \times^P V_{t_{j-1}})$ (here we set $t_0=0$), and $Z = \mathcal{E}_P \times^P (V_{t_f}^{\bot}/V_{t_f})$. This case can be treated similarly as the previous case. 
  \end{itemize}
For $c)$, if the $G-$bundle $\mathcal{E}$ has canonical reduction to the Borel subgroup and $\mu_i-\mu_{i+1}$ for all $i$, then the associated vector bundle $\mathcal{E} \times^G W$ is the direct sum of line bundles:
$$\mathcal{E} \times^G W \cong X_1 \oplus X_2 \oplus \cdots \oplus X_m \oplus \mathscr{O}_C \oplus X_m^* \oplus \cdots \oplus X_1^*.$$ Hence, we could apply the same argument as in the proof of $b)$ to prove $c)$ and notice that $|Aut_{\mathbb{G}_m}(\mathcal{L})(\mathbb{F}_q)|= q-1$ for any line bundle $\mathcal{L}$.
\end{proof}

\subsection{Counting}
\label{couting section}
Now we fix a parabolic $P$ of $G$ with the Levi quotient $L=GL(n_1)\times \cdots GL(n_t) \times SO(2m+1)$, $2(n_1+\cdots+n_t) +(2m+1)=2n+1$. This means that there exists a flag of isotropic subspaces $$0 = V_0 \subset V_1 \subset \cdots \subset V_{t} \subset V_{t}^{\bot} \subset \cdots \subset V_1^{\bot} \subset W,$$ 
where $dim(V_i/V_{i-1})=n_i$ for $1\leq i \leq t$, and $dim(V_{t}^{\bot}/V_t)= 2m+1.$ We denote $X_i$ be the vector bundle $E_L \times^L V_i/V_{i-1}$, $\mu_i$ be the slop of $X_i$, for $1 \leq i \leq t$, and $W_0=E_L \times^L V_t^{\bot}/V_t$. Notice that $W_0$ is self-dual, hence its slope $\mu_{t+1}$ is zero. 

If $\mu_i-\mu_{i+1} > 2g-2$ for some $i$, by Proposition \ref{splitting when large unstability}a), we have the following isomorphism:
\begin{align*}
E \times^G W & \cong (E_P \times^P V_i) \oplus (E_P \times^P V_i^*) \oplus (E_P \times^P (V_i^{\bot}/V_i)) \\
& = Y \oplus Y^* \oplus W_i,
\end{align*}where $Y = E_P \times^P V_i$ and $W_i = E_P \times^P (V_i^{\bot}/V_i)$. This implies that $E \times^G Sym^2(W)$ can be decomposed as the following direct sum:
\begin{align*}
 & \cong Sym^2(Y) \oplus Sym^2(Y^*) \oplus Sym^2(W_i) \oplus (Y \otimes Y^*) \oplus (Y \otimes W_i) \oplus (Y^* \otimes W_i) \\
& \cong E_P \times^P(Sym^2(V_i) \oplus Sym^2(V_i^*) \oplus Sym^2(V_i^{\bot}/V_i) \oplus (V_i \otimes V_i^*)\\
& \hspace{3cm} \oplus (V_i \otimes (V_i^{\bot}/V_i)) \oplus (V_i^* \otimes V_i^{\bot}/V_i))
\end{align*}
Let $J_{r}$ denote the anti-diagonal matrix of size $r$, then as $G-$modules, $Sym^2_0(W) \cong Sym^2(W)/\langle J_{2n+1} \rangle$. So we have:
\begin{align*}
E \times^G Sym_0^2(W) & \cong E_P \times^P\Big( \big(Sym^2(V_i) \oplus Sym^2(V_i^*) \oplus Sym^2(V_i^{\bot}/V_i) \oplus (V_i \otimes V_i^*) \\
& \hspace{4cm}\oplus (V_i \otimes (V_i^{\bot}/V_i)) \oplus (V_i^* \otimes (V_i^{\bot}/V_i))\big)/\langle J_{2n+1} \rangle \Big) \\
& \cong E_P \times^P \big( Sym^2(V_i) \oplus Sym^2(V_i^*) \oplus Sym_0^2(V_i^{\bot}/V_i) \oplus (V_i \otimes V_i^*) \\
& \hspace{5cm}\oplus (V_i \otimes (V_i^{\bot}/V_i)) \oplus (V_i^* \otimes V_i^{\bot}/V_i)) \big). 
\end{align*}
The last isomorphism comes from the following lemma:
\begin{lemma} We have the following isomorphism of $P-$modules: 
\begin{align*}
& \begin{pmatrix}
Sym^2(V_i) & V_i\otimes (V_i^{\bot}/V_i) & V_i \otimes V_i^* \\
* & Sym^2(V_i^{\bot}/V_i) & (V_i^{\bot}/V_i) \otimes V_i^* \\
* & * & Sym^2(V_i^*)
\end{pmatrix}\Big/ \langle J_{2n+1} \rangle \\
\cong & \begin{pmatrix}
Sym^2(V_i) & V_i\otimes (V_i^{\bot}/V_i) & V_i \otimes V_i^* \\
* & Sym_0^2(V_i^{\bot}/V_i) & (V_i^{\bot}/V_i) \otimes V_i^* \\
* & * & Sym^2(V_i^*)
\end{pmatrix} 
\end{align*}
In the above matrices, we need to fill in $*$ to obtain symmetric matrices. 
\end{lemma}
\begin{proof}Firstly, we can see that $\langle J_{2n+1} \rangle$ only appears in the direct sum $(V_i\otimes V_i^*)\oplus Sym^2(V_i^{\bot}/V_i)$. Additionally, we have the following  surjective morphism of $P-$modules:
\begin{align*}
\varphi: & Sym^2(V_i^{\bot}/V_i) \oplus End(V_i) & \longrightarrow  & Sym_0^2(V_i^{\bot}/V_i) \oplus End(V_i) \\
 & \hspace{2cm}(A_0, A_1)  & \longmapsto  & \Big(A_0- \frac{antr(A_0)}{size(A_0)}J_{size(A_0)}, A_1 - \frac{antr(A_0)}{size(A_0)}J_{size(A_1)} \Big),
\end{align*}
where $antr$ denotes for "anti-trace". The kernel of $\varphi$ is $\langle J \rangle$. The Proposition is proven.
\end{proof}
In general (with out the condition $\mu_i-\mu_{i+1} > 2g-2$), the following inequality is enough for our purpose:
\begin{eqnarray*}
h^0(E \times^G Sym_0^2(W))  \leq \sum_{1 \leq i < j \leq t} \bigg(h^0(X_i \otimes X_j) + h^0(X_i \otimes X_j^*) + h^0(X_i^* \otimes X_j) +\\ + h^0(X_i^* \otimes X_j^*)\bigg)+ \sum_{i=1}^t \bigg( h^0(Sym^2(X_i)) + h^0(Sym^2(X_i^*) + h^0(X_i \otimes W_0) + \\ +h^0(X_i^* \otimes W_0)+ h^0(X_i\otimes X_i^*)\bigg)+ 
+h^0(Sym^2_0(W_0)).  
\end{eqnarray*}
Similarly, the vector bundle associated to the adjoint representation $so(W)= \wedge^2(W)$ has the following "square filtration": 
\[ \left( \begin{array}{ccccccc}
\wedge^2(X_1) & X_1\otimes X_2 & \cdots & X_1\otimes W_0 & \cdots & X_1 \otimes X_2^* & X_1 \otimes X_1^* \\
* & \wedge^2(X_2) & \cdots & X_2\otimes W_0 & \cdots & X_2\otimes X_2^* & X_2 \otimes X_1^*  \\
* & * & \ddots & \vdots & \cdots & \vdots & \vdots \\
* & * & \cdots & \wedge^2(W_0) & \cdots & W_0 \otimes X_2^* & W_0 \otimes X_1^* \\
\vdots & \vdots & \vdots & \vdots & \ddots & \vdots & \vdots \\
* & * & \cdots & * & \cdots & \wedge^2(X_2^*) & X_2^* \otimes X_1^* \\
*&*&\cdots & * & \cdots & * & \wedge^2(X_1^*) \end{array}. \right)\]
Notice that the entries in the upper part of the above matrix are all semi-stable vector bundles, and Proposition \ref{splitting when large unstability} gives us an estimation of the dimension of $H^0(E \times^G \mathfrak{so}(W))$.

Now let recall our set up: $\mathcal{E}$ is a $G$-torsor which has a canonical reduction at the parabolic subgroup $P =L\times N$. Denote $E_P$ the reduction of $E$ and $E_L:= E_P \times^P L$, then we have the following canonical isomorphisms:
\begin{equation*}
V(E,\mathcal{L}) = (\mathcal{E}\times^G V) \otimes \mathcal{L}^{\otimes 2} \cong (E_L \times^L Sym^2_0(W)) \otimes \mathcal{L}^{\otimes 2} 
\end{equation*} 
Since $E_L$ is a semistable principal $L$-bundle, $E_L \times^L X$, where $X$ is an irreducible $L-$module, will be a semi-stable vector bundle. Set $\mu_i=\frac{deg(E_L\times^LX_i)}{rank (E_L\times^LX_i})$ the slope of vector bundles $E_L\times^L X_i$ (for $i = \overline{1,t}$), $d=deg(L)$. Since $L$-module $W_0$ is self-dual, the slope $\mu_{t+1}$ of $E_L\times^L W_0$ equals zero. By the definition of canonical reduction, we obtain the following inequalities:
\begin{equation*}
\mu_1 > \mu_2 > \dots >\mu_t >0
\end{equation*}
We divide into some cases:

\textbf{Case 1:} $\mu_i - \mu_{i+1} > 2d$ for some $1<i<t+1$. Then by Proposition \ref{global sections of semistable}, the vector bundle $E_L\times^L(X_{i+1}\otimes X_i^*) \otimes \mathcal{L}^{\otimes 2}$ has no non-zero global sections. Similarly, by looking at the rectangle with two opposite vertices $E_L\times^L(X_{i+1}\otimes X_i^*)$ and $Sym^2(X_1^*)$ in the above matrix form of $E_L \times^L Sym_0^2(W)$, we can see that any global section of $(E_L\times^L Sym^2_0(W))\otimes \mathcal{L}^{\otimes 2}$ has the following matrix form (see the argument at the beginning of this section):

\[ \left( \begin{array}{ccc}
D&C&A \\
C^t & B & 0\\
A^t & 0 & 0 
 \end{array} \right)\]
 where $A^t$ denotes the transpose of the matrix $A$, and the entries $0$ denote matrices $0$ with appropriate sizes. We will show that every sections of this form is not regular. 
 
\begin{lemma} In case 1, for any sections $s$ of $(E_L\times^L Sym^2_0(W))\otimes \mathcal{L}^{\otimes 2}$, there exists a point $c \in C$ such that the matrix $s(c)$ is not regular.
\end{lemma}
\begin{proof}
Multiplying with the anti-diagonal matrix $J$ does not change the regularity, hence we can consider the following alternative form of $s\in H^0(C, (E_L\times^L Sym^2_0(W))\otimes \mathcal{L}^{\otimes 2})$:
\[ T= \left( \begin{array}{ccc}
A&C&D \\
0 & B & C^*\\
0 & 0 & A^* 
 \end{array} \right)\]
where $A^*$ denotes the reflection of $A$ about the anti-diagonal. \\
If $A$ or $B$ is not regular than obviously our matrix is not regular. Now we suppose that $A$ and $B$ are regular. At any point $x\in C$, we denote $u_i$ (could be the same) and $v_j$ the eigenvalues of $A$ and $B$ respectively, $u_i, v_j \in \bar{k}$. Note that for any symmetric polynomial $f(x_1,x_2, \dots, x_n)$ with $n$ variables ($n \times n$ is the size of matrix $A$), $f(u_1,\dots,u_n)$ will determines a section of the line bundle $\mathcal{L}^{\otimes 2deg(f)}$. Now we consider the characteristic polynomial of $A$ (and also of $A^*$):
\begin{equation*}
g(X)= X^n - tr(A)X^{n-1}+\dots + (-1)^n det(A).
\end{equation*}
Then $g(T)$ will be a section of $(E_L\times^L Sym^2_0(W))\otimes \mathcal{L}^{\otimes 2n}$. In the same manner, let $g_0$ be the characteristic polynomial of $B$ then $g(T).g_0(T)$ will be a section of $(E_L\times^L Sym^2_0(W))\otimes \mathcal{L}^{\otimes 2n+2n_0}$ where $n_0\times n_0$ is the size of $B$. Now by looking at the section of the vector bundle $F$ of form $"D"$ in $g(T).g_0(T)$, we realize that the regularity of $T$ will be gone if the line bundle $det(F)$ vanishes at some points in $C$ and it happens if $deg(L)$ is large enough. Indeed, since our proof is going to be Galois invariant, we could consider everything over the algebraically closure $\bar{k}$ of $k$. For any point $v$ of $C$, set $J_A(v)$ and $J_B(v)$ the Jordan normal forms of $A(v)$ and $B(v)$: 
\begin{equation*}
J_A(v) = P(v) A(v) P^{-1}(v) \hspace{1.5cm} \text{and} \hspace{1.5cm} J_B(v) = Q(v) B(v) Q^{-1}(v). 
\end{equation*} 
It implies that 
\begin{eqnarray*}
G=\left(\begin{array}{ccc}
P(v) & 0 & 0  \\
0 & Q(v) & 0\\
0 & 0 & (P^{-1}(v))^* 
 \end{array} \right) . T(v) . \left( \begin{array}{ccc}
P^{-1}(v) & 0 & 0 \\
0 & Q^{-1}(v)  & 0\\
0 & 0 & P(v)^* 
 \end{array} \right) &\\
 = \left( \begin{array}{ccc}
J_A(v) & C_1 & D_1 \\
0 & J_B(v) & C_1^*\\
0 & 0 & J_A(v)^* 
 \end{array} \right), \hspace{3cm}
\end{eqnarray*}
for some matrices $C_1$ and $D_1$. Since $A$ and $B$ are regular, their Jordan normal forms are formed by blocks whose corresponding eigenvalues are different. Let we consider the case each $A$ and $B$ only have one eigenvalue, namely, $a$ and $b$ respectively. Then $g_0(X)=(X-b)^{n_0}$ is the characteristic polynomial of $J_B(v)$, and $g_0(T(v))$ is conjugate to (note that regularity is invariant under conjugation)
\begin{eqnarray*}
G_0= \left(\begin{array}{ccc}
(J_A(v)-bI_n)^{n_0} & C_1 & D_1  \\
0 & 0 & C_1^*\\
0 & 0 & (J_A(v)^*-bI_n)^{n_0}
 \end{array} \right).
\end{eqnarray*}
By decomposing the product of two matrices $g(G)=(G-aI_n)^n$ and $G_0$ we will obtain the following matrix:
\begin{eqnarray*}
\left(\begin{array}{ccc}
0 & 0 & D_2  \\
0 & 0 & 0\\
0 & 0 & 0
 \end{array} \right),
\end{eqnarray*}
where $D_2 = \sum_{i=0}^{n-1}(J_A(v)-aI_n)^iD_0(J_A(v)^*-aI_n)^{n-1-i}$ for some $n \times n$ matrix $D_0$. Hence $D_2$ is an upper triangular matrix with the same entries in the diagonal. If $det(D_2) = 0$ then $D_2$ is an upper triangular matrix with zero diagonal, hence it will be killed by $(J_A(v)-aI_n)^{n-1}$. It implies that $G$, and also $T(v)$, is not regular. 

The general case can be treated similarly with careful computations. \end{proof}
By the above lemma, the contribution of $\mathcal{M}_L(k)$ in this case to the average is $0$. 

From now on we only need to consider the case $\mu_i-\mu_{i+1} \leq 2d$ for all $i$ and for all canonical reduction $\mathcal{E}_P$. It will be proven by induction on $t$ that our limit is bounded by $3+f(p)$ where $f(x)$ is a rational function and $\lim\limits_{p\rightarrow \infty}f(p) = 0$. 

\textbf{Case 2:} $ d < \mu_1 \leq 2d+\mu_2$ and $0< \mu_i - \mu_{i+1} \leq 2d$ for all $i>1$. \\
We denote $\mathcal{M'}_{L,P}(k)$ be the subset of $\mathcal{M'}_L(k)$ consisting of triple $(\mathcal{E},L,\alpha)$ where $\mathcal{E}$ is a $G$-torsor which has the canonical reduction at $P$ and satisfies all conditions in this case. Now we apply formulae in section 3.2, for $d$ sufficiently large we have that:
\begin{eqnarray*}
 \dfrac{|\mathcal{M'}_{L,P}(k)|}{|\mathcal{A}_{L}(k)|} = \sum_{\mathcal{E}\in \mathcal{M'}_{L,P}(k)}\dfrac{|H^0(V(\mathcal{E},L))|}{|Aut_G(\mathcal{E})|.|\mathcal{A}_L(k)|} \hspace{7cm} \\
\leq \int\limits_{W_0}\int\limits_{0<\mu_t\leq 2d}\dots \int\limits_{0<\mu_2-\mu_3\leq 2d}\int\limits_{d<\mu_1 \leq 2d+\mu_2} g(\mu_2,\mu_3,\dots, \mu_t, W_0). \frac{h_1}{h_2}(\mu_1,\mu_2,\dots,\mu_t, W_0) d\tau, 
\end{eqnarray*}
where $$g(\mu_2,\mu_3,\dots,\mu_t,W_0) = \frac{|H^0((E_L \times^L Sym^2_0(W')\otimes \mathcal{L}^{\otimes 2})|}{|Aut_{G'}(E_{G'})|.|\mathcal{A'}_\mathcal{L}(k)|}$$ with $W' = V_1^{\bot}/V_1$ and $G'=SO(V_1^{\bot}/V_1),$ thus by induction 
\begin{equation*}
\int\limits_{W_0}\int\limits_{0<\mu_t\leq 2d}\dots \int\limits_{0<\mu_2-\mu_3\leq 2d} g(\mu_2,\mu_3,\dots, \mu_t, W_0)d\tau' \leq 3+f(q).
\end{equation*}
And \footnotesize
\begin{align*}
h_1 = |H^0(Sym^2(X_1)\otimes \mathcal{L}^{\otimes 2})|.|H^0(X_1\otimes W_0\otimes \mathcal{L}^{\otimes 2})|.|H^0(W_0\otimes X_1^*\otimes \mathcal{L}^{\otimes 2})|. |H^0(X_1\otimes X_1^*\otimes \mathcal{L}^{\otimes 2})| \\ \prod_{i=2}^t \Big(|H^0(X_1\otimes X_i\otimes  L^2)|.|H^0(X_1\otimes X_i^*\otimes \mathcal{L}^{\otimes 2})|. |H^0(X_i\otimes X_1^*\otimes \mathcal{L}^{\otimes 2})|.|H^0(X_i^*\otimes X_1^*\otimes \mathcal{L}^{\otimes 2})| \Big),
\end{align*} 
and 
\begin{eqnarray*}
h_2 =  q^{\mathlarger{4d. \big(rank(Sym^2(X_1))+rank(X_1\otimes W_0)+rank(X_1\otimes X_1^*)/2+2.\sum_{i=2}^t rank(X_1 \otimes X_i) \big)} }  \\ 
\times |H^0(\wedge^2(X_1)|.|H^0(X_1\otimes W_0|.(\prod_{i=2}^t(|H^0(X_1\otimes X_i|.|H^0(X_1\otimes X_i^*|)).\hspace{3cm} 
\end{eqnarray*} \normalsize
Now if we fix $X_2,\dots, X_t, W_0$, we can bound the integral \small
\begin{eqnarray*}
A&=&\int\limits_{d<\mu_1 \leq 2d+\mu_2} \frac{h_1}{h_2}(\mu_1,\mu_2,\dots,\mu_t, W_0) dX_1\\
&=& \sum_{d<\mu_1 \leq 2d+\mu_2} \hspace{1cm}\int\limits_{Bun_{\mu_1,r_1}^{ss}(\mathbb{F}_q)} \frac{h_1}{h_2}(X_1,X_2,\dots,X_t, W_0) dX_1 \\
&=& \bigg( \sum_{d<\mu_1\leq 2d-\mu_2} + \sum_{2d-\mu_2 <\mu_1 \leq 2d-\mu_3}+\dots +\sum_{2d-\mu_t<\mu_1\leq 2d} + \sum_{2d<\mu_1\leq 2d+\mu_t} +\dots \\ && +\sum_{2d+\mu_3<\mu_1\leq 2d+\mu_2} \bigg)\int\limits_{Bun_{\mu_1,r_1}^{ss}(\mathbb{F}_q)} \frac{h_1}{h_2}(X_1,X_2,\dots,X_t, W_0) dX_1
\end{eqnarray*} \normalsize
(Note that we could have empty cases in the above division, for example, if $\mu_2>d$ then the case $d<\mu_1\leq 2d-\mu_2$ is empty.)\\
For $d<\mu_1\leq 2d-\mu_2$, by applying the inequality (*), we obtain
\begin{align*}
\int\limits_{Bun_{\mu_1,r_1}^{ss}(\mathbb{F}_q)} \frac{h_1}{h_2}(X_1,X_2,\dots,X_t, W_0) dX_1 \hspace{8cm}\\
\leq  \int\limits_{Bun_{\mu_1,r_1}^{ss}(\mathbb{F}_q)} T.\dfrac{\mathsmaller{|H^0(W_0\otimes X_1^*\otimes \mathcal{L}^{\otimes 2})|\mathlarger{\prod}\limits_{i=2}^t \big(|H^0(X_i\otimes X_1^*\otimes \mathcal{L}^{\otimes 2})|.|H^0(X_i^*\otimes X_1^*\otimes \mathcal{L}^{\otimes 2})|\big)}}{q^{-(2d_1+r_1^2(1-g))}.q^{2d(r_2r_1+\dots +r_tr_1+r_0r_1+r_tr_1+\dots +r_2r_1 +\binom{r_1+1}{2}}}dX_1 \hspace{2.5cm}\\
\text{(T is a constant that is independent to d)} \hspace{5.5cm} \\
\leq  \int\limits_{Bun_{\mu_1,r_1}^{ss}(\mathbb{F}_q)} T'. \dfrac{q^{r_1^2(1-g)}}{q^{2d.\binom{r_1+1}{2}}}dX_1  \hspace{1cm}\text{(T' is a constant that is independent to d)} \hspace{1cm}\\ 
\leq T".q^{-2d.\binom{r_1+1}{2}} \hspace{2cm}\text{(T'' is a constant that is independent to d)} \hspace{2.2cm} 
\end{align*} 
The last inequality is the corollary of the fact that the number of semi-stable vector bundles of any fixed types is bounded by a constant that only depends on $g$ and $m$. Hence 
$$\lim_{d\rightarrow \infty} \sum_{d<\mu_1\leq 2d-\mu_2} \,\, \int\limits_{Bun_{\mu_1,r_1}^{ss}(\mathbb{F}_q)} \frac{h_1}{h_2}(X_1,X_2,\dots,X_t, W_0) dX_1  \leq$$  $$ \leq \lim_{d\rightarrow \infty} \sum_{d<\mu_1\leq 2d-\mu_2} T".q^{-2d.\binom{r_1+1}{2}} = 0.$$
If $2d-\mu_2<\mu_1 \leq 2d-\mu_3 $ then  $2d_1+h^0(X_3\otimes X_1^*\otimes \mathcal{L}^{\otimes 2})+h^0(X_3^* \otimes X_1^*\otimes \mathcal{L}^{\otimes 2}) \leq a + 4d.r_3r_1$ and $|H^0(X_2^*\otimes X_1^* \otimes \mathcal{L}^{\otimes 2})| =1$. Hence, $2d-\mu_2<\mu_1 \leq 2d-\mu_3$ leads to 
$$\int\limits_{Bun_{\mu_1,r_1}^{ss}(\mathbb{F}_q)} \dfrac{h_1}{h_2}(X_1,X_2,\dots,X_t, W_0) dX_1 \hspace{7cm}$$
 $$\leq \int\limits_{Bun_{\mu_1,r_1}^{ss}(\mathbb{F}_q)} T.\dfrac{\mathsmaller{|H^0(W_0\otimes X_1^*\otimes \mathcal{L}^{\otimes 2})|\mathlarger{\prod}\limits_{i=3}^t \big( |H^0(X_i\otimes X_1^*\otimes \mathcal{L}^{\otimes 2})|.|H^0(X_i^*\otimes X_1^*\otimes \mathcal{L}^{\otimes 2})| \big)}}{q^{-(2d_1+r_1^2(1-g))}.q^{2d(r_2r_1+\dots +r_tr_1+r_0r_1+r_tr_1+\dots +r_2r_1 +\binom{r_1+1}{2})}}dX_1 \hspace{2cm}$$
  $$\text{(T is a constant independent to d)}\hspace{7cm}$$
$$\leq  \int\limits_{Bun_{\mu_1,r_1}^{ss}(\mathbb{F}_q)} T'. \dfrac{q^{r_1^2(1-g)}}{q^{2d.(r_2r_1+\binom{r_1+1}{2})}}dX_1 \hspace{1cm} \text{(T' is a constant independent to d)} \hspace{1cm} $$
$$\leq  T".q^{-2d.(r_2r_1+\binom{r_1+1}{2})}\hspace{10.4cm}.$$
So we obtain that
\begin{eqnarray*}
&\lim\limits_{d\rightarrow \infty} \hspace{0.5cm}\mathlarger{\sum}\limits_{2d-\mu_2 <\mu_1\leq 2d-\mu_3} \hspace{1cm} \mathlarger{\int}\limits_{Bun_{\mu_1,r_1}^{ss}(\mathbb{F}_q)} \frac{h_1}{h_2}(X_1,X_2,\dots,X_t, W_0) dX_1 \\ \leq &  \lim\limits_{d\rightarrow \infty} \hspace{0.5cm}\mathlarger{\sum}\limits_{2d -\mu_2 <\mu_1\leq 2d-\mu_3} T".q^{-2d.(r_1r_2+\binom{r_1+1}{2})} = 0
\end{eqnarray*}
Similarly, in the period $2d-\mu_i<\mu_1 \leq 2d-\mu_{i+1}$ for any $i$, we also have that the limit is $0$.

If $2d+\mu_{h+1} < \mu_1 \leq 2d+\mu_h$ for $2 \leq h \leq t$. By using the same arguments as above, we have
\begin{eqnarray*}
&&\int\limits_{Bun_{\mu_1,r_1}^{ss}(\mathbb{F}_q)} \frac{h_1}{h_2}(X_1,X_2,\dots,X_t, W_0) dX_1 \\
& = & \int\limits_{Bun_{\mu_1,r_1}^{ss}(\mathbb{F}_q)} T.\dfrac{q^{2d_1+r_1^2(1-g)+2d.r_2r_1}.\prod_{i=2}^h |H^0(X_i\otimes X_1^*\otimes \mathcal{L}^{\otimes 2})|}{q^{2d(r_2r_1+\dots +r_tr_1+r_0r_1+r_tr_1+\dots +r_2r_1 +\binom{r_1+1}{2}}}dX_1 \\
& &\text{(T is a constant independent to d)} \\
&\leq & \int\limits_{Bun_{\mu_1,r_1}^{ss}(\mathbb{F}_q)} T'. \dfrac{q^{r_1^2(1-g)+2d_1+h^0(X_h\otimes X_1^* \otimes \mathcal{L}^{\otimes 2})}}{q^{2d.(r_hr_1+\dots +r_tr_1+r_0r_1+r_tr_1+\dots r_2r_1+ \binom{r_1+1}{2})}}dX_1 \\
&& \text{(T' is a constant independent to d)} \\
&\leq & T".\dfrac{q^{2r_1\mu_1+r_hr_1(\mu_h-\mu_1)}}{q^{2d.(r_{h+1}r_1+\dots +r_tr_1+r_0r_1+r_tr_1+\dots r_2r_1+ \binom{r_1+1}{2})}} \hspace{2cm}\text{(Relative trace formula)} \\
&&  \hspace{2cm}\text{($T"$ is independent to $d$)} \\
& \leq & T".\dfrac{(\mu_1-\mu_h)(2r_1-r_hr_1)+2r_1\mu_h}{q^{2d.(r_{h+1}r_1+\dots +r_tr_1+r_0r_1+r_tr_1+\dots r_2r_1 + \binom{r_1+1}{2})}} \\
& \leq & T".\dfrac{2dr_1+2r_1.2d.(t-h+1)}{q^{2d.((2t-h)r_1+ \binom{r_1+1}{2})}} \\
& \leq & T".\dfrac{2dr_1(2t-2h+3}{q^{2dr_1(2t-h+1)}} \\
& = & T".q^{-2dr_1(h-2)}.
\end{eqnarray*}
By looking at the above inequalities, we can imply that our limit will be $0$ when $h>2$, or $r_i > 1$ for some $i$, or $h=2$ and $\mu_i-\mu_{i+1} \leq d$ for some $i \leq t$. Now we will see what happen if $h=2$, $r_i=1$ for all $i$, and $d_i-d_{i+1} > d$ for any $i \leq t$. Notice that in this case our filtration of vector bundle $E \times^G W$ is split, thus we actually have the following equality:
\begin{eqnarray*}
&A=&\sum_{d+d_2 < d_1 \leq 2d+d_2} \int_{Bun_{d_1,1}(C)(\mathbb{F}_q)} \frac{h_1}{h_2}(X_1,X_2,\dots,X_{n}, W_0) dX_1 \\
&=& \sum_{d+d_2 < d_1 \leq 2d+d_2} \int_{Bun_{d_1,1}(C)(\mathbb{F}_q)} \dfrac{q^{2d_1+(1-g)+h^0(X_2\otimes X_1^*\otimes \mathcal{L}^{\otimes 2})}}{q^{4dn}} dX_1 \\
& \leq & \sum_{d+d_2 < d_1 \leq 2d+d_2} T.\dfrac{q^{2d_1 + 2d +d_2-d_1}}{q^{4dn}} \hspace{3cm}\text{($T$ is only depended on $g$)}\\
& \leq & T.\dfrac{q^{2d+2d_2+1}}{(q-1)q^{4dn-2d}} \\
& = & T.\dfrac{q^{2d_2+1}}{(q-1)q^{4dn-4d}}
\end{eqnarray*}
If $d_2 <2d(n-1)$ then our sum $A$ will be bounded by $T/((q-1)q)$, and if $d_2=2d(n-1)$ then $A$ will be bounded by a constant which is only depend on $g$. In fact, that constant equals to $1$ and reflects the Kostant section. We will see it in the following proposition:
\begin{proposition} 
If $r_i =1$ for all $i$ (i.e $t=n$), $d_i=2d(n-i+1)$ for $1 \leq i \leq n$ and we have isomorphisms: $E_L \times^L (X_{i+1} \otimes X_i^*) \cong E_L \times^L X_n^* \cong \mathcal{L}^{-2}$ for $1 \leq i \leq n-1$, then any sections of $(E_L \times^L Sym_0^2(W)) \otimes \mathcal{L}^{\otimes 2}$ factor through the Kostant section. In particular, any sections of $(E_L \times^L Sym_0^2(W)) \otimes \mathcal{L}^{\otimes 2}$ in this case are regular and they contribute to the average number $1$.
\end{proposition}
\begin{proof}
First of all, by the hypothesis and if we suppose that $deg (\mathcal{L}) >> 0$, then any global sections in $H^0(C,(E_L \times^L Sym_0^2(W)) \otimes \mathcal{L}^{\otimes 2})$ can be seen in matrix form whose coefficients are sections of line bundles. Since $H^0(C, \mathcal{O}_C) = k^*$, it is easy to see that all of entries in the lower triangular part equals zero except the first diagonal below the main diagonal. Now we will show that this kind of matrix is conjugate to the Kostant section by an element in $G(k)$. \\
Let we denote $A$ one of our considering matrix, then firstly we will try to transform all of entries in the first diagonal below the main diagonal to $1$. That can be done as follows: 
\footnotesize
\[ diag(a_1,\dots, a_n,1,a_n^{-1}, \dots,a_1^{-1}).
\left( \begin{array}{cccccc}
* & * & \cdots & * & * & *  \\
x_1 & * & \cdots & * & * & *  \\
0 & x_2 & \cdots & * & * & * \\
\vdots & \vdots & \ddots & \vdots & \vdots & \vdots  \\
0 & 0 & \cdots & x_2 & * & *  \\
0 & 0 & \cdots & 0 & x_1 & * 
 \end{array} \right).diag(a_1^{-1},\dots, a_n^{-1},1,a_n,\dots,a_1)\]
 \[=\left( \begin{array}{cccccc}
* & * & \cdots & * & * & *  \\
a_1^{-1}a_2x_1 & * & \cdots & * & * & *  \\
0 & a_2^{-1}a_3x_2 & \cdots & * & * & * \\
\vdots & \vdots & \ddots & \vdots & \vdots & \vdots  \\
0 & 0 & \cdots & a_2^{-1}a_3x_2 & * & *  \\
0 & 0 & \cdots & 0 & a_1^{-1}a_2x_1 & * 
 \end{array} \right)\]  \normalsize
 Now we can finish this step by taking $a_n :=x_n ; a_{n-1}:=x_nx_{n-1}; \dots; a_1 :=x_nx_{n-1}\dots x_1.$ \\
 In the second step we will try to transform the main diagonal into the zero diagonal. We consider the following matrix in $G$:
 \[ C=\left( \begin{array}{cccccc}
1 & a_1 & a_2 & \cdots & a_{2n-1} & b  \\
0 & 1 & 0 & \cdots & 0 & -a_{2n-1}   \\
0 & 0 & 1 & \cdots & 0 & -a_{2n-2} \\
\vdots & \vdots & \vdots & \ddots &  \vdots & \vdots  \\
0 & 0 & 0 & \cdots  & 1 & -a_1  \\
0 & 0 & 0 & \cdots  & 0 & 1 
 \end{array} \right) \]
 where $b=-\dfrac{1}{2}\sum_{i=1}^{2n-1}a_ia_{2n-i}$. Then we can choose $a_i$ such that 
\[ C . \left( \begin{array}{cccccc}
* & * & \cdots & * & * & *  \\
1 & * & \cdots & * & * & *  \\
0 & 1 & \cdots & * & * & * \\
\vdots & \vdots & \ddots & \vdots & \vdots & \vdots  \\
0 & 0 & \cdots & 1 & * & *  \\
0 & 0 & \cdots & 0 & 1 & * 
 \end{array} \right) . C^{-1} = \left( \begin{array}{cccccc}
0 & 0 & \cdots & 0 & * & *  \\
1 & * & \cdots & * & * & *  \\
0 & 1 & \cdots & * & * & 0 \\
\vdots & \vdots & \ddots & \vdots & \vdots & \vdots  \\
0 & 0 & \cdots & 1 & * & 0  \\
0 & 0 & \cdots & 0 & 1 & 0 
 \end{array} \right) = E. \]
 
Similarly, by considering a matrix of the form
\[ D=\left( \begin{array}{cccccccc}
1 & 0 & 0 & 0 & \cdots & 0 & 0 & 0 \\
0 & 1 & a_1 & a_2 & \cdots & a_{2n-3} & b & 0 \\
0 & 0 & 1 & 0 & \cdots & 0 & -a_{2n-3} & 0  \\
0 & 0 & 0 & 1 & \cdots & 0 & -a_{2n-4} & 0 \\
\vdots & \vdots & \vdots & \vdots & \ddots &  \vdots & \vdots & \vdots \\
0 & 0 & 0 & 0 & \cdots  & 1 & -a_1 & 0  \\
0 & 0 & 0 & 0 & \cdots  & 0 & 1 & 0 \\
0 & 0 & 0 & 0 & \cdots  & 0 & 0 & 1
 \end{array} \right) \]
 we can choose $a_i$ such that $D.E.D^{-1}$ has the desirable second row. By doing the same way, after $n$ steps we will obtain a matrix that belongs to our Kostant section.
\end{proof}
\textbf{Case 3:} $d-g-1\leq \mu_1 \leq d$. In this case, the contribution will be zero since
\begin{eqnarray*}
&&2d_1+h^0(W_0\otimes X_1^*\otimes \mathcal{L}^{\otimes 2})+ h^0(Sym^2(X_1^*)\otimes \mathcal{L}^{\otimes 2})+\sum_{i=2}^t \bigg( h^0(X_i\otimes X_1^*\otimes \mathcal{L}^{\otimes 2}) \\ && \hspace{10cm} +h^0(X_i^*\otimes X_1^*\otimes \mathcal{L}^{\otimes 2})\bigg) \\
&\leq & h + (1-r_1)d_1 - r_0d_1 +2d \bigg(r_2r_1+\dots + r_tr_1 +r_0r_1+r_tr_1+\dots + r_2r_1+ \\ && \hspace{11cm}+\binom{r_1+1}{2}\bigg)  \\
&& \text{(where $h$ is a constant independent to $d$)}\\
&\leq & h-r_0d_1 +2d \bigg(r_2r_1+\dots + r_tr_1 +r_0r_1+r_tr_1+\dots + r_2r_1+\binom{r_1+1}{2}\bigg). 
\end{eqnarray*}
Hence
\begin{eqnarray*}
& \lim_{d\rightarrow \infty} \sum_{d-g+1 \leq \mu_1 \leq  d} \int_{Bun_{\mu_1,r_1}^{ss}(\mathbb{F}_q)} \frac{h_1}{h_2}(X_1,X_2,\dots,X_t, W_0) dX_1  \\ \leq  & \lim_{d\rightarrow \infty} \sum_{(d-g+1)r_1 \leq d_1\leq dr_1} q^{h - d_1} = 0.
\end{eqnarray*}
\textbf{Case 4:} $ 0< \mu_1 < d-g-1$ or $E$ is semi-stable. In this case we will obtain the constant $2$. In fact, by applying Riemann-Roch theorem to every summands of the vector bundle $V(E, \mathcal{L})$, we see that when $d=deg(\mathcal{L})$ is large enough then
\begin{equation*}
h^0(C, V(E, \mathcal{L})) = (4n^2+6n)d + (2n^2+3n)(1-g).
\end{equation*} 
Now for $G=SO(W)$ a split group over $\mathbb{F}_q$ and $K=\mathbb{F}_q(C)$, we write $\mu_G$ for the right invariant Haar measure on $G(\mathbb{A}_K)$ which gives measure 1 to the open compact subgroup $G(\hat{\mathcal{O}}_K) \subset G(\mathbb{A}_K).$ Since $G$ is semisimple, there exists the Tamagawa measure $\tau_G$ on $G(\mathbb{A}_K)$ and it can be defined by
\begin{equation}
\tau_G=  q^{(2n^2+n)(1-g)}. \bigg[ \prod_{x\in C} \int_{G(\mathcal{O}_{K_v})}|\omega_G|_x \bigg] \mu_G,
\end{equation}
where $\omega_G$ is a non-vanishing invariant differential form of top degree on $G$. We can see that $\tau_G$ is independent to the choice of $\omega_G$ (by Product formula). On the other hand, since $G$ is split, we can choose $\omega_G$ satisfying $[\kappa(x): \mathbb{F}_q].v_x(\omega) =1$ for all $x \in C$. As a result, we obtain the following equality:
\begin{equation}
\mu_{x,\omega}(G(\mathcal{O}_{K_v})):= \int_{G(\mathcal{O}_{K_v})}|\omega_G|_x \mu_G = \frac{|G(\kappa(x))|}{|\kappa(x)|^{dim(G)}}.
\end{equation}
Note that the quotient $G(K) \backslash G(\mathbb{A})$ carries a right action of the compact group $G(\hat{\mathcal{O}}_K)$. We may therefore write $G(K) \backslash G(\mathbb{A})$ as a union of orbits, indexed by the collection of double cosets
\begin{equation*}
G(K) \backslash G(\mathbb{A}/G(\hat{\mathcal{O}}_K)).
\end{equation*}
Applying the formulas (1) and (2), we calculate
\begin{eqnarray*}
\tau_G(G(K) \backslash G(\mathbb{A})) &=& \sum_{\gamma} \frac{\tau(G(\hat{\mathcal{O}}_K))}{|G(\hat{\mathcal{O}}_K) \cap \gamma^{-1}G(K)\gamma|} \\
&=& q^{(2n^2+n)(1-g)}(\prod_{x\in C} (\frac{|G(\kappa(x))|}{|\kappa(x)|^{dim(G)}}) \sum_{\gamma} \frac{1}{|G(\hat{\mathcal{O}}_K) \cap \gamma^{-1}G(K)\gamma|} \\
&=& q^{(2n^2+n)(1-g)}(\prod_{x\in C} (\frac{|G(\kappa(x))|}{|\kappa(x)|^{dim(G)}}) \sum_{\mathcal{P} \in Bun_G(C)(\mathbb{F}_q} \frac{1}{|Aut_G(\mathcal{P})|}.
\end{eqnarray*}
Now if we consider the counting measure weighted by the size of automorphism group on $Bun_G(C)$ then by using the fact that the Tamagawa number $\tau(G):= \tau_G(G(K) \backslash G(\mathbb{A})) $ of $G$ is equal to $2$, we have that
\begin{equation*}
|Bun_G(C)(\mathbb{F}_q)| = 2. q^{(2n^2+n)(1-g)}. \zeta_C(2)^{-1}.\zeta_C(4)^{-1}\cdots \zeta_C(2n)^{-1}.
\end{equation*} 
Come back to the case $ 0< \mu_1 < d-g-1$ or $E$ is semi-stable ($\mu_1 =0$), the average number can be computed as follows: 
\begin{eqnarray*}
&&\lim_{d \rightarrow \infty} \dfrac{\mathlarger{\int}_{Bun_G^{\mu_1<d-g-1}(\mathbb{F}_q)}|\mathcal{M}_{L,E}(k)|dE}{\mathcal{A}_L(k)} \\
&=& \lim_{d \rightarrow \infty} \dfrac{\mathlarger{\int}_{Bun_G^{\mu_1<d-g-1}(\mathbb{F}_q)}|H^0(C, V(E, L)^{reg}|dE}{\mathlarger{\prod}_{i=2}^{2n+1}|H^0(C, L^{\otimes 2i})|} \\
&=& \lim_{d \rightarrow \infty} \dfrac{|H^0(C, V(E,L))|\mathlarger{\int}_{Bun_G^{\mu_1<d-g-1}(\mathbb{F}_q)}\dfrac{|H^0(C, V(E, L)^{reg}|}{|H^0(C, V(E,L))|}dE}{q^{2n(2n+3)d+2n(1-g)}} \\
&=&  \lim_{d \rightarrow \infty} \dfrac{q^{(4n^2+6n)d+(2n^2+3n)(1-g)}\mathlarger{\int}_{Bun_G^{\mu_1<d-g-1}(\mathbb{F}_q)}(\prod_{i=1}^{2n}\zeta_C(2i)^{-1})dE}{q^{2n(2n+3)d+2n(1-g)}} \\
&=& \lim_{d \rightarrow \infty} q^{(2n^2+n)(1-g)}\mathlarger{\int}_{Bun_G^{\mu_1<d-g-1}(\mathbb{F}_q)}(\prod_{i=1}^{2n}\zeta_C(2i)^{-1})dE \\
&=& \lim_{d \rightarrow \infty} q^{(2n^2+n)(1-g)}|Bun_G^{\mu_1<d-g-1}(\mathbb{F}_q)|\prod_{i=1}^{2n}(\zeta_C(2i)^{-1}) \\
&=&q^{(2n^2+n)(1-g)}|Bun_G(\mathbb{F}_q)|\prod_{i=1}^{2n}(\zeta_C(2i)^{-1}) \\
&=&2
\end{eqnarray*}
\subsubsection{Average size of 2-Selmer groups}
Our main theorem can be proved as follows:
\begin{theorem}
Suppose that $q> 4^{2n+1}$, then 
\begin{equation*}
\limsup_{deg(\mathcal{L}) \rightarrow \infty}\dfrac{\mathlarger{\sum}\limits_{\substack{\alpha \in [S/\mathbb{G}_m](C) \\ \mathcal{L}(H_\alpha) \cong \mathcal{L}}} |Sel_2(E_\alpha)|}{\mathlarger{\sum}\limits_{\substack{\alpha \in [S/\mathbb{G}_m](C) \\ \mathcal{L}(H_\alpha)\cong \mathcal{L}}} 1} \leq 3 + f(q)
\end{equation*}
where $\lim_{q \rightarrow \infty} f(q) = 0$.
\label{main theorem 1}
\end{theorem}
\begin{proof}
The theorem is the corollary of what we have done so far. In fact
\begin{eqnarray*}
&&\limsup_{deg(\mathcal{L}) \rightarrow \infty}\dfrac{\mathlarger{\sum}\limits_{\substack{\alpha \in [S/\mathbb{G}_m](C) \\ \mathcal{L}(H_\alpha) \cong \mathcal{L}}} |Sel_2(E_\alpha)|}{\mathlarger{\sum}\limits_{\substack{\alpha \in [S/\mathbb{G}_m](C) \\ \mathcal{L}(H_\alpha)\cong \mathcal{L}}} 1} \\
&=& \limsup_{deg(\mathcal{L}) \rightarrow \infty}\dfrac{\mathlarger{\sum}\limits_{\substack{\alpha \in [S/\mathbb{G}_m](C) \\ \mathcal{L}(H_\alpha) \cong \mathcal{L} \\ E_{\alpha[2](K) = \{ 0\}}}} |Sel_2(E_\alpha)|+\mathlarger{\sum}\limits_{\substack{\alpha \in [S/\mathbb{G}_m](C) \\ \mathcal{L}(H_\alpha) \cong \mathcal{L} \\ E_{\alpha[2](K) \neq \{ 0\}}}} |Sel_2(E_\alpha)|}{H^0(C, \mathcal{L}^4 \oplus \mathcal{L}^6 \oplus \cdots \oplus \mathcal{L}^{4n+2})} \\
&\leq & \limsup_{deg(\mathcal{L}) \rightarrow \infty}\dfrac{|\mathcal{M}_{\mathcal{L}}(k)|+\frac{2^{2n-1}-1}{2^{2n-1}}\mathlarger{\sum}\limits_{\substack{\alpha \in [S/\mathbb{G}_m](C) \\ \mathcal{L}(H_\alpha) \cong \mathcal{L} \\ E_{\alpha[2](K) \neq \{ 0\}}}} |Sel_2(E_\alpha)|}{H^0(C, \mathcal{L}^4 \oplus \mathcal{L}^6 \oplus \cdots \oplus \mathcal{L}^{4n+2})}  \hspace{1cm} \text{(by Proposition \ref{compare1})}\\
&= & \limsup_{deg(\mathcal{L}) \rightarrow \infty}\dfrac{|\mathcal{M}_{\mathcal{L}}(k)|}{|\mathcal{A}_{\mathcal{L}}(k)|} \hspace{3cm} \text{(by Lemma \ref{large char contributes 0})}\\
& \leq & 3 + f(q)  \hspace{5cm}\text{(by section \ref{couting section})}
\end{eqnarray*}
\end{proof}
From the above theorem, the following corollary is immediate.
\begin{corollary}
Suppose that $q> 4^{2n+1}$, then 
\begin{equation*}
\limsup_{d \rightarrow \infty}\dfrac{\mathlarger{\sum}\limits_{\substack{\alpha \in [S/\mathbb{G}_m](C) \\ deg(\mathcal{L}(H_\alpha))\leq d \\ }} |Sel_2(E_\alpha)|}{\mathlarger{\sum}\limits_{\substack{\alpha \in [S/\mathbb{G}_m](C) \\ deg(\mathcal{L}(H_\alpha))\leq d \\ }} 1} \leq 3 +f(q),
\end{equation*}
where $\lim_{q \rightarrow \infty} f(q) = 0$.
\end{corollary}

Now if we compute the average over the range that $H_\alpha$ is regular, we will have a better bound for the average. Indeed, we are going to show that the bound in this case is $3$, and it is due to the fact that there is no regular universal family in the case 2 in section \ref{couting section}). To prove it we need two lemmas:
\begin{lemma}
Let $(A,m)$ be a regular Noetherian local ring and $f$ is an element of $m\backslash \{0\}$. Then $A/(f)$ is regular if and only if $f \notin m^2$.
\end{lemma}
\begin{proof}
See Corollary 2.12 in \cite{Liu06}
\end{proof}
\begin{lemma}
Suppose that $r_i=1$ for all $1\leq i \leq n+1$, $2d \geq \mu_i - \mu_{i+1} >d$ (note that $\mu_{n+1} = 0$), and there exists $i$ such that $2d > \mu_i - \mu_{i+1}$ (see case 2 in counting section), then for all $s \in H^0(C, V(E, \mathcal{L}))$ satisfying that the corresponding universal hyperelliptic curve $W_s$ is coincide with its minimal integral model, $W_s$ is not regular (in the meaning of Definition 3.3)  
\label{ignore case 2 of counting section}
\end{lemma}
\begin{proof}
With the above hypothesis, for $d$ large enough, any section $s \in H^0(C, V(E, \mathcal{L}))$ has the following self-adjoint matrix form: (see \ref{global sections of semistable})
\[ s=\left( \begin{array}{cccccc}
* & * & \cdots & * & * & *  \\
x_1 & * & \cdots & * & * & *  \\
0 & x_2 & \cdots & * & * & * \\
\vdots & \vdots & \ddots & \vdots & \vdots & \vdots  \\
0 & 0 & \cdots & x_2 & * & *  \\
0 & 0 & \cdots & 0 & x_1 & * 
 \end{array} \right), \]
 where $x_i \in H^0(C, (E_L \times^L X_i^*\otimes X_{i+1})\otimes \mathcal{L}^{\otimes 2})$ and here $X_{n+1}:= W_0$. Suppose that $\mu_l-\mu_{l+1} < 2d$, then $(E_L \times^L X_L^* \otimes X_{l+1})\otimes \mathcal{L}^{\otimes 2}$ is a line bundle of positive degree. Consequently, there exists a point $v \in C$ such that $x_l(v) = 0$. Suppose that $W_s$ is minimal, we will show that $W_{K_v}$ is not regular over $Spec(\mathcal{O}_{K_v})$. Now by decompose the determinant of $s-xI$ at the column containing $x_l$ we see that the characteristic polynomial of $s$ has the following form:
 \begin{equation*}
 f_s(x) = det(s-xI) = det(A-xI_l)^2.f_1(x)+ x_l.det(A-xI_l).f_2(x) + x_l^2.f_3(x),
 \end{equation*}
 where we may consider $f_i(x) \in \mathcal{O}_{K_v}[x]$, and 
 \[ A=\left( \begin{array}{cccccc}
* & * & \cdots & * & * & *  \\
x_1 & * & \cdots & * & * & *  \\
0 & x_2 & \cdots & * & * & * \\
\vdots & \vdots & \ddots & \vdots & \vdots & \vdots  \\
0 & 0 & \cdots & x_{l-2} & * & *  \\
0 & 0 & \cdots & 0 & x_{l-1} & * 
 \end{array} \right) \].
 Since $g(x) = det(A-xI_l) \in \mathcal{O}_{K_v}[x]$ has non-invertible image in $k_v[x]$, where $k_v:= \mathcal{O}_{K_v}\backslash \mathfrak{m}_v$, it is contained in a maximal ideal $\mathfrak{m}' \subset \mathcal{O}_{K_v}[x]$ and $\mathfrak{m}_v \subset \mathfrak{m}'$. Set $\mathfrak{m} =(\mathfrak{m}',y)$ a maximal ideal of $B=\mathcal{O}_{K_v}[x,y]$, it is easy to see that $B_{\mathfrak{m}}$ is regular. By applying Lemma 3.11 to the regular local ring $B_{\mathfrak{m}}$ and the non-zero element $y^2-f_s(x) \in \mathfrak{m}^2$, we imply that $Spec(B/(y^2-f_s(x))$ is not regular. Hence $W_{K_v}$ is not regular over $Spec(\mathcal{O}_{K_v})$.
\end{proof}
\begin{theorem}
Suppose that $q> 4^{n(2n+1)}$, then 
$$
\limsup_{d \rightarrow \infty}\dfrac{\mathlarger{\sum}\limits_{\substack{\alpha \in [S/\mathbb{G}_m](C) \\ deg(\mathcal{L}(H_\alpha)\leq d \\ H_\alpha \text{is minimal and regular}}} |Sel_2(E_\alpha)|}{\mathlarger{\sum}\limits_{\substack{\alpha \in [S/\mathbb{G}_m](C) \\ deg(\mathcal{L}(H_\alpha)\leq d \\ H_\alpha \text{is minimal and regular}}} 1} \leq 3 
$$
\end{theorem}
\begin{proof} The proof is almost identical to the proof of the theorem \ref{main theorem 1}, with the notices from Remark \ref{remark on transversal poly} (the regular and minimal conditions are only related to the characteristic polynomials, more precisely, they are conditions on $S$).
\end{proof}
Now we consider the family of transversal hyperelliptic curves. In this case, we also can ignore the case 2 in the counting section by using the similar argument as in lemma \ref{ignore case 2 of counting section}. Furthermore, by \ref{selmer equal torsor in transversal case}, we have an equality $|Sel_2(E_\alpha)|=|H^1(C, E_\alpha[2])|$ for any transversal $\alpha \in [S/\mathbb{G}_m](C)$. Thus, we have the following limit:
\begin{theorem} Suppose that $char(q) > 3$, then
$$
\lim_{d \rightarrow \infty}\dfrac{\mathlarger{\sum}\limits_{\substack{\alpha \in [S/\mathbb{G}_m](C) \\ deg(\mathcal{L}(H_\alpha)\leq d \\ H_\alpha \text{is transversal}}} |Sel_2(E_\alpha)|}{\mathlarger{\sum}\limits_{\substack{\alpha \in [S/\mathbb{G}_m](C) \\ deg(\mathcal{L}(H_\alpha)\leq d \\ H_\alpha \text{is transversal}}} 1} = 3 
$$
\end{theorem}
\begin{proof}
The computation is almost identification as the proof of theorem \ref{main theorem 1}. In fact, the only difference is in the case 4 of counting section \ref{couting section}. Denote $\mathcal{M}_{\mathcal{L}}^{trans}(k)$ to be the preimage of the transversal locus $\mathcal{A}_{\mathcal{L}}^{trans}(k)$ via the natural map $\mathcal{M} \rightarrow \mathcal{A}$. Then if $ 0< \mu_1 < d-g-1$ or $E$ is semi-stable ($\mu_1 =0$), the average number in the transversal case can be computed as follows: 
\begin{eqnarray*}
&&\lim_{d \rightarrow \infty} \dfrac{\mathlarger{\int}_{Bun_G^{\mu_1<d-g-1}(\mathbb{F}_q)}|\mathcal{M}_{L,E}^{trans}(k)|dE}{|\mathcal{A}_{\mathcal{L}}^{trans}(k)|} \\
&=& \lim_{d \rightarrow \infty} \dfrac{\mathlarger{\int}_{Bun_G^{\mu_1<d-g-1}(\mathbb{F}_q)}|H^0(C, V(E, L)^{reg})^{trans}|dE}{\mathlarger{\frac{|\mathcal{A}_{\mathcal{L}}^{trans}(k)|}{|\mathcal{A}_{\mathcal{L}}(k)|}\prod}_{i=2}^{2n+1}|H^0(C, L^{\otimes 2i})|} \\
&=& \lim_{d \rightarrow \infty} \dfrac{|H^0(C, V(E,L))|\mathlarger{\int}_{Bun_G^{\mu_1<d-g-1}(\mathbb{F}_q)}\dfrac{\frac{|H^0(C, V(E, L)^{reg})^{trans}|}{|H^0(C, V(E,L))|}}{\frac{|\mathcal{A}_{\mathcal{L}}^{trans}(k)|}{|\mathcal{A}_{\mathcal{L}}(k)|}}dE}{q^{2n(2n+3)d+2n(1-g)}} \\
&=&  \lim_{d \rightarrow \infty} \dfrac{q^{(4n^2+6n)d+(2n^2+3n)(1-g)}\mathlarger{\int}_{Bun_G^{\mu_1<d-g-1}(\mathbb{F}_q)}(\prod_{i=1}^{2n}\zeta_C(2i)^{-1})dE}{q^{2n(2n+3)d+2n(1-g)}} \hspace{1cm}\text{by Proposition \ref{transversal locus 1}}\\
&=& \lim_{d \rightarrow \infty} q^{(2n^2+n)(1-g)}\mathlarger{\int}_{Bun_G^{\mu_1<d-g-1}(\mathbb{F}_q)}(\prod_{i=1}^{2n}\zeta_C(2i)^{-1})dE  \hspace{1cm} \\
&=& \lim_{d \rightarrow \infty} q^{(2n^2+n)(1-g)}|Bun_G^{\mu_1<d-g-1}(\mathbb{F}_q)|\prod_{i=1}^{2n}(\zeta_C(2i)^{-1}) \\
&=&q^{(2n^2+n)(1-g)}|Bun_G(\mathbb{F}_q)|\prod_{i=1}^{2n}(\zeta_C(2i)^{-1}) \\
&=&2.
\end{eqnarray*}
\end{proof}
Since we actually have the limit in the transversal case, the density of the transversal locus help us to produce a lower bound for the average size:
\begin{corollary}
Suppose that $char(q) >3$, then 
\begin{equation*}
\liminf_{deg(\mathcal{L}) \rightarrow \infty}\dfrac{\mathlarger{\sum}\limits_{\substack{\alpha \in [S/\mathbb{G}_m](C) \\ \mathcal{L}(H_\alpha) \cong \mathcal{L}}} |Sel_2(E_\alpha)|}{\mathlarger{\sum}\limits_{\substack{\alpha \in [S/\mathbb{G}_m](C) \\ \mathcal{L}(H_\alpha)\cong \mathcal{L}}} 1} \geq 3 \prod_{v \in |C|}(1-\alpha_v),
\end{equation*}
where $\alpha_v= \frac{|\{x \in S(\mathcal{O}_{K_v}/(\varpi_v^2))| \Delta(x) \equiv 0 \hspace{0.2cm}mod (\varpi_v^2) \}|}{|k(v)^{4n}|}.$
\end{corollary}
\begin{proof}
We have that 
\begin{eqnarray}
& \liminf \limits_{deg(\mathcal{L}) \rightarrow \infty}\dfrac{\mathlarger{\sum}\limits_{\substack{\alpha \in [S/\mathbb{G}_m](C) \\ \mathcal{L}(H_\alpha) \cong \mathcal{L}}} |Sel_2(E_\alpha)|}{\mathlarger{\sum}\limits_{\substack{\alpha \in [S/\mathbb{G}_m](C) \\ \mathcal{L}(H_\alpha)\cong \mathcal{L}}} 1} \\
\geq & \liminf \limits_{deg(\mathcal{L}) \rightarrow \infty}\dfrac{\mathlarger{\sum}\limits_{\substack{\alpha \in [S/\mathbb{G}_m](C) \\ \mathcal{L}(H_\alpha) \cong \mathcal{L} \\ \text{$\alpha$ is transversal}}} |Sel_2(E_\alpha)|}{\mathlarger{\sum}\limits_{\substack{\alpha \in [S/\mathbb{G}_m](C) \\ \mathcal{L}(H_\alpha)\cong \mathcal{L} \\ \text{$\alpha$ is transversal}}} 1}. \frac{\mathlarger{\sum}\limits_{\substack{\alpha \in [S/\mathbb{G}_m](C) \\ \mathcal{L}(H_\alpha)\cong \mathcal{L} \\ \text{$\alpha$ is transversal}}} 1}{\mathlarger{\sum}\limits_{\substack{\alpha \in [S/\mathbb{G}_m](C) \\ \mathcal{L}(H_\alpha)\cong \mathcal{L} }} 1} \\
= & 3 \prod_{v \in |C|}(1-\alpha_v) \hspace{3cm} \text{By Proposition \ref{transversal locus 1}}
\end{eqnarray}
\end{proof}  
\section{Hyperelliptic curves with a marked Weierstrass point and a marked rational non-Weierstrass point}
In this section, we will consider the family of Hyperelliptic curves with a marked Weierstrass point and a marked non-Weierstrass point over a function field. By introducing the integral model of these hyperelliptic curves, we will able to define the height of hyperelliptic curves. The height will help us to order hyperelliptic curves and then we can state our main results where we take the average over this family of hyperelliptic curve. Furthermore, we can use the integral model to interpret each hyperelliptic curve as a $C-$point of a quotient stack.
\subsection{Weiertrass equation and height}
Given a smooth hyperelliptic curve $H$ of genus $m\geq 2$ over the function field $K$ of the smooth curve $C$, and assume that $H$ has a marked rational Weierstrass point $P_1$ and a marked rational non-Weierstrass point $P_2$. Without loss of generality, we may assume that under the natural map $H \rightarrow \mathbb{P}^1$, $P_1$ maps to $\infty \in \mathbb{P}^1(K)$, and $P_2$ maps to $0 \in \mathbb{P}^1(K)$. Therefore, we have an affine Weierstrass equation of $H$:
\begin{equation}
y^2=f(x)=x^{2m+1}+a_1x^{2m}+\dots + a_{2m}x+e^2,
\end{equation} \index{weierstrass}
where $a_i\in K$ and $e\in K^{\times}$ such that the discriminant of the polynomial $f(x):$ $\Delta(a_1,\dots, a_{2m}, e) \neq 0$. Denote the multi-set $(a_1,\dots, a_{2m}, e)$ by $\underline{a}$. Then $\underline{a}$ is unique up to the following identification: 
$$(a_1,\cdots,a_{2m},e) \equiv (\lambda^2a_1,\cdots, \lambda^{4m}a_{2m}, \lambda^{2m+1}e) \hspace{1cm} \lambda \in K^{\times}.$$
Now we define the minimal integral model of a given hyperelliptic curve $H$ as follows (c.f. \cite{Tho16}). First of all, we choose an affine Weierstrass equation of $H$ with $a_1,a_2,\dots, a_{2m},e \in K$ as above. Then for each point $v \in |C|$, we can choose an integer $n_v$ which is the smallest integer satisfying that: the tuple $$(\varpi_v^{2n_v}a_1, \varpi_v^{4n_v}a_2, \cdots, \varpi_v^{4mn_v}a_{2m}, \varpi_v^{2m+1}e)$$ has coordinates in $\mathcal{O}_{K_v}$. Given $(n_v)_{v \in |C|}$, we define the invertible sheaf $\mathcal{L}_H \subset K$ whose sections over a Zariski open $U \subset C$ are given by
$$\mathcal{L}_H(U) = K \cap \big(\prod_{v \in U}\varpi_v^{-n_v}\mathcal{O}_{K_v} \big).$$
Then it is easy to see that $a_i \in H^0(C, \mathcal{L}_H^{\otimes 2i})$ and $e \in H^0(C, \mathcal{L}_H^{\otimes 2m+1}).$ Furthermore, the stratum $(\mathcal{L}_H, \underline{a})$ is minimal in the sense that there is no proper subsheaf $\mathcal{M}$ of $\mathcal{L}_H$ such that $a_i \in H^0(C, \mathcal{M}^{\otimes 2i})$ and $e \in H^0(C, \mathcal{M}^{\otimes 2m+1}).$ Conversely, given a minimal strata $(\mathcal{L}, \underline{a})$ satisfying that $\Delta(\underline{a}) \neq 0$, we consider a subscheme of $\mathbb{P}^2(\mathcal{L}^{2m+1} \oplus \mathcal{L}^{2} \oplus \mathcal{O}_C)$ that is defined by 
$$Z^{2m-1}Y^2 = X^{2m+1} +a_1ZX^{2m} +\dots + a_{2m}Z^{2m}X +e^2Z^{2m+1}.$$ This is a flat family of curves $\mathcal{H} \rightarrow C$, and the generic fiber $H$ is a hyperelliptic curve over $K(C)$ with a marked rational Weierstrass point and a marked rational non-Weiertrass point. Furthermore, the associated minimal data of $H$ is exactly $(\mathcal{L}, \underline{a})$. Hence we have just shown the surjectivity of the following map $\phi_{\mathcal{L}}$ with a given line bundle $\mathcal{L}$ over $C$:
$$\phi_{\mathcal{L}}: \{ \text{minimal tuples} \hspace{0.2cm}(\mathcal{L}, \underline{a}) \} \rightarrow \{ \text{Hyperelliptic curves $(H,P_1,P_2)$ such that $\mathcal{L}_H \cong \mathcal{L}$} \}.$$
Moreover, the sizes of fibers of $\phi_{\mathcal{L}}$ can be calculated as follows
\begin{proposition}
Given a line bundle $\mathcal{L}$ over $C$, the map $\phi_{\mathcal{L}}$ defined as above is surjective, and the preimage of $(H,P_1,P_2)$ is of size $\frac{|\mathbb{F}_q^{\times}|}{|Aut(H,P_1,P_2)|}$, here $Aut(H,P_1,P_2)$ denotes the subset of all elements in $Aut(H)$ which preserve the marked points $P_1$ and $P_2$. 
\end{proposition}
\begin{proof}
Suppose that $(H,P_1,P_2)$ is a hyperelliptic curve with the associated minimal data $(\mathcal{L}, \underline{a}).$ The tuple of sections $\underline{a}$ is well-defined upto the following identification: 
$$\underline{a} \equiv \lambda. \underline{a} = (\lambda^2a_1, \dots, \lambda^{4m}a_{2m}, \lambda^{2m+1}e), \hspace{1cm} \lambda \in \mathbb{F}_q^{\times}.$$
In the other words, there is a transitive action of $\mathbb{F}_q$ on the fiber $\phi_{\mathcal{L}}^{-1}(H)$. Furthermore, the stabilizer of any element in $\phi_{\mathcal{L}}^{-1}(H)$ is exactly $Aut(H,P_1,P_2)$. Hence, the size of $\phi_{\mathcal{L}}^{-1}(H;P_1,P_2)$ is $\frac{|\mathbb{F}_q^{\times}|}{|Aut(H,P_1,P_2)|}$. 
\end{proof}

\begin{definition}{\textbf{(Height of hyperelliptic curve)}} The height of the hyperelliptic curve $(H,P_1,P_2)$ is defined to be the degree of the associated line bundle $\mathcal{L}_H$. 
\end{definition}
\begin{remark}
Given $d \in \mathbb{Z}$, there are finitely many isomorphism classes of hyperelliptic curves over $K$ whose height are less than $d$.
\end{remark}
Now we are able to state the main theorem of this chapter. Recall that the 2-Selmer group of a given hyperelliptic curve $H$ over $K(C)$ is the 2-Selmer group of the Jacobian $E$ of $H$, and by definition it is the kernel of $\beta \circ \alpha: H^1(K, E[2]) \rightarrow \prod_{v \in |C|}H^1(K_v, E)$, where $\alpha,$ and $\beta$ are natural maps in the following diagram:
\footnotesize\[
\begin{tikzcd}
0 \arrow{r} & E(K)/2E(K) \arrow{r} \arrow{d} & H^1(K, E[2]) \arrow{r} \arrow{d}{\alpha} & H^1(K, E)[2] \arrow{r} \arrow{d} & 0 \\
0 \arrow{r} & \prod_v E(K_v)/2E(K_v) \arrow{r} & \prod_v H^1(K_v, E[2]) \arrow{r}{\beta} & \prod_v H^1(K_v, E)[2] \arrow{r} & 0
 \end{tikzcd}
\] \normalsize
\begin{theorem}
Assume that $q > 16^{\frac{m^2(2m+1)}{2m-1}}$, and $p=$char($\mathbb{F}_q$) $> 3$, then
\begin{eqnarray*}
& \limsup\limits_{deg(\mathcal{L}) \rightarrow \infty} \frac{\sum \limits_{\substack{\text{H is hyperelliptic} \\ \mathcal{L}(H) \cong \mathcal{L}}} \frac{|Sel_2(H)|}{|Aut(H)|}}{\sum \limits_{\substack{\text{H is hyperelliptic} \\ \mathcal{L}(H) \cong \mathcal{L}}} \frac{1}{|Aut(H)|}} \\ &\leq 4. \zeta_C((2m+1)^2). \prod_{v \in |C|}\big(1+c_{2m-1}|k(v)|^{-2}+\dots+c_1|k(v)|^{-2m}-2|k(v)|^{(2m+1)^2}\big) \\ & \hspace{10cm} +2 + f(q),
\end{eqnarray*} where $\lim_{q \rightarrow \infty} f(q) =0$, and $c_i$ are constants which only depend on $m$ and $p$.
\end{theorem}
By Proposition 3.1, the above theorem is equivalent to 
\begin{theorem}
With the same hypothesis as in the previous theorem, we have that
\begin{eqnarray*}
&\limsup\limits_{deg(\mathcal{L}) \rightarrow \infty} \frac{\sum \limits_{\substack{\text{$(\mathcal{L}, \underline{a})$ is minimal} \\ \Delta(\underline{a}) \neq 0}} |Sel_2(H_{\underline{a}})|}{\sum \limits_{\substack{\text{$(\mathcal{L}, \underline{a})$ is minimal} \\ \Delta(\underline{a}) \neq 0}} 1} \leq \\ & 4. \zeta_C((2m+1)^2). \prod_{v \in |C|}\big(1+c_{2m-1}|k(v)|^{-2}+\dots+c_1|k(v)|^{-2m}-2|k(v)|^{(2m+1)^2}\big) + \\ & \hspace{10cm}2 + f(q),
\end{eqnarray*}
where $H_{\underline{a}}$ is the hyperelliptic curve that is corresponding to the tuple of sections $\underline{a}$, $\lim_{q \rightarrow \infty} f(q) =0$, and $c_i$ are constants which only depend on $m$ and $p$.   
\end{theorem}
If we order the set of hyperelliptic curves over $K$ by height, the following corollary of the above theorem give an upper bound for the average size of 2-Selmer groups:
\begin{corollary}
Assume that $q>4^{2m+1}$, and char($\mathbb{F}_q$) is "good", then
$$\limsup_{d \rightarrow \infty} \frac{\sum \limits_{\substack{\text{$(\mathcal{L}, \underline{a})$ is minimal} \\ \Delta(\underline{a}) \neq 0; deg(\mathcal{L} \leq d}} |Sel_2(H_{\underline{a}})|}{\sum \limits_{\substack{\text{$(\mathcal{L}, \underline{a})$ is minimal} \\ \Delta(\underline{a}) \neq 0; deg(\mathcal{L} \leq d}} 1}$$ $$\leq 4. \zeta_C((2m+1)^2). \prod_{v \in |C|}\big(1+c_{2m-1}|k(v)|^{-2}+\dots+c_1|k(v)|^{-2m}-2|k(v)|^{(2m+1)^2}\big)+ 2 + f(q),$$where $\lim_{q \rightarrow \infty} f(q) =0$, and $c_i$ are constants which are only depended on $m$ and $p$.
\end{corollary}
The above error teem $f(q)$ can be removed and the limsup becomes the normal limit if we take the average over the set of transversal hyperelliptic curves. The transversality can be determined as follows:
\begin{definition} Let $H$ be a hyperelliptic curve over $K$ with an associated minimal data $(\mathcal{L}, \underline{a})$. Then $H$ is called to be transversal if the discriminant $\Delta(\underline{a}) \in H^0(C, \mathcal{L}^{4m(2m+1)})$ is square-free. 
\end{definition}
\begin{theorem}
If char($\mathbb{F}_q$) is "good", then 
$$\lim_{d \rightarrow \infty} \frac{\sum \limits_{\substack{\text{$(\mathcal{L}, \underline{a})$ is transversal} \\ deg(\mathcal{L} \leq d}} |Sel_2(H_{\underline{a}})|}{\sum \limits_{\substack{\text{$(\mathcal{L}, \underline{a})$ is transversal} \\  deg(\mathcal{L} \leq d}} 1} = 6.$$
\end{theorem}
Over $S$, the universal curve $H_S$ is defined to be the subscheme of $\mathbb{P}^3(S)$:
$$Z^{2m-1}Y^2= X^{2m+1}+a_1ZX^{2m}+\dots + a_{2m}Z^{2m}X + e^2Z^{2m+1},$$where $\underline{a}=(a_i,e) \in S$.
This is a flat family of integral projective curves over $S$, hence we can define the relative Jacobian $J_S=Pic^0_{H_S/S}.$ The next section will provide a close relation between $BJ_S[2]$ and 2-Selmer groups. Consequently, we will be able to restate our main theorems in the stack language. 
\subsection{2-torsion group and 2-Selmer group}
This section is almost identical to the section \ref{2-selmer group and the first coho} in chapter 1. We will state the main results and then give sketchy proofs if required. 

Given a hyperelliptic curve $(H,P_1,P_2)$ over the function field $K(C)$, let denote $\mathcal{H} \rightarrow C$ be the minimal integral model of $H$. We also have the relative generalized Jacobian $\mathcal{J}$ of $\mathcal{H}$ whose generic fiber is the Jacobian $J$ of $H$. Recall that the set of isomorphism classes of $\mathcal{J}[2]-$torsors over $C$ can be identified with the \'{e}tale cohomology group $H^1(C, \mathcal{J}[2])$. By restriction to the generic fiber of $C$, we obtain a homomorphism
\begin{equation}
H^1(C,\mathcal{J}[2]) \rightarrow H^1(K, J[2]).
\end{equation} 
We obtain the following results:
\begin{proposition}
The homomorphism $(14)$ factors through the 2-Selmer group $Sel_2(J).$ 
\end{proposition}
And now in the transversal case, $Sel_2(J)$ can be identified with $H^1(C, \mathcal{J}[2])$ via the above map.
\begin{proposition} If the hyperelliptic curve $H$ is transversal, then 
\begin{eqnarray*}
|Sel_2(J)| = |H^1(C, \mathcal{J}2])|
\end{eqnarray*}
\end{proposition}
\begin{proof}
C.f. Proposition \ref{selmer equal torsor in transversal case}
\end{proof}
In general case, the size of $Sel_2(J)$ and $H^1(C, \mathcal{J}[2])$ can be compared as follows:
\begin{proposition}
We have that
\begin{eqnarray*}
|Sel_2(J)| \leq |H^1(C, \mathcal{J}[2])|, & \text{when}\hspace{0.2cm} J[2](K) = 0, \\
|Sel_2(J)| \leq 2^{2m-1}|H^1(C, \mathcal{J}[2])|, & \text{otherwise.}
\end{eqnarray*}
\label{compare2}
\end{proposition}
To summary, in general, $|Sel_2(J)|$ is bounded by $|H^1(C, \mathcal{J}[2]|$ except the case our Jacobian $J$ has a 2-torsion $K-$rational point. However, if we make an assumption that the size of our base field $q$ is large enough, then the contribution of $Sel_2(J)$ in this case to the average is zero. More precisely, we have
\begin{lemma}
If $q > 16^{\frac{m^2(2m+1)}{2m-1}}$ then the contribution of the case $J[2](K) \neq 0$ to the average is zero. In the other words, we have the following limit: 
\begin{equation*}
\limsup_{deg(\mathcal{L}) \rightarrow \infty}\dfrac{\mathlarger{\sum}\limits_{\substack{\underline{a}\in H^0(C, \mathcal{L}^2 \oplus \cdots \oplus \mathcal{L}^{4m} \oplus \mathcal{L}^{2m+1}) \\ J_{\underline{a}}[2](K) \neq \{ 0\}}} |H^1(C,\mathcal{J}_{\underline{a}}[2])|}{\mathlarger{\sum}\limits_{\substack{\underline{a}\in H^0(C, \mathcal{L}^2 \oplus \cdots \oplus \mathcal{L}^{4m} \oplus \mathcal{L}^{2m+1})}} 1} = 0
\end{equation*}
\end{lemma}
\begin{proof}
Let $H$ be the hyperelliptic curve over $C$ defined by $(\mathcal{L}, \underline{a})$, then the smooth locus $C'$ of the map $H \rightarrow C$ is determined by the condition $\Delta(\underline{a}) \neq 0$, notice that $\Delta \in H^0(C,\mathcal{L}^{4n(2n+1)})$. Denote $\mathcal{J}$ the corresponding Jacobian of $H$, then by the smoothness of $H$ over $C'$, any $K_v-$points of $\mathcal{J}$ can be extended as  $C_v'-$points. By using the Selmer condition, we imply that any elements in the 2-Selmer group of $J$ can be lifted to $\mathcal{J}[2]-$torsors over $C'$. Consequently, we get
\begin{equation*}
|Sel_2(J)| \leq |H^1(C', \mathcal{J}[2])|.
\end{equation*}
When $\mathcal{J}[2](C) \neq 0$, there exists a section $c \in H^0(C, \mathcal{L}^{\otimes 2})$ such that the $(x,z)-$ polynomial defining $H$ has a factorization:
$$x^{2m+1}+a_1x^{2m}z + \cdots + a_{2m}xz^{2m}+e^2z^{2m+1} $$$$= (x-cz)(x^{2m}+b_1x^{2m-1}z+b_2x^{2m-2}z^2+\cdots + b_{2m}z^{2m}).$$
It means that $\underline{a}$ can be expressed in terms of $c$ and $\{b_j \}_{1\leq j \leq 2m}$, where $b_j \in H^0(C, \mathcal{L}^{2j}),$ for all $j$, and $-c.b_{2m}$ is a square of a section in $H^0(C, \mathcal{L}^{2m+1})$. If $d= deg(\mathcal{L})$ is large enough, then by using Riemann-Roch theorem, the number of multiple sets $\underline{a}$ will be bounded above by $q^{2d+(2+4+\cdots+4m-2+2m+1)d+(2m+1)(1-g)}=q^{(4m^2+3)d+(2m+1)(1-g)}$.\\
Now we consider the following interpretation for $\mathcal{J}[2]-$tosors: any $\mathcal{J}[2]-$tosors over $C'$ can be considered as tame \'{e}tale covers of $C'$ of degree $2^{2m}$. Hence there is a natural map:
\begin{equation*}
\phi : H^1(C',\mathcal{J}[2]) \rightarrow \{\text{tame \'{e}tale covers of $C'$ of degree $4^m$} \}.
\end{equation*}
The number of points where $H_\alpha$ is singular $|C-C'|$ is bounded by the degree of $\Delta(H_\alpha)$, so $|C-C'| \leq 4m(2m+1)d$. As a consequence, the number of topological generators of $\pi_1^{tame}(C')$ is less than $2g+4m(2m+1)d$. The size of $H^1(C',\mathcal{J}[2])$ can be estimated if we understand the fiber of $\phi$. Let $M$ is a degree $4^{m}$ \'{e}tale cover of $C'$, then giving $M$ the structure of $\mathcal{J}[2]-$torsor is equivalent to giving an action map:
\begin{equation*}
\psi: \mathcal{J}[2] \times_{C'} M \longrightarrow M
\end{equation*}
which is compatible with the structure maps to $C'$ and satisfies that the following natural map 
\begin{eqnarray*}
\mathcal{J}[2] \times_{C'} M & \longrightarrow & M \times_{C'} M \\
(g,m) & \mapsto & (g.m,m)
\end{eqnarray*}
is isomorphic. \\ 
 Since everything is proper and flat over $C'$, the map $\psi$ is totally defined by the generic map $\psi_K: (\mathcal{J}[2]\times_{C'} M)_K \rightarrow M_K$. As $K-$vector spaces, $dim(M_K) = 2^{2m}$ and $dim(\mathcal{J}[2]\times_{C'} M)_K = 2^{4m}$, hence the number of maps giving $M$ the structure of a $\mathcal{J}[2]-$tosors is bounded by $2^{6m}$, so is the fiber of $\phi$. \\
Now we obtain that the average in the case $\mathcal{J}[2](C) \neq 0$ is bounded by:
\begin{equation*}
\dfrac{q^{2^{6m}}.4^{m(2g+4m(2m+1)d)}.q^{(4m^2+3)d+(2m+1)(1-g)}}{q^{(2m+1)^2d+(2m+1)(1-g)}} = \dfrac{a.4^{m(4m(2m+1)d)}}{q^{(4m-2)d}},
\end{equation*}
where $a$ is a constant independent to $d$. This goes to zero as $d$ goes to infinity if $q^{4m-2} > 4^{4m^2(2m+1)}$, or equivalently $q > 16^{\frac{m^2(2m+1)}{2m-1}}$. The lemma is completed.
\end{proof}
From now on, we will assume that $q > 16^{\frac{m^2(2m+1)}{2m-1}}$ if we work on the general case, and there are no assumptions for the transversal case. Hence we always have that $|Sel_2(J_{\underline{a}})| \leq |H^1(C, \mathcal{J}[2]|$ for any tuples $\underline{a}$, and $|Sel_2(J_{\underline{a}})| = |H^1(C, \mathcal{J}[2]|$ if $\underline{a}$ is transversal. We now restate our main theorem in stack language as follows.
\begin{itemize}
\item[1.] Recall that $S=Spec(K[a_1,\dots,a_{2m},e]) \cong \mathbb{A}^{2m+1}.$ Then any tuple $(\mathcal{L}, \underline{a})$ can be seen as a $C-$point of the quotient stack $[S/\mathbb{G}_m]$, where the action of $\mathbb{G}_m$ on $S$ is given by $\lambda. (a_1,\dots,a_{2m},e) = (\lambda^2a_1,\dots,\lambda^{4m}a_{2m},\lambda^{2m+1}e).$ We set $\mathcal{A}= Hom(C, [S/\mathcal{G}_m])$, then $A(k)$ classifies tuples $(\mathcal{L}, \underline{a}).$
\item[2.] Since the universal Jacobian $J_S$ is a group scheme over $S$, there is a natural map of quotient stacks
$$[BJ_S[2]/\mathbb{G}_m] \xrightarrow{\psi} [S/\mathbb{G}_m].$$
Given a morphism $\alpha: C \rightarrow [S/\mathbb{G}_m]$, as in the step 1, we obtain a family of curve $H_\alpha \rightarrow C$. Denote $J_\alpha= \alpha^*J_S$, then $J_\alpha$ is exactly the relative Jacobian of $H_\alpha$ over $C$. An isomorphism class of $J_\alpha[2]-$torsor over $C$ can be seen as a morphism $ \beta: C \rightarrow [BJ_S[2]/\mathbb{G}_m]$ that fits in the following commutative diagram:
\[
\begin{tikzcd}
C \arrow{r}{\beta} \arrow{rd}{\alpha} & \text{$[BJ_S[2]/\mathbb{G}_m]$} \arrow{d}{\psi} \\
 & \text{$[S/\mathbb{G}_m]$}
 \end{tikzcd}
\]
Hence if we set $\mathcal{M}= Hom(C, [BJ_S[2]/\mathbb{G}_m])$ then we have a natural map $\mathcal{M} \rightarrow \mathcal{A}$ where the fiber $\mathcal{M}_\alpha$ over $\alpha \in \mathcal{A}(k)$ classifies isomorphism classes of $J_\alpha-$torsors over $C$. 
\item[3.] Notice that the natural map $\mathcal{M} \rightarrow \mathcal{A}$ is compatible with maps to $Hom(C, B\mathbb{G}_m)$. 
\end{itemize}
Our main theorem can be translated into
\begin{theorem}
Suppose that $q > 16^{\frac{m^2(2m+1)}{2m-1}}$. Then we have that
$$\limsup_{deg(\mathcal{L}) \rightarrow \infty} \frac{|\mathcal{M}_{\mathcal{L}}(k)|}{|\mathcal{A}_{\mathcal{L}}(k)|} \leq 4.\prod_{v \in |C|}\big(1+c_{2m-1}|k(v)|^{-2}+\dots+c_1|k(v)|^{-2m}\big) + 2 + f(q),$$where $lim_{q \rightarrow \infty} f(q) = 0$, and $c_i$ are constants which are only depended on $m$ and $p$. If $p>2m+1$ then $c_i$ is only depended on $m$.
\end{theorem}
Let $\mathcal{A}^{trans}(k)$ be the subset of transversal elements in $\mathcal{A}(k)$, and $\mathcal{M}^{trans}(k)$ be the preimage of $\mathcal{A}^{trans}(k)$ under the natural map $\mathcal{M} \rightarrow \mathcal{A}$. Then in transversal case, we have the following limit:
\begin{theorem}
$$\lim_{deg(\mathcal{L}) \rightarrow \infty} \frac{|\mathcal{M}^{trans}_{\mathcal{L}}(k)|}{|\mathcal{A}^{trans}_{\mathcal{L}}(k)|} = 6.$$
\end{theorem}
One of the main ingredients in the proof of the above theorems is the close relationship between 2-torsion subgroups of Jacobians of our considering hyperelliptic curves and the stabilizer group schemes of a representation of $SO(2m+1) \times SO(2m+1)$ that appears in the Vinberg theory of $\theta-$group. In the next sections, we are going to introduce that representation and then explain the mentioned connection to 2-torsion subgroups of Jacobians.

\subsection{Vinberg representation of $SO(V_1) \times SO(V_2)$}
Let $(V_1, <|>_1)$ and $(V_2, <|>_2)$ are split $(2m+1)-$dimensional orthogonal spaces of discriminant $1$ and $-1$ respectively. Then we can find a basis $\{f_1, f_2, \dots, f_{2m+1} \}$ of $V_1$ such that the Gram matrix of $<|>_1$ is 
\[ B=\left( \begin{array}{ccccc}
 &  &  &  & 1 \\
 &  &  & 1 &   \\
 &  & \dots &  &  \\
 & 1 &  &  &   \\
1 &  &  &  &  
 \end{array} \right). \]
Similarly, there exists a basis $\{f_1', f_2', \dots, f_{2m+1}' \}$ of $V_2$ such that the Gram matrix of $<|>_2$ is $-B$. Now we can define the special orthogonal groups $G_i$ that corresponds to $V_1$ and $V_2$ :
\begin{equation*}
G_i := SO(V_i) = \bigg\{ T \in GL(V_i)\hspace{0.2cm}| \hspace{0.2cm} T^{*}.T=I; det(T)=1 \bigg\}, 
\end{equation*}where $T^* \in GL(V)$ denotes the adjoint transformation of $T$ which is uniquely determined by the formula
$$\langle Tv,w\rangle_i = \langle v,T^*w \rangle_i.$$ 
Notice that the matrix of $T^*$ with respect to our standard basis (for both $V_1$ and $V_2$) can be obtained by taking the reflection about anti-diagonal of the matrix of $T$. Set $G=G_1 \times G_2$, $V= V_1 \oplus V_2$ and consider the following  representation of $G$ 
\begin{eqnarray*} W &=& \bigg\{\text{self-adjoint operators $T: V \rightarrow V$ with block diagonal zero} \bigg\}  \\
&=& \bigg\{T = \begin{pmatrix}
  0 & A \\
  -A^* & 0  
 \end{pmatrix};\hspace{0.3cm}A: V_2 \rightarrow V_1 ; \hspace{0.3cm-A^*:V_1 \rightarrow V_2}\bigg\} \\
 & \equiv & V_1 \otimes V_2,
\end{eqnarray*}where $G$ acts on $W$ by conjugation. For each element $T \in W,$ the corresponding characteristic polynomial is of the form
$$g_T(x) = f_T(x^2) = x^{2n}+a_1x^{2n-2}+\dots +a_{n-1}x^2+e^2,$$where $n=2m+1$, $e  =det(A)$, and $f_T(x) $ is the characteristic polynomial of $-A.A^*$. The functions $a_1,a_2,\dots, a_{n-1}, e$ are homogenous $G-$invariant functions on $W$of degree $2,4,\dots, 2n-2$ and $n$ respectively. It is well-known that the ring of $G-$invariant functions on $W$ is freely generated by them if our base field $K$ is of characteristic $0$. In general, we still have a $G-$equivariant map
$$\pi: W \longrightarrow \text{S:= Spec}\big(K[a_1,a_2,\dots,a_{n-1},e]\big),$$where the action of $G$ on $S$ is trivial. 

\subsubsection{Regular locus and two Kostant sections}
\begin{definition}(Kostant section) A Kostant section of $(W,G)$ is a linear subvariety $\kappa$ of $W$ for which the restriction of function $k[W]^G \rightarrow k[\kappa]$ is an isomorphism.
\end{definition}
From Vinberg theory in characteristic $0$, there are exactly two Kostant sections (up to conjugation) in our case. For positive characteristic, Paul Levi \cite{Lev08} made it available with the assumption that char($p$) is good. Notice that the number of Kostant sections (up to conjugation), by construction, equals to the number of $G(\overline{k})-$orbits of the nilpotent regular locus. In our case, we give the precise description of Kostant sections as follows: for each point $c=(a_1,a_2,\dots,a_{n-1},e) \in S$, we define an associated element $T_c$ in $W$: 
\begin{equation}
T_c = \begin{pmatrix}
  0 & A_c \\
  -A_c^* & 0  
 \end{pmatrix}
\end{equation}where 
\begin{equation}
A_c = \begin{pmatrix}
  b_{2m} & \cdots & b_{m+1} & e & 0 & 0 & \cdots & 0 \\
  0 & \cdots & 0 & 0 & 0 &  0 &\cdots & 1 \\
  \vdots & \vdots & 0 & 0 & 0 & 1 & \vdots & 0 \\
  b_m & \cdots & b_1 & 0 & 1 & 0 & \cdots & 0 \\
  0 & \cdots & 1 & 0 & 0 & 0 & \cdots & 0 \\
   \vdots & \vdots & 0 & 0 & 0 & 0 & \vdots & 0 \\
  1 & \cdots & 0 & 0 & 0 & 0 & \cdots & 0 \\  
 \end{pmatrix}
\end{equation}where $b_i=\dfrac{(-1)^{i-1}a_i}{2}.$ It could be checked that $\pi (T_c) = c$, thus we have already defined a section of the invariant map $\pi$:
\begin{align*}
\kappa_1 : \text{S} \longrightarrow W \\
c \longrightarrow T_c
\end{align*}
Similarly, we define
\begin{equation}
T_c' = \begin{pmatrix}
  0 & A_c' \\
  -{A'}_c^* & 0  
 \end{pmatrix}
\end{equation}where 
\begin{equation}
A_c' = \begin{pmatrix}
  b_{2m} & 0 &\cdots & b_{m} & 0  & 0 & \cdots & 1 \\
  \vdots & \vdots & \cdots & \vdots & \vdots &  \vdots &\ddots & \vdots \\
b_{m+2} & 0 & \cdots & b_2 & 0  &  1 &\cdots & 0 \\  
  b_{m+1} & 0 & \cdots & b_1 & 1  & 0 & \cdots & 0 \\
  e & 0 & \cdots & 0 & 0 & 0  & \cdots & 0 \\
  0 & 0 &\cdots & 1 & 0 & 0  & \cdots & 0 \\
   \vdots & \vdots &  \ddots & \vdots & \vdots & \vdots & \vdots & \vdots \\
  0 & 1 & \cdots & 0 & 0 & 0 & \cdots & 0 \\  
 \end{pmatrix}.
\end{equation}
Then we will obtain the second section of $\pi$:
\begin{align*}
\kappa_2 : \text{S} \longrightarrow W \\
c \longrightarrow T_c'
\end{align*}

Now let recall the definition of regularity:
\begin{definition} An element $T$ in $W(\overline{K})$ is called to be regular if its stabilizer $Stab_{G_{\overline{K}}}(T)$ is finite. The condition of being regular is open, and we write $W^{reg}$ for the open subscheme of regular elements of $W$.
\end{definition}

\begin{corollary}
Over an algebraic closed field $\bar{K}$, any element in $W^{reg}(\bar{K})$ is conjugate over $G(\bar{K})$ with an element in one of two Kostant sections $\kappa_i$.  
\end{corollary}
The following proposition gives us a necessary condition of being regular, it will be very helpful in the counting section. 
\begin{proposition}Let $T = \begin{pmatrix}
  0 & A \\
  -A^* & 0  
 \end{pmatrix} \in W$ be a regular element, then  product matrices $A.A^*$ and $A^*.A$ are regular.  
\end{proposition}
\begin{proof} Suppose that $A.A^*$ is not regular, then there exists a non-zero polynomial $h(x) \in K[x]$ of degree $t < n$
$$h(x)=x^t+b_1x^{t-1}+\dots + b_t$$satisfying $h(-A.A^*)=0.$ It is easy to see that $0=A^*.h(-A.A^*)= h(-A^*.A).A^*,$ thus
$$h(T^2).T = T = \begin{pmatrix}
  0_n & h(-A.A^*).A \\
  -h(-A^*A).A^{*} & 0_n  
 \end{pmatrix}=0_{2n}.$$ Since the degree of $h(x^2).x$ is less than $2n$, we imply that $T$ is not regular. The case that $A^*.A$ is regular is similar.
\end{proof}
Given a regular operator $T \in W,$ we consider two quotient rings $L=K[x]/(f(x)),$ and $M=K[x]/(g(x))\cong K[T],$ where $g(x)=f(x^2)$ is the characteristic polynomial of $T$. We have an embedding of $K-$algebras: $L \hookrightarrow M$ by $x \mapsto x^2.$ We can describe the stabilizer of $T$ under the action of $G$ as follow:
\begin{proposition}The stabilizer $Stab_G(T)(K)$ of a regular operator $T \in W$ whose characteristic  polynomial is $g(x)=f(x^2)$ is isomorphic to the kernel of the norm map $Res_{L/K}(\mu_2) \rightarrow \mu_2$, where $L=K[x]/(f(x)).$ In particular, the finite group scheme $Stab_G(T)$ has order $2^r$ over $K$, where $r+1$ is the number of distinct roots of $f(x)$ in the separable closure $K^s$.
\end{proposition} 
\begin{proof}
Any elements in the stabilizer of $T = \begin{pmatrix}
  0 & A \\
  -A^* & 0  
 \end{pmatrix}$ is of the form $\begin{pmatrix}
  B & 0 \\
  0 & C  
 \end{pmatrix}$, where $B,C \in SO(n)$ satisfying that $BAC^*=A$. By squaring $T$, we deduce that the submatrices $B$ and $C$ commute with the matrix $AA^*$ and $A^*A$ respectively. Without loss of generality, we may assume that our matrix $T$ lies in the first Kostant section. Thus the matrix $A.A^*$ is regular in $GL(V_1)$, and if we denote its characteristic is $f(x)$, we can identify $L$ with $K[A.A^*]$. As in chapter 1, we know that the stabilizer of $AA^*$ in $SO(V_1)$ can be identified with 
 $$\{h \in L=K[x]/(f(x)) | h^2=1, \, Nm_{L/K}(h)=1 \} \cong Ker \{Res_{L/K}(\mu_2) \xrightarrow{Nm} \mu_2 \}.$$Hence given $\begin{pmatrix}
  B & 0 \\
  0 & C  
 \end{pmatrix} \in Stab_G(T)$, there exist uniquely an element $h(x) \in L$ such that $B = h(AA^*)$. For any polynomial $P(x)\in K[x]$ and two square matrices $D$ and $E$, we can prove that $det(P(D.E))=det(P(E.D)).$ In fact, if we express $P(D.E)$ as $D.H+aI,$ for some square matrix $H$, then $P(E.D)=H.D +aI$. Combine with the well-known equality $det(I+B.C)=det(I+C.B)$, we have completed the proof of $det(P(D.E))=det(P(E.D)).$ Apply that observation to the case $D=A^*$ and $E=A$, we observe that $det(h(A^*A) =1$. On the other hand, since $f(x)$ is also the characteristic polynomial of $A^*A$, we deduce that $h(A^*A)^2 = I_n$ the identity matrix. We have just seen that $h(A^*A) \in SO(V_2)$. Since
 $$B.A.h(A^*A)^*= h(AA^*).A.h(A^*A) = h(AA^*).h(AA^*).A = A,$$ the matrix $\begin{pmatrix}
  h(AA^*) & 0 \\
  0 & h(A^*A)  
 \end{pmatrix}$ stabilizes $T$.  We will now prove that $C= h(A^*A)$. First of all, by setting $C= h(A^*A) + C_1$, the matrix $C_1$ needs to satisfy that $C_1. A^* =0$. If $A^*$ is invertible then $C_1=0$. Otherwise, by elementary computation, we can see that the entries of $C_1$: $C_{ij}=0$ for all $(i,j): i \neq m$. Since the determinant of $A$ is $0$, we have that:
 \begin{align*}
 & \hspace{1cm} A = \begin{pmatrix}
  b_{2m} & \cdots & b_{m+1} & 0 & 0 & 0 & \cdots & 0 \\
  0 & \cdots & 0 & 0 & 0 &  0 &\cdots & 1 \\
  \vdots & \vdots & 0 & 0 & 0 & 1 & \vdots & 0 \\
  b_m & \cdots & b_1 & 0 & 1 & 0 & \cdots & 0 \\
  0 & \cdots & 1 & 0 & 0 & 0 & \cdots & 0 \\
   \vdots & \vdots & 0 & 0 & 0 & 0 & \vdots & 0 \\
  1 & \cdots & 0 & 0 & 0 & 0 & \cdots & 0 \\  
 \end{pmatrix}, \\
 \end{align*}thus all entries in the middle row and column of the product $A^*A$ are zeros. This implies that the central entry of $C_1$ is uniquely determined by the condition $det(C_1+h(A^*A))=1$, and hence it must be equal to zero. By the above discussion, the central entry of $(C_1+C_1^*)h(A^*A)$ is $0$, and every entries except the central one of $C_1C_1^*$ are zeros. On the other hand, by decomposing the product $C.C^*$, the condition $C \in SO(V_2)$ is equivalent to $(C_1+C_1^*)h(A^*A)+C_1.C_1^*=0.$ This equality implies that $C_1.C_1^*=0$, thus $C_1 + C_1^*=0$, and hence $C_1=0$. We have just proven the claim that $C=h(A^*A)$. 
 
For the case $T \in \kappa_2$, similarly, by using the fact that $A^*.A$ is regular in $GL(V_2)$, we also can identify $Stb_G(T)$ with a subset of $L$ as above. And if $T$ belongs to both sections, we can just need to choose one of sections to start with. 

The proof of the Proposition is completed. 
\end{proof}
We also can compute the infinitesimal stabilizer as follows: the induced action of $\mathfrak{g}=Lie(G)$ on $W$ is: for any element $X = \begin{pmatrix}
X_1 &0 \\ 0 &X_2
\end{pmatrix} \in \mathfrak{g}=so(V_1) \times so(V_2)$ and $T \in W$, then $X*T=[X,T]=XT-TX.$ If $T$ is regular, let assume that $T \in \kappa_1$, then $X_1 = 0$ since $X_1=X_1^*$ that is a consequence of the fact that any matrix that commutes with $A.A^*$ is of the form $h(AA^*)$ for some polynomial $h(x) \in K[x]$. We deduce that $X_2.A^*$ must be the zero matrix. Using the trivial computation and the fact that $X_2=-X_2^*$, we can show that $X_2=0$. Thus, $Stab_{\mathfrak{g}}(T)$ is trivial. This implies that the action map
\begin{eqnarray*}
G \times_{S}W^{reg} & \longrightarrow & W^{reg}\times_{S} W^{reg} \\
(g,v)&\mapsto& (g.v,v)
\end{eqnarray*}
is \'{e}tale, and the universal stabilizer $I$ of the action of $G$ on $W^{reg}$
$$I= (G\times_{S}W^{reg}) \times_{W^{reg}\times_{S} W^{reg}}W^{reg},$$where $W^{reg} \rightarrow W^{reg}\times W^{reg}$ is the diagonal map, is a quasi-finite \'{etale} group scheme over $W^{reg}$ (base change of an \'{etale} map is \'{etale}).   

\begin{proposition} There exists a unique group scheme $I_{S}$ over $S$ equipped with a $G-$invariant isomorphism $\pi^*I_{S} \rightarrow I$ over $W^{reg}$. As a corollary, there is a $\mathbb{G}_m-$ equivariant isomorphism of stacks $[BI_{S}]\cong [W^{reg}/G],$ where $BI_{S}$ is the relative classifying stack of $I_{S}$ over $S$.
\label{stabilizer and orbits}
\end{proposition}
\begin{proof}
The action map $(H_{ad})^{\theta} \times \kappa_1 \rightarrow W^{reg}$ is \'{etale} and surjective, we set  $I_{S}= \kappa_1^*I$ and try to prove that two \'{etale} group schemes $I$ and $\pi^*k_1^*I$ are isomorphic over $W^{reg}.$ This is a consequence of decent data theory.
\end{proof}

\subsection{Stabilizer group scheme and Jacobian of hyperelliptic curve}
For each element $T \in W^{reg}$, let $f_T(x^2)$ be the characteristic polynomial of $T$, we consider the projective curve in $\mathbb{P}^3$ with the affine equation: $y^2=f_T(x)$. As a result, we obtain $H_{W^{reg}}$ a flat family of integral projective curves over $W^{reg}$. By the representability of the relative Picard functor, we obtain the scheme $Pic_{H_{W^{reg}}/W_{reg}}$ locally of finite type over $W^{reg}$, and also the relative Jacobian $J_{W^{reg}}=Pic^0_{H_{W^{reg}}/W_{reg}}$ over $W^{reg}$. Over $S$, recall that the universal curve $H_S$ is defined to be the subscheme of $\mathbb{P}^3(S)$:
$$Z^{2m-1}Y^2= X^{2m+1}+a_1ZX^{2m}+\dots + a_{2m}Z^{2m}X + e^2Z^{2m+1}.$$
This is a flat family of integral projective curves over $S$, hence we also can define the relative Jacobian $J_S=Pic^0_{H_S/S}.$ By definition, we obtain a canonical isomorphism
$$J_{W^{reg}} \rightarrow J_S \times_S W^{reg}.$$
Now we will see the connection between the 2-torsion subgroup $J_{W^{reg}}[2]$ and the stabilizer $I_{W^{reg}}$.
\begin{proposition}
There is a canonical isomorphism of \'{e}tale group schemes over $W^{reg}$:
$$I_{W^{reg}} \cong J_{W^{reg}}[2]$$
\label{can2}
\end{proposition}
\begin{proof}
Let $W^{reg}= W_1 \cup W_2$, where $W_i$ is the orbit of the Kostant section $\kappa_i$. Thus, it is enough to show that there is a canonical isomorphism $I_{W^{reg}}|_{W_i} \cong J_{W^{reg}}[2]|_{W_i}$. Let do it over $W_1$ and the case of $W_2$ will be similar.

Denote $B_1$ and $B_2$ the bilinear forms associated to the quadratic spaces $V_1$ and $V_2$. For each $T = \begin{pmatrix}
0 & A \\ -A^* & 0
\end{pmatrix} \in W_1,$ the matrix $T_1=-A.A^* \in GL(V_1)$ is regular. We define $B_{1,T_1}(v_1,w_1) = B_1(v_1, T_1w_1)$ for $v_1, w_1 \in V_1$. Then denote $Q_1$ and $Q_{1,T_1}$ the corresponding quadratic forms on $V_1$. Define $\mathcal{P}$ to be the pencil of quadrics on the projective space $\mathbb{P}(V_1)$ spanned by $Q_1$ and $Q_{1,T_1}$, and set $B$ be the base locus of $\mathcal{P}$. In \cite{Wan2}, X. Wang showed that both $I_T$ and $J_T[2]$ act on the Fano variety of $B$ whose points are projective $(n-1)-$planes contained in the smooth part of $B$. By varying $T$, we obtain that $I_{W_1}$ and $J_{W_1}[2]$ share a common principal homogeneous space. Furthermore, by using the fact that these two actions commute, we are able to construct a canonical isomorphism of \'{e}tale group schemes $I_{W_1}$ and $J_{W_1}[2]$.
\end{proof}
\begin{remark}
The previous isomorphism $I_{W^{reg}} \rightarrow J_{W^{reg}}[2]$ is $G-$equivariant by construction. Hence, it descends to an isomorphism of group schemes over $S$: $I_S \rightarrow J_S[2]$. By Proposition \ref{stabilizer and orbits}, we have a $\mathbb{G}_m-$equivariant isomorphism of quotient stacks
$$BJ_S[2] \cong [W^{reg}/G]$$
\end{remark}
Proposition 3.20 provides another interpretation of $\mathcal{M}(k)$: from the isomorphisms $I_S \cong J_S[2]$ and $BI_S \cong [W^{reg}/G]$, we deduce that 
$$\mathcal{M} \cong Hom(C, [W^{reg}/G\times \mathbb{G}_m].$$
Consequently, $\mathcal{M}_{\mathcal{L}}(k)$ classifies tuples $(\mathcal{E}, s)$ where $\mathcal{E}$ is a principal $G-$bundle and $s$ is a global section of the vector bundle $(W^{reg}\times^G \mathcal{E}) \otimes \mathcal{L}$. In the next section, we will try to estimate the size of $H^0(C,(W^{reg}\times^G \mathcal{E}) \otimes \mathcal{L})$ for a given $G-$bundle $\mathcal{E}$.

\subsection{Density of regular locus}
To prove our theorem, we need to estimate the number of global regular sections of some vector bundles. It is not easy to calculate it directly. Instead, we will firstly estimate the total number of global sections. Then the results in (\cite{HLN14} section 5) tell us that we will be able to estimate the number of global regular sections if we know the density of the regular locus $W^{reg}$ in $W$. The next subsection will help us to compute the local density.
\subsubsection{Orbits over finite fields via Galois cohomology}
The content of this section is based on the paper of Bhargava and Gross \cite{BG13} where they described rational orbits with a fixed invariant via Galois cohomology. We adopt their arguments in our case to estimate the number of rational orbits and then the size of regular locus over finite fields. 

Let $k^s$ denote a separable closure of the field $k$. If $M$ (respectively $J$) is a commutative finite \'{e}tale group scheme (a smooth algebraic group) over $k$, we denote $H^1(k,M)=H^1(Gal(k^s/k),M(k^s))$ ($H^1(k,J)=H^1(Gal(k^s/k),J(k^s))$ respectively) be the corresponding Galois cohomology group (pointed set of first cohomology classes). Notice that in case $k$ is finite field of odd order, every non-degenerate quadratic space of odd dimension is split, so $H^1(k,SO(V_i))=0$ for $i=1,2$. Thus, $H^1(k,G)=H^1(k,SO(V_1) \times SO(V_2))=0$.

For the rest of this section, we assume that $k$ is a finite field. Let $T \in W^{reg}(k)$ be a regular self-adjoint operator with the invariant $a=(a_1,\dots,a_2m,e) \in S(k)$, and let $G_T\subset G$ be the finite \'{e}tale subgroup stabilizing $T$. For any self-adjoint operator $L$ in $W(k)$ that is in the same orbit as $T$ over $k^s$, we have $L=gTg^{-1}$ for some $g \in G(k^s)$. This defines an element in $H^1(k,G_T)$ as follows: for any $\sigma \in Gal(k^s/k)$, the element $c_\sigma =g^{-1}g^{\sigma}$ lies in $G_T(k^s)$, and the map $\sigma \rightarrow c_\sigma$ defines a $1-$cocycle on the Galois group with values in $G_S(k^s)$, and hence defines an element in $H^1(k,G_T)$. It can be checked that the cohomology class of that $1-$cocycle depends only on the $G(k)$-orbit of $T$. Conversely, given a $1-$cocycle $c_\sigma$, then it has the form $g^{-1}g^{\sigma}$ since $H^1(k,G)=0$. So we obtain an associated operator $L=gTg^{-1}$ that is defined over $k$ since $\sigma(L)=L$ for all $\sigma \in Gal(k^s/k)$ by the definition of the cocycle $c_\sigma$. We have just proved the statement $i)$ of the following proposition: (c.f. \cite{BG13})
\begin{proposition}
\begin{itemize}
\item[i)]Given an operator $T \in W^{reg}(k)$, there is a bijection between the set of $G(k)-$orbits in $W^{reg}(k) \cap G(k^s).T$ and the  set $H^1(k, G_T)$. 
\item[ii)]For any $a=(a_1, \dots, a_{2m},e) \in S(k)$, the size of $W_a^{reg}(k)$ is bounded above by $2.|G(k)|$.
\end{itemize}
\end{proposition}
\begin{proof}
For $ii)$, firstly recall that the action of $G(k^s)$ on $W_a^{reg}(k^s)$ has at most two orbits. Hence there exist $T_1$ and $T_2$ (they could be the same) in $W^{reg}_a(k)$ such that $W^{reg}_a(k) \subset (G(k^s).T_1 \cup G(k^s).T_2)$. We will finish the proof by proving that the size of $W^{reg}_a(k) \cap G(k^s).T_1$ is equal to $|G(k)|$. In fact, by Proposition 3.19, if we set $f(x)= x^{2m+1} + a_1x^{2m} + \dots + a_{2m}x + e^2$ and denote $L= k[x]/(f(x))$, then $G_T(k)$ is isomorphic to the kernel of the norm map: $Res_{L/k}(\mu_2) \rightarrow \mu_2$. By $i)$ and Kummer theory, the number of  $G(k)-$orbits in $W^{reg}(k)$ equals to $|(L^*/L^{*2})_{N \equiv 1}| = |L^*[2]_{N=1}|=|G_T(k)|.$ Hence, we can finish the proof of $ii)$ by using the Orbit-Stabilizer theorem.
\end{proof}
\subsubsection{Regular locus in the transversal case}
Recall that in the chapter 1 we have computed the density of regular sections and also transversal regular sections by using the results of Poonen (see \cite{Poo03}). By looking back to our method there, we can see that it is essentially based on the fact that any regular vectors with the same invariant are conjugate over algebraically closed field. It is no longer true in our current situation where we have two Kostant sections. But if we restrict to the transversal part, we still have:
\begin{proposition}
Denote $k=\mathbb{F}_q$ a finite field and $\bar{k}$ its algebraically closure, let $f(x) \in k[x]$ satisfying that the order of its roots in $\bar{k}$ is at most $2$, and if $x$ divides $f(x)$ then $x^2 \nmid \, f(x)$. Then the action of $G(\bar{k})$ on $V_f^{reg}(\bar{k})$ is transitive.
\end{proposition} 
\begin{proof}
In \cite{Sha16}, the similar result for separable characteristic polynomial $f(x)$ (the regular semi-simple case) is given. Here we will try to generalize the result for $f(x)$ with some conditions which later on can be seen to be closely related to the transversal condition. Given two elements $S$ and $T$ in $V_f^{reg}(\bar{k})$, with out loss of generality, we may assume that $S\in \kappa_1$ and $T \in \kappa_2$, or precisely : $S=\kappa_1(f)= \begin{pmatrix}
0 & A_1 \\ -A_1^* & 0
\end{pmatrix} T= \kappa_2(f)=\begin{pmatrix}
0 & A_2 \\ -A_2^* & 0
\end{pmatrix}.$ Since $-A_1A_1^*$ is regular, $f(x)$ is also the minimal polynomial of $-A_1A_1^*$. Equivalently, for each root $\lambda_i$ of $f(x)$ of order $n_i$, the vector space of generalized eigenvectors of $-A_1A_1^*$ corresponding to $\lambda_i$ has dimension $n_i$. By the hypothesis of $f(x)$, we have three cases of roots as follow:
\begin{itemize}
\item[Case 1:] If $\lambda \neq 0$ is a single root of $f(x)$, then $\pm \sqrt{\lambda}$ are single roots of $f(x^2)$ the characteristic polynomial of $S$. If $v_{\sqrt{\lambda}}$ is the unique (up to scalar) non-zero $\sqrt{\lambda}-$eigenvector of $S$, then it will have the form
$$v_{\sqrt{\lambda}}= \begin{pmatrix}
v_\lambda \\ \frac{-1}{\sqrt{\lambda}} A_1^*v_\lambda
\end{pmatrix},$$where $v_\lambda$ is the unique $\lambda-$eigenvector of $-A_1A_1^*$. Similarly, for $(-\sqrt{\lambda})$, we can choose an eigenvector as follows:
$$v_{-\sqrt{\lambda}}= \begin{pmatrix}
v_\lambda \\ \frac{1}{\sqrt{\lambda}} A_1^*v_\lambda
\end{pmatrix}.$$
\item[Case 2:] If $\lambda=0$ is a single root of $f(x)$, then (up to scalar) we denote $v_0$ and $v_0^*$ to be the unique non-zero $0-$eigenvector of $A_1$ and $A_1^*$ respectively. In that case, the basic of the $2-$dimensional vector space of $0-$eigenvectors of $S$ is
$$\Bigg\{\begin{pmatrix}
0 \\ v_0
\end{pmatrix};\begin{pmatrix}
v_0^* \\ 0
\end{pmatrix}\Bigg\}$$ 
\item[Case 3:] If $\lambda \neq 0$ is a double root of $f(x)$ and the eigenspace $V_{1,\lambda}$ of $-A_1A_1^*$ corresponding to $\lambda$ has dimension $2$. Then we can choose an orthogonal basic $\{v_1,v_2\}$ of $J_\lambda$ with respect to the quadratic form $(V_1 ,Q_1)$. Then the product in $(V,Q)$:
$$<(v_1,A^*v_1);(v_2,A^*v_2)> = <v_1,v_2> + <A^*v_1,A^*v_2> = 0 + <AA^*v_1,v_2> = 0.$$
This helps us to define an orthogonal basic of $V_{\sqrt{\lambda}} \oplus V_{\sqrt{\lambda}} \subset  V$: 
$$\Bigg\{\begin{pmatrix}
v_1 \\ \frac{1}{\sqrt{\lambda}}A_1^*v_1
\end{pmatrix};\begin{pmatrix}
v_2 \\ \frac{1}{\sqrt{\lambda}}A_1^*v_2
\end{pmatrix};\begin{pmatrix}
v_1 \\ \frac{-1}{\sqrt{\lambda}}A_1^*v_1
\end{pmatrix};\begin{pmatrix}
v_2 \\ \frac{-1}{\sqrt{\lambda}}A_1^*v_2
\end{pmatrix}\Bigg\}$$ 
\item[Case 4:] If $\lambda \neq 0$ is a double root of $f(x)$ and the eigenspace $V_{1,\lambda}$ of $-A_1A_1^*$ corresponding to $\lambda$ has dimension $1$. Then there are an eigenvector and an generalized eigenvector of $S$ in $V$ that is corresponding to $\sqrt{\lambda}$, we denote them by $v_1$ and $v_2$ respectively. Then firstly, as in case 1, we obtain two eigenvectors of $S$ corresponding to the eigenvalues $\pm \sqrt{\lambda}$:
$$v_{\sqrt{\lambda}}= \begin{pmatrix}
v_1 \\ \frac{-1}{\sqrt{\lambda}} A_1^*v_1
\end{pmatrix},v_{-\sqrt{\lambda}}= \begin{pmatrix}
v_1 \\ \frac{1}{\sqrt{\lambda}} A_1^*v_1
\end{pmatrix}.$$
And the following vector is  the generalized eigenvector in $V_{\sqrt{\lambda}}$: $$\begin{pmatrix}
v_2 \\ \frac{-1}{\sqrt{\lambda}}A_1^*v_2+\frac{1}{\lambda}A_1^*v_1
\end{pmatrix}.$$ By replacing $v_2$ by $v_2 + c.v_1$ in the above formula, we still get a generalized eigenvector. Therefore we can choose the constant $c$  to get an orthogonal basic for $V_{\sqrt{\lambda}}$. From that we also obtain an orthogonal basic for $V_{-\sqrt{\lambda}}$: 
$$\Bigg\{\begin{pmatrix}
v_1 \\ \frac{1}{\sqrt{\lambda}}A_1^*v_1
\end{pmatrix};\begin{pmatrix}
v_2 \\ \frac{1}{\sqrt{\lambda}}A_1^*v_2+\frac{-1}{\lambda}A_1^*v_1
\end{pmatrix}\Bigg\}$$ 
\end{itemize} 
 Upshot, we have just constructed an orthogonal basic of $(V,Q)$ that consists of generalized eigenvectors of $S$ such that: for each root $\lambda \neq 0$ of $f(x)$ we will have one (or two) pair (pairs) of eigenvectors $w_{\pm \sqrt{\lambda}}$ of $S$ satisfying that $w_{\sqrt{\lambda}} + w_{-\sqrt{\lambda}} \in V_1$ and $w_{\sqrt{\lambda}}-w_{-\sqrt{\lambda}} \in V_2$. If $\lambda$ is in case 4, we also have a pair of generalized eigenvectors $^gw_{\pm \sqrt{\lambda}}$ with the same properties. If $\lambda=0$ then we will have two eigenvector $w_{1,0}$ and $w_{2,0}$ satisfying $w_{i,0} \in V_i$. 
 
Similarly, we also can construct an orthogonal basic $ \{ w_{\pm \sqrt{\lambda}}', ^gw_{\pm \sqrt{\lambda}}', w_{1,0}', w_{2,0}' \}$ that consists of generalized eigenvectors of $T$ having the same properties as above (here we write down all of possible generalized eigenvectors, for a specific case they may not appear in that orthogonal basic). Since $<w_{\sqrt{\lambda}}+w_{-\sqrt{\lambda}}, w_{\sqrt{\lambda}} - w_{-\sqrt{\lambda}}> =0,$ we have that $Q(w_{\sqrt{\lambda}})=Q(w_{-\sqrt{\lambda}})$. Similarly two vectors in all of these pairs $^gw_{\pm \sqrt{\lambda}}, w_{\pm \sqrt{\lambda}}', ^gw_{\pm \sqrt{\lambda}}$ have the same norm w.r.t $Q$. By scaling we may assume that $Q(w_{\pm \sqrt{\lambda}})=Q(w_{\pm \sqrt{\lambda}}')$; $Q(^gw_{\pm \sqrt{\lambda}})=Q(^gw_{\pm \sqrt{\lambda}}')$ and if $f(0)=0$ we also assume that $Q(w_{i,0}) = Q(w_{i,0}')$ for $i=1,2.$ 

From the above construction, the linear map $g: V \rightarrow V$ taking the $w_{\pm \sqrt{\lambda}}, ^gw_{\pm \sqrt{\lambda}},$ and $w_{i,0}$ (if we have) to $w_{\pm \sqrt{\lambda}}', ^gw_{\pm \sqrt{\lambda}}',$ and $w_{i,0}'$ respectively, is orthogonal, and conjugation by $g$ takes $T$ to $S$. Using the properties that $\{ w_{\sqrt{\lambda}} + w_{-\sqrt{\lambda}} ; \, ^gw_{\sqrt{\lambda}}+ \,^gw_{-\sqrt{\lambda}}; w_{1,0} \}_{\lambda}$ span $V_1$ and $\{ w_{\sqrt{\lambda}} - w_{-\sqrt{\lambda}} ; \, ^gw_{\sqrt{\lambda}}-\, ^gw_{-\sqrt{\lambda}}; w_{2,0} \}_{\lambda}$ span $V_2$ (similar for the orthogonal basic related to $T$) , we implies that $g$ preserves $V_1$ and $V_2$, hence $g \in O(V_1) \times O(V_2)$. Conjugating by $g$ multiplies the pfaffian by the determinant of $g$. Hence if the pfaffian of $S$ (also of $T$) is non-zero, we implies that $det(g)=1$. It means that $g \in H^{\theta}$. Since $G$ is the connected component of $H^\theta$ containing the identity, $H^\theta/G \cong \{I_{2n},-I_{2n} \}$, and $-I_{2n}$ acts trivially on $W\cong V_1 \otimes V_2$, we implies that $W//H^\theta = W //G$. As a result, $S$ and $T$ are conjugated by an element in $G$. 

If $det(S)=det(T)=0$ and $det(g) =-1$, then by considering $g'$ that is exactly the same as $g$ except that $g'$ map $w_{1,0}$ to $-w_{1,0}'$, we still have that $S$ and $T$ are conjugated by $g'$, and note that $det(g') = 1$. The same arguments as above will now finish the proof.    
\end{proof}
We also need the following lemma:
\begin{lemma}
Let $k=\mathbb{F}_q$ denote a finite field, and $f(x) \in S(k[[t]])$ is a polynomial of degree $n$ with coefficients in the complete local ring $k[[t]]$. Assume that $ord_t(\Delta(f)) <2$, then any elements of $W_f(k[[t]])$ is regular, i.e. for any $x \in W_f(k[[t]])$, the image $\bar{x} = x (mod \,t)$ is in $W^{reg}(k)$.
\end{lemma}
\begin{proof}
 If $\Delta(f)$ is a unit in $k[[t]]$ then $\Delta(\bar{f}) = \overline{\Delta(f)} = \Delta(f) \, mod \,t$ is non-zero in $k$. Hence, by \cite{Sha16} and the previous Proposition, in case $\bar{f}(0)=0$, $G(\bar{k})$ acts transitively on $W_{\bar{f}}(\bar{k})$, and consequently $W_{\bar{f}}(k) \subset W^{reg}(k)$.
 
If $ord_{t}(\Delta(f))=1$, and we assume that $\bar{x} = x (mod \, t) \in W(k)$ is not regular. In this case, by Definition 3.16 of regularity and the proof of Proposition 3.19, we can deduce that $A.A^* \in GL(V_1)$ or $A^*.A \in GL(V_2)$ is not regular, where $x= \begin{pmatrix}
0 & A_x \\ A_x^* & 0
\end{pmatrix} \in W_f(k[[t]]$, and $A = A_x (mod \, t)$. Without loss of generality, we assume that the matrix $A.A^*$ is not regular as an element in $GL(V_1)(k)$. This is equivalent to that the dimension of the centralizer of $A.A^*$ in $\mathfrak{g}_k:=Lie(GL(V_1))_k$ is not equals to the rank of $GL(V_1)$, and hence it is at least $rank(GL(V_1)+2$ (see \cite{SS70} III. 3.25). By setting $c=Cent_{\mathfrak{g}_{k((t))}}(A_xA_x^*) \cap \mathfrak{g}_{k[[t]]}$, we define an adjoint map $g:=ad(A_x.A_x^*): \mathfrak{g}_{k[[t]]}/c \rightarrow \mathfrak{g}_{k[[t]]}/c$. Then we have that $det(g)= \Delta(A_x.A_x^*)=\Delta(f),$ up to units in $k[[t]]$. If $A.A^*$ is not regular then $\bar{g}=g \, mod\, (t)$ has kernel of dimension at least $2$, hence $ord_t(det(g) \geq 2$, a contradiction.  
\end{proof}
Now we can compute the density of regular locus in the transversal case:
\begin{proposition} If $v$ is a place of $K$, we define 
$$\alpha_v= \frac{|\{x \in S(\mathcal{O}_{K_v}/(\varpi_v^2))| \Delta(x) \equiv 0 \hspace{0.2cm}mod (\varpi_v^2) \}|}{|k(v)^{2n}|}$$and
$$\beta_v= \frac{|\{x \in W(\mathcal{O}_{K_v}/(\varpi_v^2))| \Delta(x) \equiv 0 \hspace{0.2cm}mod (\varpi_v^2) \}|}{|k(v)^{2n^2}|},$$
Then we have the following equalities
\begin{itemize}
\item[1.] $$\lim_{deg(\mathcal{L})\rightarrow \infty} \frac{|\Gamma(C,\mathcal{L}^{\otimes 2}\oplus \mathcal{L}^{\otimes 4} \oplus \cdots \oplus \mathcal{L}^{\otimes 2n-2} \oplus \mathcal{L}^{\otimes n})^{sf}|}{|\Gamma(C,\mathcal{L}^{\otimes 2}\oplus \mathcal{L}^{\otimes 4} \oplus \cdots \oplus \mathcal{L}^{\otimes 2n-2} \oplus \mathcal{L}^{\otimes n})|} = \prod_{v \in |C|}(1-\alpha_v).$$
\item[2.] $$\lim_{deg(\mathcal{L})\rightarrow \infty} \frac{|\Gamma(C,W^{reg}(\mathcal{E},\mathcal{L}))^{sf}|}{|\Gamma(C,W(\mathcal{E},\mathcal{L}))|}= \prod_{v \in |C|}(1-\beta_v)$$
\item[3.]$$\frac{\prod_{v \in |C|}(1-\beta_v)}{\prod_{v \in |C|}(1-\alpha_v)} = \prod_{v \in |C|}\frac{|G(k(v))|}{|k(v)|^{n^2-n}}$$
\end{itemize}
Here the upper script "sf" stands for "square free", i.e. $\Gamma()^{sf}$ is the set of sections whose invariants are transversal to the discriminant locus.
\label{Trasversal locus}
\end{proposition}
\begin{proof}
The first two equalities can be showed by using the results in \cite{HLN14} Section $5$ and notice that by the previous lemma, any element in $W_a(\mathcal{O}_{K_v})$, where $\Delta(a) \not\equiv 0 \hspace{0.15cm}\text{mod}(\varpi_v^2)$, is regular. Now we will prove the last equality by showing that locally:
$$\frac{1-\beta_v}{1-\alpha_v} = \frac{|G(k(v))|}{|k(v)|^{n^2-n}}.$$
To do that, for a given transversal element $a=(a_1,\dots,a_{n-1},e) \in S(R)$, we will count the size of $W^{reg}_a(R)$, where $R=k(v)[\epsilon]/(\epsilon^2)$. Set $T=\overline{T} + \epsilon H \in W^{reg}_a(R)$ and $a= \overline{a} + \epsilon b,$ where $\overline{T},H \in W(k(v))$ and $\overline{a},b=(b_2+\dots,b_{2n+1}) \in S(k(v))$, by Proposition 3.24, we firstly observe that there are $|G(k(v))|$ choices of $\overline{T}$ such that $\pi(\overline{T})= \overline{a}$. With a fixed $\overline{T}$, by considering $H$ and $b$ as elements in the tangent spaces of $W^{reg}$ and $S$, respectively, we can see that the tangent map:
$$d\pi: T_{\overline{T}}V^{reg} \rightarrow T_{\overline{a}}S $$will maps $H$ to $b$. Since $\pi: V^{reg} \rightarrow S$ is smooth, the number of choices of $H$ will be the size of the fiber of $d\pi$ at $b$, and it is equal to $q^{dim_{k(v)}(T_{\overline{T}}W^{reg}) - dim_{k(v)}(T_{\overline{a}}S)}=q^{n^2-n}$. Let $a$ varies the set $S^{trans}(R)$ we will obtain the desired equality.
\end{proof}
\subsubsection{Regular locus in the general case}
Now we consider the general case (without the transversal property). 
\begin{proposition}
\begin{itemize}
\item[1.] We have the following limit
$$\lim_{\deg(\mathcal{L} \rightarrow \infty}\frac{|\Gamma(C, W^{reg}(\mathcal{E},\mathcal{L}))|}{|\Gamma(C, W(\mathcal{E},\mathcal{L}))|} = \prod_{v \in |C|}\frac{c_v}{|k(v)|^{n^2}},$$where $c_v = |W^{reg}(k(v)|$.
\item[2.]The above limit is bounded by 
$$\zeta_C(2)^{-2}\dots \zeta_C(2m)^{-2}.\prod_{v \in |C|}\big(1+c_{2m-1}|k(v)|^{-2}+\dots+c_1|k(v)|^{-2m}\big),$$where $c_i$ are constants which are only depended on $m$ and $p$. If $p>2m+1$ then $c_i$ is only depended on $m$.
\end{itemize}
\label{general locus}
\end{proposition}
\begin{proof}
The first statement is proved in \cite{HLN14} where they use the result of Poonen \cite{Poo03}. To prove the second part, for each element $a \in S(k(v)),$ we will bound the size of $W^{reg}_a(k(v))|$. We have two cases:
\begin{itemize}
\item[Case 1:] If $a$ satisfies the hypothesis in the Proposition 3.24, then by Proposition 3.24, $|W^{reg}_a(k(v))|=|G(k(v))|$. 
\item[Case 2:] If $a$ does not satisfy the hypothesis in the Proposition 3.24, then by Proposition 3.23 $ii)$, $|W_a^{reg}(k(v))| \leq 2|G(k(v))|$.
\end{itemize} 
Our job now is to calculate the number of invariants $a=(a_1,\dots,a_{2m},e)$ in the second case above. We also have several cases as follows:
\begin{itemize}
\item[Case 1:] If the corresponding polynomial $f_a(x)$ is divided by $x^2$, then $a_{2m}=0$ and $e=0$. Hence, the total number of $a$'s in this case is $|q^{2m-1}|$, where $q=|k(v)|$.
\item[Case 2:] If $f_a(x)$ has a root $\alpha$ of order $e>2$ in $\overline{k(v)}$, we denote $m_\alpha(x)$ the minimal polynomial of $\alpha$ over $k(v)$, then 
$$f_a(x)= m_\alpha(x)^e.g(x) \,\,\,\textbf{if $m_\alpha(x)$ is separable,}$$
$$f_a(x)=m_\alpha(x).g(x), \textbf{where $m_\alpha(x)=h(x^{p^t})$ for some $t \in \mathbb{N}$}$$ 
In both cases, $a$ is defined by the coefficients of $m_\alpha(x)$ and $g(x)$. In the former case, if we set the degree of $m_\alpha$ and $g$ by $m_1$ and $m_2$ respectively, then $a$ can be defined by $m_1+m_2= 2m+1-(e-1)m_1$ coordinates. Hence the total number of $a$ in this case is bounded by $q^{2m+1-(e-1)m_1}$. In the later case, we also easily deduce that $a$ is defied by at most $2m+1-p-1$ coordinates, thus, the total number of $a$'s is bounded by $q^{2m+1-p-1}$. Note that we only have finite "types" of $m_\alpha$ ("type" here means the choice of the degree of $m_\alpha$ in the former case and the choice of $h(x^{p^t})$ in the later case). So the total number of $a \in S(k(v))$ satisfying the corresponding $f_a(x)$ has a root in $\overline{k(v)}$ of order at least $3$ is bounded by 
$$\sum_{i=1}^{2m-1}c_iq^{i},$$where $c_i$ are constants that are only depended on $m$ and $p$.  
\end{itemize}
The upper bound of the limit in 1) is the consequence of the above calculation.
\end{proof}
\subsubsection{Regular locus in minimal case}
To take the average over the set of hyperelliptic curves, we need to consider the minimal data $(\mathcal{L}, \underline{a})$. Note that the transversal condition implies the minimal condition. Furthermore, the results in \cite{HLN14} Section $5$ also help us to see that the density of minimal locus is the product of local densities. The local condition for a minimal data is that at a closed point $v \in |C|$, the tuple of sections $\underline{a}$ does not come from $\mathcal{L}(-v)$. The when $deg(\mathcal{L}) >>0$, the density of tuples $\underline{a}$ that come from $\mathcal{L}(-v)$ is 
\begin{eqnarray}
 & \frac{|H^0 \big(C, \mathcal{L}(-v))^{\otimes 2} \oplus \cdots \oplus (\mathcal{L}(-v))^{\otimes 4m} \oplus (\mathcal{L}(-v))^{\otimes 2m+1}\big)|}{|H^0 \big(C, \mathcal{L}^{\otimes 2} \oplus \cdots \oplus \mathcal{L}^{\otimes 4m} \oplus \mathcal{L}^{\otimes 2m+1} \big)|} \\
 = & \frac{1}{|k(v)|^{(2m+1)^2}}
\end{eqnarray}
We have just proved the following result:
\begin{proposition} Given a line bundle of degree big enough, the density of minimal tuples $ \underline{a} \in H^0 \big(C, \mathcal{L}^{\otimes 2} \oplus \cdots \oplus \mathcal{L}^{\otimes 4m} \oplus \mathcal{L}^{\otimes 2m+1} \big)$ is $\zeta_C((2m+1)^2)^{-1}$.
\end{proposition}
By using  similar argument as in the previous subsection, we obtain the following estimation:
\begin{proposition} For a given $G-$bundle $\mathcal{E}$, we denote $\Gamma(C, W^{reg}(\mathcal{E}, \mathcal{L}))^{min}$ the set of sections in $W^{reg}(\mathcal{E}, \mathcal{L})$ whose associated data $(\mathcal{L}, \underline{a})$ is minimal. Similarly for the notation $\Gamma(C, \mathcal{L}^2 \oplus \cdots \oplus \mathcal{L}^{4m} \oplus \mathcal{L}^{2m+1})^{min}$ - the set of minimal tuples $\underline{a}$. Then 
  $$ \lim\limits_{deg(\mathcal{L}) \rightarrow \infty} \dfrac{\dfrac{|\Gamma(C,W^{reg}(\mathcal{E},\mathcal{L}))^{min}|}{|\Gamma(C,W(\mathcal{E},\mathcal{L}))|}}{\dfrac{|\Gamma(C,\mathcal{L}^{\otimes 2}\oplus \mathcal{L}^{\otimes 4} \oplus \cdots \oplus \mathcal{L}^{\otimes 2n-2} \oplus \mathcal{L}^{\otimes n})^{min}|}{|\Gamma(C,\mathcal{L}^{\otimes 2}\oplus \mathcal{L}^{\otimes 4} \oplus \cdots \oplus \mathcal{L}^{\otimes 2n-2} \oplus \mathcal{L}^{\otimes n})|}}$$
$$ \leq \zeta_C(2)^{-2}\dots \zeta_C(2m)^{-2}. \zeta_C((2m+1)^2). \prod_{v \in |C|}\big(1+c_{2m-1}|k(v)|^{-2}+\dots+c_1|k(v)|^{-2m} $$  $$\hspace{10cm} -2|k(v)|^{(2m+1)^2}\big),$$
 where $c_i$ are the same as in Proposition 3.27.
\end{proposition}
\subsection{Counting}
Let recall some notations: $V_1$ and $V_2$ are orthogonal spaces over $k =\mathbb{F}_q$ of dimension $n=2m+1$, $G= SO(V_1) \times SO(V_2)$ is split, and $W=V_1 \otimes V_2$ a representation of $G$. We can see each element in $W$ as a skew-self adjoint matrix whose diagonal blocks are $0$:
$$W(k) = \bigg\{\begin{pmatrix}
  0 & A \\
  -A^* & 0  
 \end{pmatrix} \big| A \in M_n(k) \bigg\} $$where $A^*$ is a matrix obtained from $A$ by taking the transpose via the anti diagonal. Denote $\mathcal{G}$ the set of $G-$bundles. The goal of this section is to estimate the following limit:
 \begin{equation}
 \lim_{deg(\mathcal{L}) \rightarrow \infty} \int_{\mathcal{E} \in \mathcal{G}} \frac{|H^0(C, (\mathcal{E}\times^GW^{reg}) \otimes \mathcal{L})|}{|Aut(\mathcal{E})|.|\mathcal{A}_\mathcal{L}(k)|}.
 \end{equation}
The denominator in the above limit can be easily calculated (using the similar arguments as those in chapter 1). In fact, $\mathcal{A}_\mathcal{L}(k)$ classifies hyperelliptic curves over $C$ whose coefficients in their affine Weierstrass equation as in Section 3.1 all come from $\mathcal{L}$. This implies that when $deg(\mathcal{L})$ is large enough, we have
$$|\mathcal{A}_\mathcal{L}(k)| = |H^0(C, \mathcal{L}^{\otimes 2} \oplus \cdots \oplus \mathcal{L}^{\otimes 2n-2} \oplus \mathcal{L}^{\otimes n})| = q^{n^2d+n(1-g)} $$where $d$ is the degree of $\mathcal{L}$ and $g$ denotes the genus of the curve $C$.
\subsubsection{Automorphism group of G-bundle}
If $E$ is a $G-$bundle, then $E$ can be expressed as the product $E_1 \times E_2$, where $E_i$ are $SO(V_i)-$bundles. Hence, $Aut_G(E) = Aut_{SO(V_1)}(E_1) \times Aut_{SO(V_2)}(E_2)$, and then we could apply the results in section 1 chapter 1 to estimate the size of automorphic groups.  Suppose that the bundle $E_1$ has the canonical reduction $(P_1, \sigma_1)$, and the parabolic subgroup $P_1$ has the Levi quotient given by 
$$ L_1 \cong GL(r_1) \times GL(r_2) \times \cdots \times GL(r_t) \times SO(r_0),$$where $r_0 + 2\sum_{i=1}^{h}r_i = n$. In the other words, there exists a flag of isotropic subspaces 
$$0=V_{1,0} \subset V_{1,1} \subset \cdots \subset V_{1,h} \subset V_{1,h}^{\bot} \subset \cdots \subset V_{1,1}^{\bot} \subset V_1,$$where $dim(V_{1,i}/V_{1,i-1})= x_i$ for $i = \overline{1,h}$ and $dim( V_{1,t}^{\bot}/V_{1,h}) = x_0$. From that we obtain a filtration of the vector bundle $E_1 \times^{SO(V_1)}V_1:$
\footnotesize $$0={E_1}_{P_1} \times^{P_1}V_{1,0} \subset \cdots \subset {E_1}_{P_1} \times^{P_1}V_{1,h} \subset {E_1}_{P_1} \times^{P_1}V_{1,h}^{\bot} \subset \cdots \subset {E_1}_{P_1} \times^{P_1}V_{1,1}^{\bot} \subset E_1 \times^{SO(V_1)} V_1$$\normalsize satisfying that the quotient bundles $X_i = {E_1}_{P_1} \times^{P_1}V_{1,i}/{E_1}_{P_1} \times^{P_1}V_{1,i-1}$ for $i= \overline{1,h}$ and $X_{0}={E_1}_{P_1} \times^{P_1}V_{1,h}^{\bot}/{E_1}_{P_1} \times^{P_1}V_{1,h}$ are semistable. If we denote the slope of the vector bundle $X_i$ by $x_i$, then by definition of the canonical reduction, we deduce that $x_1 > x_2 > \cdots > x_h >x=0 =0.$ \\
Similarly, for the $SO(V_2)-$bundle $E_2$ we associate it with a unique parabolic subgroup $P_2$ of $SO(V_2)$ and a set of semistable vector bundles $Y_i$ for $i = \overline{0,l}$ satisfying
$$t_0 + 2\sum_{i=1}^{l}t_i = n,$$
$$y_1 > y_2 > \cdots > y_l >y_0 = 0,$$where $t_i$ and $y_i$ denote the rank and the slope of vector bundle $Y_i$, respectively. With these notations, we can estimate the size of the automorphic group as follow:
\begin{proposition}
\begin{itemize}
\item[(i)] There exists a constant $c$ that is only depended on $n$ and $g$ such that for any $G-$bundles $E$ with canonical reduction to $P$, we have
\begin{align*}
-c \leq dim(Aut_G(E)) - dim(Aut_L(E_L)) - \sum_{i=1}^t \big( h^0(\wedge^2X_i) + h^0(X_i \otimes X_0) \big) - \\ -\sum_{t\geq j > i > 0}\big(h^0(X_i \otimes X_j)+h^0(X_i\otimes X_j^*)\big)-\sum_{i=1}^l\big(h^0(\wedge^2Y_i)+h^0(Y_i \otimes X_0)\big)- \\ -\sum_{l\geq j > i > 0}\big(h^0(Y_i \otimes Y_j)+ h^0(Y_i \otimes Y_j^*)\big) \leq c \hspace{4cm}
\end{align*}
\item[(ii)] In particular, if $x_i-x_{i+1} > 2g-2$ for all $i$ and $y_j - y_{j+1} > 2g-2$ for all $j$, then the constant $c$ in $(i)$ can be taken to be $0$.
\end{itemize}
\end{proposition}

\subsubsection{General case}
Given a $G-$bundle $E$ as above (with the canonical reduction to $P$ and associated vector bundles $X_i$ for $0 \leq i \leq t$, and $Y_j$ for $0\leq j \leq l$), we firstly assume that $x_i-x_{i+1}$ for $0 \leq i \leq t$, $x_t$, $y_j-y_{j+1}$ for $0 \leq j \leq l$, and $y_l$ are all bigger than $2g-2$. This condition makes sure that the filtration associated with the canonical reduction of $E$ is split. Precisely, we can express the vector bundles $E_i \times^{SO(V_i)}V_i$ as direct sums:
\begin{eqnarray}
E_1 \times^{SO(V_1)}V_1 = X_0 \oplus \bigoplus_{i=1}^t \big(X_i \oplus X_i^*\big), \\
E_2 \times^{SO(V_2)}V_2 = Y_0 \oplus \bigoplus_{j=1}^l \big(Y_j \oplus Y_j^*\big).
\end{eqnarray}
As a result, any global sections of the vector bundle $E \times^G W$ is of the following matrix form: 
$$\begin{pmatrix}
  0 & A \\
  -A^* & 0  
 \end{pmatrix},$$ where $A$ is the section of 
 $$\begin{pmatrix}
  X_1\otimes Y_1 & X_1\otimes Y_2 & \cdots & X_1\otimes Y_0 & X_1\otimes Y_l^* & \cdots & X_1\otimes Y_1^* \\
  X_2\otimes Y_1 & X_2\otimes Y_2 & \cdots & X_2\otimes Y_0 & X_2\otimes Y_l^* & \cdots & X_2\otimes Y_1^* \\
  \vdots & \vdots & \vdots & \vdots & \vdots & \vdots & \vdots \\
  X_0\otimes Y_1 & X_0\otimes Y_2 & \cdots & X_0\otimes Y_0 & X_0\otimes Y_l^* & \cdots & X_0\otimes Y_1^* \\  
 X_t^*\otimes Y_1 & X_t^*\otimes Y_2 & \cdots & X_t^*\otimes Y_0 & X_t^*\otimes Y_l^* & \cdots & X_t^*\otimes Y_1^* \\
\vdots & \vdots & \vdots & \vdots & \vdots & \vdots & \vdots \\
  X_1^*\otimes Y_1 & X_1^*\otimes Y_2 & \cdots & X_1^*\otimes Y_0 & X_1^*\otimes Y_l^* & \cdots & X_1^*\otimes Y_1^* 
 \end{pmatrix}$$
By looking at the slopes of vector bundles $X_i$ and $Y_j$, we obtain some cases as follows: 
\begin{description}
\item[Case 1:]If $l=0$, or $t=0$. 
\item[Case 2:]If $t=l=1,$ $x_1 > d$, and $y_1>d$. In this case, the vector bundles $X_0\otimes Y_1^*\otimes \mathcal{L}$, $X_1^*\otimes Y_1\otimes \mathcal{L}$, and $X_1^*\otimes Y_1^* \otimes \mathcal{L}$ have negative degrees. By Proposition \ref{global sections of semistable}, they have no global section, hence, any sections $\alpha$ in $H^0(C,(E\times^GW)\otimes \mathcal{L})$ will have the following form:
$$A_\alpha = \begin{pmatrix}
   a & b & c \\
   d & e & 0 \\
   f & 0 & 0 
 \end{pmatrix}.$$
 This implies that
 $$A_\alpha . A_\alpha^* = \begin{pmatrix}
   cf & be+cd & 2ac+b^2 \\
   0 & e^2 & be+cd \\
   0 & 0 & cf 
 \end{pmatrix}.$$
 We deduce that $\Delta(H_\alpha)=0$, thus, the generic fiber of $H_\alpha$ is not a hyperelliptic curve over $K(C)$. 
 \item[Case 3:] If $t=l=1,$ $x_1 > d$, $y_1 \leq d$, and $x_1-y_1 \leq d$. In this case 
 $$A_\alpha = \begin{pmatrix}
   a & b & c \\
   d & e & f \\
   g & 0 & 0 
 \end{pmatrix}$$
where $g \in H^0(C,X_1^*\otimes Y_1)$. To make sure that $det(A_\alpha) \neq 0$, we need to put an extra condition that is $r_1 \leq t_1$. Then \footnotesize
$$H^0\big((\mathcal{E}\times^GW) \otimes \mathcal{L}\big)= \begin{pmatrix}
   H^0(X_1\otimes Y_1\otimes \mathcal{L}) & H^0(X_1\otimes Y_0\otimes \mathcal{L}) & H^0(X_1\otimes Y_1^*\otimes \mathcal{L}) \\
   H^0(X_0\otimes Y_1\otimes \mathcal{L}) & H^0(X_0\otimes Y_0\otimes \mathcal{L}) & H^0(X_0\otimes Y_1^*\otimes \mathcal{L}) \\
   H^0(X_1^*\otimes Y_1\otimes \mathcal{L}) & 0 & 0 
 \end{pmatrix},$$ \normalsize hence
\begin{align*}
&\hspace{1cm}\frac{|H^0(C, (\mathcal{E}\times^GW) \otimes \mathcal{L})|}{|Aut(\mathcal{E})|.|\mathcal{A}_\mathcal{L}(k)|} \\
&= \frac{q^{r_1t_0x_1+r_1t_1(x_1+y_1)+d(n^2-r_1t_0-r_1t_1)}}{|Aut(X_1)\times Aut(X_0)\times Aut(Y_1) \times Aut(Y_0)| .q^{r_1x_1(r_1-1+r_0)+t_1y_1(t_1-1+t_0)+dn^2}} \\
&= \frac{q^{-r_1x_1(t_1-r_1-1)-t_1y_1(t_1-r_1+t_0-1)-dr_1(t_0+t_1)}}{|Aut(X_1)\times Aut(X_0)\times Aut(Y_1) \times Aut(Y_0)|} \\
&\leq  \frac{1}{q^{-dr_1+dr_1(t_0+t_1)}|Aut(X_1)\times Aut(X_0)\times Aut(Y_1) \times Aut(Y_0)|.} 
\end{align*}
So the contribution of this case will be $0$.
\item[Case 4:] If $t=l=1,$ $x_1 \leq d$, $y_1 \leq d$, and $x_1+y_1 > d$. Then similar to the above calculations, we obtain that
\begin{align*}
&\hspace{1cm}\frac{|H^0(C, (\mathcal{E}\times^GW) \otimes \mathcal{L})|}{|Aut(\mathcal{E})|.|\mathcal{A}_\mathcal{L}(k)|} \\
&= \frac{q^{r_1t_1(x_1+y_1)+d(n^2-r_1t_1)}}{|Aut(X_1)\times Aut(X_0)\times Aut(Y_1) \times Aut(Y_0)| .q^{r_1x_1(r_1-1+r_0)+t_1y_1(t_1-1+t_0)+dn^2}} \\
&= \frac{1}{q^{r_1x_1(r_1+r_0-t_1-1)+t_1y_1(t_1-r_1+t_0-1)+dr_1t_1}|Aut(X_1)\times Aut(X_0)\times Aut(Y_1) \times Aut(Y_0)|} \\
&\leq  \frac{1}{q^{dr_1t_1}|Aut(X_1)\times Aut(X_0)\times Aut(Y_1) \times Aut(Y_0)|.} 
\end{align*}
Hence this case gives the 0 contribution in the average.
\item[Case 5:] If $t=1, l=2$. By considering the slope of $X_1, Y_1,$ and $Y_2$, we will obtain a lot of subcases. Let begin with the general form of $\mathcal{E}\times^GW \otimes \mathcal{L}$: 
$$\begin{pmatrix}
   X_1\otimes Y_1\otimes \mathcal{L} & X_1\otimes Y_2\otimes \mathcal{L} & X_1\otimes Y_0\otimes \mathcal{L} & X_1\otimes Y_2^*\otimes \mathcal{L} & X_1\otimes Y_1^*\otimes \mathcal{L} \\
   X_0\otimes Y_1\otimes \mathcal{L} & X_0\otimes Y_2\otimes \mathcal{L} & X_0\otimes Y_0\otimes \mathcal{L} & X_0\otimes Y_2^*\otimes \mathcal{L} & X_0\otimes Y_1^*\otimes \mathcal{L} \\
   X_1^*\otimes Y_1\otimes \mathcal{L} & X_1^*\otimes Y_2\otimes \mathcal{L} & X_1^*\otimes Y_0\otimes \mathcal{L} & X_1^*\otimes Y_2^*\otimes \mathcal{L} & X_1^*\otimes Y_1^*\otimes \mathcal{L})
 \end{pmatrix}.$$ \normalsize
The above form lead to:
\begin{itemize}
\item If $|x_1 -y_1| > d$ then $det(H_\alpha)=0$ for any $\alpha \in H^0(\mathcal{E}\times^GW \otimes \mathcal{L}).$ So we can ignore this case.
\item If $x_1-y_2 > d$ and $t_1<r_1$, then $det(H_\alpha)=0$ for any $\alpha \in H^0(\mathcal{E}\times^GW \otimes \mathcal{L}).$ This case will not counted in the average.
\item If $|x_1 -y_1| \leq d$, $x_1-y_2 > d$, $t_1 \geq r_1$, and $y_1 > d$, then for any $\alpha \in H^0(\mathcal{E}\times^GW \otimes \mathcal{L}):$
$$A_\alpha.A_\alpha^* \in  \begin{pmatrix}
   H^0(X_1\otimes X_1^*\otimes \mathcal{L}^2) & H^0(X_1\otimes X_0\otimes \mathcal{L}^2) & H^0(X_1\otimes X_1\otimes \mathcal{L}^2) \\
   0 & H^0(X_0\otimes X_0\otimes \mathcal{L}^2) & H^0(X_0\otimes X_1\otimes \mathcal{L}^2) \\
   0 & 0 &  H^0(X_1^*\otimes X_1\otimes \mathcal{L}^2)
 \end{pmatrix}.$$
 It is easy to see that in this case $\Delta(H_\alpha)=0$, thus, we will not count this case in our average.
\item If $|x_1 -y_1| \leq d$, $x_1-y_2 > d$, $t_1 \geq r_1$, and $y_1 \leq d$. Then
\begin{align*}
|H^0(C, (\mathcal{E}\times^GW) \otimes \mathcal{L})| = q^{r_1x_1(2t_2+t_1+t_0)+t_1y_1r_1+d(n^2-r_1t_0-r_1t_1-2r_1t_2)} \\
|Aut(\mathcal{E})|.|\mathcal{A}_\mathcal{L}(k)| = |Aut_L(\mathcal{E}_L)|. q^{r_1x_1(r_1-1+r_0)+t_1y_1(t_1-1+2t_2+t_0)+t_2y_2(t_2-1+t_0)+n^2}
\end{align*} 
Since $t_1 \geq r_1$, $r_1+r_0 \geq 2t_2 +t_1+t_0$, and $dr_1(t_0+t_1) \geq r_1x_1$, we have that
$$\frac{|H^0(C, (\mathcal{E}\times^GW) \otimes \mathcal{L})|}{|Aut(\mathcal{E})|.|\mathcal{A}_\mathcal{L}(k)|} \leq \frac{1}{|Aut_L(\mathcal{E}_L)|q^{2dr_1t_2}}.$$
Hence the contribution of this case will be $0$ when $d \rightarrow \infty$.
\item If $|x_1 -y_1| \leq d$, $x_1-y_2 \leq d$ and $x_1 > d$. Then to make sure that $\Delta(H_\alpha) \neq 0$ and $det(A_\alpha) \neq 0$, it is required that $y_2 \leq d$ and $t_1+t_2 \geq r_1$. If $y_1 \leq d$ then 
$$|H^0(C, (\mathcal{E}\times^GW) \otimes \mathcal{L})| = q^{r_1x_1(t_2+t_1+t_0)+t_1y_1r_1+ t_2y_2r_1+d(n^2-r_1t_0-r_1t_1-r_1t_2)}.$$
It is easy to see that
 $$r_1x_1(t_2+t_1+t_0) -dr_1(t_0+t_1) \leq r_1x_1(r_1-1+r_0),$$ 
 $$t_1y_1r_1+t_2y_2r_1 \leq t_1y_1(t_1-1+2t_2+t_0)+t_2y_2(t_2-1+t_0).$$
 The above inequalities imply that
 $$\frac{|H^0((\mathcal{E}\times^GW) \otimes \mathcal{L})|}{|Aut(\mathcal{E})|.|\mathcal{A}_\mathcal{L}(k)|} \leq \frac{1}{|Aut_L(\mathcal{E}_L)|q^{dr_1t_2}},$$and the contribution of this case to the average will be $0$ when $d \rightarrow \infty$. \\
 If $y_1 > d$, then 
 $$|H^0((\mathcal{E}\times^GW) \otimes \mathcal{L})| = q^{r_1x_1(t_2+t_1+t_0)+t_1y_1(r_1+r_0)+ t_2y_2r_1+d(n^2-r_1t_0-r_1t_1-r_1t_2-t_1r_0)}.$$
 To make sure that $det(A_\alpha) \neq 0$, we need to put an extra condition: $r_1 \geq t_1$. Then
 \begin{align*}
 e&=&r_1x_1(r_1-1+r_0)+t_1y_1(t_1-1+2t_2+t_0)+t_2y_2(t_2-1+t_0)+dn^2- \hspace{2cm}\\
 &&-\big(r_1x_1(t_2+t_1+t_0)+t_1y_1(r_1+r_0)+ t_2y_2r_1+d(n^2-r_1t_0-r_1t_1- \hspace{2cm}\\ && r_1t_2-t_1r_0)\big) \hspace{2cm}\\
 &\geq & r_1x_1(t_1+t_2-r_1-1) +t_1y_1(r_1-t_1-1)+t_2y_2(t_2-r_1)+ \hspace{2cm}\\ && dr_1(t_0+t_1+t_2)+dt_1r_0. \hspace{2cm}
 \end{align*}
 If $r_1=t_1$, then 
 $$e \geq -t_1y_1-t_2y_2r_1+3dr_1+dt_1r_0 \geq dt_1(r_0-t_2) \geq 2d \,\,\text{(since $y_1 \leq 3d$}).$$
 If $r_1 > t_1$ then
 $$e \geq -r_1x_1+dr_1(t_0+t_1)-t_2y_2r_1+dr_1t_2+dt_1r_0 > dt_1r_0 \,\,\,\text{(since $x_1\leq 2d, y_2 \leq d)$}$$
From the above inequalities, we conclude that this case has no effect to our average when $d \rightarrow \infty$.
\item If $|x_1 -y_1| \leq d$, $x_1 \leq d$, and $y_2 > d$. To make sure that $det(A_\alpha) \neq 0$, the extra condition we need to add is that $t_1+t_2 \leq r_1$. Then 
$$|H^0(C, (\mathcal{E}\times^GW) \otimes \mathcal{L})|\hspace{5cm}$$ $$= q^{r_1x_1(t_2+t_1)+t_1y_1(r_1+r_0)+ t_2y_2(r_1+r_0)+d(n^2-r_1t_2-r_1t_1-r_0t_2-r_0t_1)}.$$
And 
\begin{align*}
 e&= r_1x_1(r_1-1+r_0)+t_1y_1(t_1-1+2t_2+t_0)+t_2y_2(t_2-1+t_0)+dn^2- \\
 & -\big(r_1x_1(t_2+t_1)+t_1y_1(r_1+r_0)+ t_2y_2(r_1+r_0)+d(n^2-r_1t_2-r_1t_1- \\ & \hspace{10cm} r_0t_2-r_0t_1)\big) \\
 &\geq  r_1x_1(t_1+t_2+t_0-r_1-1) +t_1y_1(r_1-t_1-1)+t_2y_2(r_1-t_2-2t_1-1)\\&\hspace{8cm}+dr_1(t_1+t_2)+dr_0(t_1+t_2) \\
 & \geq  2dt_2(-t_1-1)+d(t_1+t_2)^2 +dr_0(t_1+t_2) \\
 &\geq  dr_0(t_1+t_2). 
 \end{align*}
\item If $|x_1 -y_1| \leq d$, $x_1 \leq d$, $x_1+y_2 > d$, $y_2 \leq d$, and $y_1 > d$. Then $r_1 \geq t_1$ and
$$|H^0(C, (\mathcal{E}\times^GW) \otimes \mathcal{L})| = q^{r_1x_1(t_2+t_1)+t_1y_1(r_1+r_0)+ t_2y_2r_1+d(n^2-r_1t_1-r_1t_2-t_1r_0)}.$$
\begin{align*}
 e&= r_1x_1(r_1-1+r_0)+t_1y_1(t_1-1+2t_2+t_0)+t_2y_2(t_2-1+t_0)+dn^2- \\
 &-\big(r_1x_1(t_2+t_1)+t_1y_1(r_1+r_0)+ t_2y_2r_1+d(n^2-r_1t_1-r_1t_2-t_1r_0)\big) \\
 &\geq  r_1x_1(t_1+t_2+t_0-r_1-1) +t_1y_1(r_1-t_1-1)-t_2y_2r_1)+dr_1(t_1+t_2)+dt_1r_0\\
 & \geq t_1y_1(r_1-t_1-1)-t_2y_2r_1+dr_1(t_1+t_2)+dt_1r_0.
 \end{align*}
If $r_1=t_1$, then $r_0=t_0+2t_2 \geq 3$. Thus,
 $$e \geq -t_1y_1-t_2y_2r_1+dr_1t_2+dr_1t_1+3dt_1 > dr_1t_1 > d \,\,\text{(since $y_1 \leq 2d$}).$$
 If $r_1 > t_1$ then
 $$e \geq -t_2y_2r_1+dr_1t_2+dt_1(r_0+r_1) > dt_1r_0 \,\,\,\text{(since $y_2 \leq d)$}$$
From the above inequalities, we conclude that this case has no effect to our average when $d \rightarrow \infty$.
\item If $x_1 \leq d$, $x_1+y_2 > d$, and $y_1 \leq d$. Then
$$|H^0(C, (\mathcal{E}\times^GW) \otimes \mathcal{L})| = q^{r_1x_1(t_2+t_1)+t_1y_1r_1+ t_2y_2r_1+d(n^2-r_1t_1-r_1t_2)}.$$
We consider
\begin{align*}
 e&= r_1x_1(r_1-1+r_0)+t_1y_1(t_1-1+2t_2+t_0)+t_2y_2(t_2-1+t_0)+dn^2- \\
 &-\big(r_1x_1(t_2+t_1)+t_1y_1r_1+ t_2y_2r_1+d(n^2-r_1t_1-r_1t_2)\big) \\
 &\geq  r_1x_1(t_1+t_2+t_0-r_1-1) +t_1y_1(r_1+r_0-t_1-1)-t_2y_2r_1)+dr_1(t_1+t_2)\\
 & \geq -t_2y_2r_1+dr_1(t_1+t_2) \hspace{1cm}\text{(since $t_1+t_2+t_0>r_1,$ and $r_1+r_0>t_1$)} \\
 & \geq dr_1t_1 \hspace{1cm} \text{(since $y_2 < d$)}.
 \end{align*}
It implies that the contribution in this case is $0$. 
\item If $|x_1 -y_1| \leq d$, $x_1+y_2 \leq d$, and $y_1 > d$. It is necessary that $r_1 \geq t_1$. Then
$$|H^0(C, (\mathcal{E}\times^GW) \otimes \mathcal{L})| = q^{r_1x_1t_1+t_1y_1(r_1+r_0)+d(n^2-r_1t_1-r_0t_1)}.$$
We consider
\begin{align*}
 e&= r_1x_1(r_1-1+r_0)+t_1y_1(t_1-1+2t_2+t_0)+t_2y_2(t_2-1+t_0)+dn^2- \\
 &-\big(r_1x_1t_1+t_1y_1(r_1+r_0)+d(n^2-r_1t_1-r_0t_1)\big) \\
 &\geq  r_1x_1(r_1+r_0-t_1-1) +t_1y_1(r_1-t_1-1)+dt_1(r_1+r_0)\\
 & \geq -t_2y_2r_1+dr_1(t_1+t_2) \hspace{1cm}\text{(since $t_1+t_2+t_0>r_1,$ and $r_1+r_0>t_1$)} \\
 & \geq dr_1t_1 \hspace{1cm} \text{(since $y_2 < d$)}.
 \end{align*}
It implies that the contribution in this case is $0$. 
\item If $|x_1 -y_1| \leq d$, $x_1+y_2 \leq d$, $x_1+y_1 > d$, and $y_1 \leq d$. Then
$$|H^0(C, (\mathcal{E}\times^GW) \otimes \mathcal{L})| = q^{r_1x_1t_1+t_1y_1r_1+d(n^2-r_1t_1)}.$$
We consider
\begin{align*}
 e&= r_1x_1(r_1-1+r_0)+t_1y_1(t_1-1+2t_2+t_0)+t_2y_2(t_2-1+t_0)+dn^2- \\
 &-\big(r_1x_1t_1+t_1y_1r_1+d(n^2-r_1t_1)\big) \\
 &\geq  r_1x_1(r_1+r_0-t_1-1) +t_1y_1(r_1+r_0-t_1-1)+dt_1r_1\\
 & \geq dr_1t_1 \hspace{1cm}\text{(since $r_1+r_0>t_1$)} \\
 \end{align*}
It implies that the contribution in this case is $0$.
\item If $x_1+y_1 \leq d$. We will treat this case later.
\end{itemize}
Now we will consider the general cases:
\begin{proposition} Given a $G-$bundle $\mathcal{E}$ with the associated data $(P,X_i,Y_j)$ such that $x_1+y_1>d$. If $h \neq l$ or $h=l \neq m$, then 
$$\frac{|\{\alpha \in H^0(C, (\mathcal{E}\times^GW) \otimes \mathcal{L})| Det(A_\alpha) \neq 0, \Delta(H_\alpha) \neq 0 \}|}{|Aut_G(\mathcal{E})|.|\mathcal{A}_\mathcal{L}(k)|} \leq \frac{c}{|Aut_L(\mathcal{E}_L)|.q^d}$$, where $c$ is only depended on $g$ and $n$, $d$ is the degree of the line bundle $\mathcal{L}$. 
\end{proposition}
\begin{proof}
We will prove this proposition by induction. Notice that we have already consider some initial cases. Let assume that the statement is true for all pair $(h',l')$, where $h' \leq h$, $l' \leq l$, and $h'+l' < h+l$. Now having fixed numbers of $X_i$ and $Y_j$, we will find $(x_i,r_i,y_j,t_j)_{0\leq i \leq h; 0\leq j \leq l}$ such that the fractional expression:
$$A=\frac{|\{\alpha \in H^0(C, (\mathcal{E}\times^GW) \otimes \mathcal{L})| Det(A_\alpha) \neq 0, \Delta(H_\alpha) \neq 0 \}|}{|Aut_G(\mathcal{E})|.|\mathcal{A}_\mathcal{L}(k)|}$$is "maximal". \\
Note that to prove our inequality, firstly we can make use of the semi-stable filtration associated to the canonical reduction of $\mathcal{E}$. Then we approximate the dimensions of each components in that filtration by their degrees. More precisely, we can replace $H^0(C, (\mathcal{E}\times^GW) \otimes \mathcal{L})$ the numerator of $A$ by 
$$H^0\big((X_1\oplus\cdots \oplus X_h \oplus X_0 \oplus X_h^* \oplus \cdots \oplus X_1^*)\otimes(Y_1\oplus\cdots \oplus Y_h \oplus Y_0 \oplus Y_l^* \oplus \cdots \oplus Y_1^*)\otimes \mathcal{L}\big).$$And the denominator of $A$ can be replaced by:
$$|Aut_L(\mathcal{E}_L)|.q^{\sum_{i=1}^h\big(r_ix_i(r_i-1+2r_{i+1}+\cdots+2r_h+r_0)\big)+\sum_{j=1}^l\big(t_jy_j(t_j-1+2t_{j+1}+\cdots+2t_l+t_0)\big) +n^2 }.$$Our problem now is to prove 
$$A'=\dfrac{\mathsmaller{H^0\big((X_1\oplus\cdots \oplus X_h \oplus X_0 \oplus X_h^* \oplus \cdots \oplus X_1^*)\otimes(Y_1\oplus\cdots \oplus Y_h \oplus Y_0 \oplus Y_l^* \oplus \cdots \oplus Y_1^*)\otimes \mathcal{L}\big)}}{q^{\sum_{i=1}^h\big(r_ix_i(r_i-1+2r_{i+1}+\cdots+2r_h+r_0)\big)+\sum_{j=1}^l\big(t_jy_j(t_j-1+2t_{j+1}+\cdots+2t_l+t_0)\big) +n^2 }} \leq \frac{c}{q^d}$$
Given the value of the slopes $x_i$ and $y_j$, we will find the rank $r_i$ and $t_j$ such that $A$ is as large as possible. Now if we fix all of $r_i$ and $t_j$ except $r_1$ and $r_2$, we also assume that there is no relations attached. Then $r_1+r_2$ is a fixed number and we could consider $r_1$ as the only variable in $A'$. The numerator of $A'$ is a power of $q$ with the power is a linear expression of $r_1$, and the denominator is a power of $q$ with the power is a degree 2 polynomial of $r_1$. Moreover, since
\begin{align*}
&r_1x_1(r_1+2r_2+\cdots +2r_h+r_0)+r_2x_2(r_2+2r_3+\cdots +2r_h+r_0)\\ &=x_1r_1(n-1-r_1)+x_2r_2^2+ax_1+b \\
& = -(x_1-x_2)r_1^2+a'x_1+b',
\end{align*}where $a',b'$ are some constants, we implies that $A'$ will obtain the maximal value at the extreme points of $r_1$. For example, if there is no relation attached to $r_1$ (relate to the condition that $Det(A_\alpha) \neq 0$, and $\Delta(H_\alpha) \neq 0$), then the extreme values of $r_1$ is $0$ and $r_1+r_2$. In both cases, we have already reduced the value of $h$ and hence we could apply the induction hypothesis. 

Keep it in mind, we now need to consider the conditions that $Det(A_\alpha) \neq 0$, and $\Delta(H_\alpha) \neq 0$. In fact, for our purpose, it is enough to consider the following necessary condition (we call the condition X) that are going to lead to some simple linear inequalities on $r_i$ and $t_j$:
\begin{itemize}
\item[i)] One of the necessary conditions for $Det(A_\alpha) \neq 0$ is that there is no zero bottom right $i \times (2m+2-i)$ blocks, for any $i$, in $A_\alpha$.  
\item[ii)] By \cite{Sha16} lemma 7.5, if for some $i <2m+1$ the bottom right $i \times (2m+1-i)$ and $(2m+1-i) \times i$ blocks in $A_\alpha$ are zero, then $H_\alpha$ will have discriminant zero.
\end{itemize}
Now we will prove a statement that given $x_i,y_j$ and vary $r_i,t_j$, if $A'$ is maximal and it does not satisfy the Proposition then $r_1=t_1$. Firstly, we assume that $x_1 \geq y_1 >d$. Base on the condition X, we consider the following cases:
\begin{itemize}
\item[Case 1:]There exist $e$ and $f$ bigger than $1$ such that $x_1-y_e \leq d$, $x_1-y_{e+1} > d$, $y_1-x_f \leq d$, and $y_1-x_{f+1}>d$. Then the condition X implies that 
$$t_1+t_2+\cdots + t_e \geq r_1+\cdots+r_{e'}$$and$$r_1+r_2+\cdots+r_f \geq t_1+\cdots + t_{f'},$$where $e'$ is the biggest number satisfying that $x_{e'}-y_{e+1}>d$, and similarly, $f'$ is the biggest number satisfying that $y_{f'}-x_{f+1}>d$. If $e'>1$ then by fixing everything except $r_1$ and $r_2$, we observe that $A'$ is maximal when $r_1=0$ or $r_2=0$. By induction, $A'$ will satisfy the proposition. Similarly for the case $f'>1$, hence we can assume that $e'=f'=1$. If $r_1+r_2+\cdots +r_f \geq t_1+t_2$ then by using the same argument as before, we conclude that $A'$ is bigger if $t_1=0$ or $t_2=0$, thus $A'$ will satisfy the proposition. If $r_1+r_2+\cdots +r_f < t_1+t_2$, then the condition $t_1+t_2+\cdots+t_e \geq r_1$ can be ignored. As a result, $r_1$ and $r_2$ will always go in pair in every inequalities that are implied by the conditions X. So $A'$ will satisfy the proposition in this case.
\item[Case 2]Without loss of generality, we assume that $e=1$ and $f>1$, then 
$$t_1 \geq r_1+\cdots+r_{e'}$$$$r_1+r_2+\cdots+r_f \geq t_1+\cdots + t_{f'},$$where $e',f'$ are defined in the same way as above. Similar to the case 1, if $e'>1$ then $A'$ will satisfy the Proposition. \\
If $f'>1$, then $A'$ is maximal only if $t_1=r_1$ or $t_1=t_1+t_2$. So if $A'$ is maximal and does not satisfy the Proposition then $t_1=r_1$. \\
If $f'=1$ and $r_1+\cdots +r_f  \geq t_1+t_2$, then we will have the same conclusion as the case $f'>1$ above. \\ 
If $f'=1$ and $r_1+\cdots +r_f  < t_1+t_2$, then by considering the pair $(t_1,t_2)$, we imply that $A'$ is maximal only if $t_1=x_1$ or $t_1=x_1+\cdots x_f$. In the later case, we argue similarly as the case $e'>1$ to conclude that $A'$ satisfies the Proposition.
\item[Case 3]If $e=f=1$, we can deduce that $r_1=t_1$ from the conditions X.
\end{itemize}
By removing all parts related to $X_1$ and $Y_1$, and then apply the same argument as above, it can be seen that $A'$ is maximal and it does not satisfy the Proposition only if $r_2=t_2$. Continue this way we obtain that the only case we need to take care is the case $h=l$ and $r_i=t_i$ for all $0 \leq i \leq h$. In this case, we could also assume that $x_1-y_1 \leq d$, $x_1-y_2>d$, and let $f$ is the number between $2$ and $h$ satisfying $y_1-x_f \leq d$, and $y_1-x_{f+1}>d$ (here $x_{h+1}:=x_0=0$), then the power of $q$ related to $X_1$ and $Y_1$ in $A'$ can be approximated as follows:
\begin{align*}
e &=\sum_{i=1}^h \big(h^0(X_i\otimes Y_1 \otimes \mathcal{L})+h^0(X_1^*\otimes Y_1 \otimes \mathcal{L})\big) +h^0(X_0\otimes Y_1 \otimes \mathcal{L}) + \\ &\hspace{1cm}+ h^0(X_1\otimes Y_0 \otimes \mathcal{L}) + \sum_{i=2}^h \bigg( h^0(X_1\otimes Y_i \otimes \mathcal{L}) +h^0(X_1 \otimes Y_i^* \otimes \mathcal{L}) \bigg) - \\ &\hspace{1cm}-\sum_{i=2}^h \big(h^0(X_1\otimes X_i)+h^0(X_1\otimes X_i^*) + h^0(Y_1\otimes Y_i)+h^0(Y_1\otimes Y_i^*) \big)+\\
& \hspace{1cm}+\sum_{i=1}^f h^0(X_i\otimes Y_1^* \otimes \mathcal{L}) - h^0(\wedge^2X_1)-h^0(\wedge^2Y_1)-h^0(X_1\otimes X_0)- \\ & \hspace{8cm}-h^0(Y_1\otimes Y_0)-4(r_1n-r_1^2)\\
&\approx r_1x_1(r_1+2r_2+\cdots +2r_h+r_0)+r_1y_1(r_1+\cdots +r_f+2r_{f+1}+\cdots+2r_h+r_0)\\ 
&\hspace{0.5cm}+r_1(r_2x_2+\cdots + r_fx_f)-r_1d(r_1+3r_2+\cdots+3r_f+4r_{f+1}+\cdots+4r_h+2r_0)\\
&\hspace{6cm}-r_1(x_1+y_1)(r_1-1+2r_2+\cdots+2r_h+r_0) \\
&=r_1\big(x_1-y_1(r_2+\cdots+r_f-1)+(r_2x_2+\cdots + r_fx_f)-d(r_1+3r_2+\cdots+3r_f+\\
& \hspace{9.5cm}+4r_{f+1}+\cdots+4r_h+2r_0)\big)
\end{align*}
Notice that $x_i$ and $y_i$ need to satisfy the following conditions:
\begin{align*}
|x_i-y_i| &\leq d \\
|x_i-x_{i+1}| &\leq 2d \\
|y_i-y_{i+1}| &\leq 2d \\
x_h \leq d \hspace{0.5cm} &\text{or} \hspace{0.5cm} y_h \leq d.
\end{align*}
If $y_1 \leq x_f$, $r_i \neq 1$ for some $i$ then 
\begin{align*}
e/r_1 & \approx x_1+y_1+\sum_{i=2}^fr_i(x_i-y_1)-d(r_1+3r_2+\cdots+3r_f+4r_{f+1}+\cdots+4r_h+2r_0) \\
&\leq 2y_1+d+\sum_{i=2}^fr_id-d(r_1+3r_2+\cdots+3r_f+4r_{f+1}+\cdots+4r_h+2r_0) \\
&\leq 2x_f+d-d(r_1+2r_2+\cdots+2r_f+4r_{f+1}+\cdots+4r_h+2r_0) \\
& \leq d(4h-4f+5)-d(4h-2f+2) \\
&\leq-d 
\end{align*}
If $y_1 \leq x_f$, $r_i= 1$ for all $i$ then 
\begin{align*}
e/r_1 & \approx x_1+y_1+\sum_{i=2}^fr_i(x_i-y_1)-d(r_1+3r_2+\cdots+3r_f+4r_{f+1}+\cdots+4r_h+2r_0) \\
&\leq 2x_f+d+\sum_{i=2}^{f-1}r_id-d(r_1+3r_2+\cdots+3r_f+4r_{f+1}+\cdots+4r_h+2r_0) \\
&\leq 2x_f+d-d(r_1+2r_2+\cdots+2r_{f-1}+3r_f+4r_{f+1}+\cdots+4r_h+2r_0) \\
& \leq d(4h-4f+5)-d(4h-2f+2) \\
&\leq-d 
\end{align*}
If there exist $2\leq t \leq f-1$ such that $x_t \geq y_1 > x_{t+1}$, then
\begin{align*}
e/r_1 & \approx x_1+y_1+\sum_{i=2}^fr_i(x_i-y_1)-d(r_1+3r_2+\cdots+3r_f+4r_{f+1}+\cdots+4r_h+2r_0) \\
& \leq 2x_t+d+d(r_2+\cdots+r_t)-d(r_1+3r_2+\cdots+3r_f+4r_{f+1}+\cdots+4r_h+2r_0) \\
& \leq d(4h-4t+5)-d(r_1+2r_2+\cdots+2r_t+3r_{t+1}+\cdots+3r_f+4r_{f+1}+\cdots+\\ &\hspace{12cm}+4r_h+2r_0) \\
& \leq d(4h-4t+5)-d(4h-2t+2) \\
&<-d
\end{align*}
If $y_1 > x_2$, and $f>2$, then
\begin{align*}
e/r_1 & \approx x_1+y_1+\sum_{i=2}^fr_i(x_i-y_1)-d(r_1+3r_2+\cdots+3r_f+4r_{f+1}+\cdots+4r_h+2r_0) \\
& \leq 2x_f+3d-d(r_1+3r_2+\cdots+3r_f+4r_{f+1}+\cdots+4r_h+2r_0) \\
& \leq d(4h-4f+7)-d(4h-f) \\
&<-d
\end{align*}
If $y_1 > x_2$, and $f=2$, then
\begin{align*}
e/r_1 & \approx x_1+y_1+r_2(x_2-y_1)-d(r_1+3r_2+\cdots+3r_f+4r_{f+1}+\cdots+4r_h+2r_0) \\
& \leq x_1+x_2-d(4h-f+1) \hspace{1cm}\text{(since $\exists i$ such that $r_i \neq 1$)}\\
& \leq d(4h-4f+6)-d(4h-f+1) \\
&=-d
\end{align*}
We have just finished the proof of the Proposition with the assumption $x_1\geq y_1>d$. 

Now let consider the case $x_1>d$ and $y_1\leq d$. So the condition X becomes to
$$t_1 +\dots +t_e \geq r_1+\dots r_{e'},$$
where $e$ is the biggest number satisfying that $x_1-y_e \leq d$ ($e$ exists since $x_1-y_1 \leq d$ and $x_1>d$), and $e'$ is the biggest number satisfying that $x_{e'}-y_{e+1}>d$. Using exactly the same argument as in the case $y_1>d$, we deduce that $e=e'=1$ is the necessary condition to make the value of $A'$ maximal. Additionally, if we assume that $A'$ does not satisfy the Proposition, then by induction and the same argument as the first case, we deduce that $h=l$ and $r_i=t_i$ for all $i$. Now we will finish this case by considering the part related to $X_1$ and $Y_1$ in $A'$: let denote $f$ be the biggest number such that $y_1+x_f>d$, then the power of $q$ related to $x_1$ and $y_1$ in $A'$ is \small
\begin{align*}
e &=\sum_{i=1}^h \big(h^0(X_i\otimes Y_1 \otimes \mathcal{L})+h^0(X_1^*\otimes Y_1 \otimes \mathcal{L})\big) +h^0(X_0\otimes Y_1 \otimes \mathcal{L}) + h^0(X_1\otimes Y_0 \otimes \mathcal{L}) +\\
&\hspace{0.5cm}+ \sum_{i=2}^h \big( h^0(X_1\otimes Y_i \otimes \mathcal{L}) +h^0(X_1 \otimes Y_i^* \otimes \mathcal{L}) \big) +\sum_{i=0}^h h^0(X_i\otimes Y_1^* \otimes \mathcal{L}) + \\
&\hspace{0.5cm}+\sum_{i=f+1}^h h^0(X_i^*\otimes Y_1^* \otimes \mathcal{L}) - h^0(\wedge^2X_1)- h^0(\wedge^2Y_1)-h^0(X_1\otimes X_0)-h^0(Y_1\otimes Y_0)- \\ 
&\hspace{0.5cm}-\sum_{i=2}^h \big(h^0(X_1\otimes X_i)+h^0(X_1\otimes X_i^*) + h^0(Y_1\otimes Y_i)+h^0(Y_1\otimes Y_i^*) \big) -4d(r_1n-r_1^2)\\
&\approx r_1x_1(r_1+2r_2+\cdots +2r_h+r_0)+r_1y_1(r_1+\cdots +r_f)+\\ 
&\hspace{0.5cm}+r_1(r_2x_2+\cdots + r_fx_f)-r_1d(r_1+3r_2+\cdots+3r_f+2r_{f+1}+\cdots+2r_h+r_0)\\
&\hspace{0.5cm}-r_1(x_1+y_1)(r_1-1+2r_2+\cdots+2r_h+r_0) \\
&=r_1\big(x_1-y_1(r_2+\cdots+r_f+2r_{f+1}+\cdots+2r_h+r_0-1)+(r_2x_2+\cdots + r_fx_f)- \\ &\hspace*{2cm}-d(r_1+3r_2+\cdots+3r_f+2r_{f+1}+\cdots+2r_h+r_0)\big) \\
&<r_1\big(2d+2d(r_2+\cdots+r_f)- d(r_1+3r_2+\cdots+3r_f+2r_{f+1}+\cdots+2r_h+r_0)\big) \\
& \leq -d
\end{align*} \normalsize
By induction, we have just proved the Proposition in the case $x_1>d$ and $y_1\leq d$. 

The last case we need to consider is $x_1<d$, $y_1<d$, and $x_1+y_1>d$. In this case, the condition X becomes empty. Hence, if $h$ is different than $1$, by fixing the sum $r_1+r_2$ and using the same argument as above, we deduce that $A'$ satisfies the Proposition. Similar story for the case $l \neq 1$. Thus we only need to consider the case $h=l=1$, but we have already treated this case at the beginning of this section. 

The proof is completed.
\end{proof}
\begin{remark}
From the above discussion we can see that the case that could contribute a positive portion in the average, is the case where $r_i$ and $t_j$ are all number $1$, and also the differences $x_i-x_{i+1}$ and $y_j-y_{j+1}$ are close to $2d$. We have two ideas cases as follows:
\begin{itemize}
\item[Kostant 1:] $x_i=2m-2i+2$,$y_i=2m-2i+1$ for all $1 \leq i \leq m$;
\item[Kostant 2:] $y_i=2m-2i+2$,$x_i=2m-2i+1$ for all $1 \leq i \leq m$.
\end{itemize}  
The reason we named them Kostant is that they reflex the role of two Kostant sections in our average.
\end{remark}
Based on the above remark, we divide the remaining case into some cases as follows:
\item[Case 6:] If $h=l=m$, i.e. $X_i$ and $Y_j$ are all line bundles, $(4m-3)d <x_1+x_2<(4m-2)d $. In this case the relating $X_1,Y_1$ part of $A$ is $\frac{n_1}{d_1}$ where \small
\begin{align*}
n_1&= \prod_{i=2}^{m}\bigg(|H^0(X_i\otimes Y_1\otimes \mathcal{L})|.|H^0(X_i^*\otimes Y_1\otimes \mathcal{L})|.|H^0(X_1\otimes Y_i\otimes \mathcal{L})|.|H^0(X_1\otimes Y_i^*\otimes \mathcal{L})|\bigg) \\
& \hspace{1cm} \times|H^0(X_1\otimes Y_1\otimes \mathcal{L})|.|H^0(X_0\otimes Y_1\otimes \mathcal{L})|.|H^0(X_1\otimes Y_0\otimes \mathcal{L})|.|H^0(X_1\otimes Y_1^*\otimes \mathcal{L})|\\ & \hspace{10.5cm} \times |H^0(X_2\otimes Y_1^*\otimes \mathcal{L})| \\&= |H^0(X_1^*\otimes Y_1\otimes \mathcal{L})|.|H^0(X_2\otimes Y_1^*\otimes \mathcal{L})|.q^{nx_1+(n-2)y_1+(2n-2)d+(2n-2)(1-g)} 
\end{align*}
 \normalsize and 
\begin{align*}
d_1&= q^{(n^2-(n-2)^2)d+2(1-g)+x_1(n-2)+y_1(n-2)+(2n-4)(1-g)}.|Aut(X_1)||Aut(Y_1)| \\
&=(q-1)^2q^{(4n-4)d+(2n-2)(1-g)+x_1(n-2)+y_1(n-2)}.
\end{align*}
Hence the contribution of this range to the average is bounded above by 
\begin{align*}
&\hspace{0.4cm}\sum_{d/2+x_2<y_1<d+x_2} \hspace{0.4cm}\sum_{y_1\leq x_1<y_1+d}  \frac{n_1(X_1,Y_1)}{d_1(X_1,Y_1)} \\
&= \sum_{d/2+x_2<y_1<d+x_2} \hspace{0.4cm}\sum_{y_1 \leq x_1 < y_1+d } \frac{|H^0(X_1^*\otimes Y_1\otimes \mathcal{L})|.|H^0(X_2\otimes Y_1^*\otimes \mathcal{L})|q^{2x_1}}{(q-1)^2q^{(2n-2)d}} \\
&\leq \sum_{d/2+x_2<y_1<d+x_2} \hspace{0.4cm}\sum_{y_1 \leq x_1 < y_1+d } \frac{T}{(q-1)^2q^{x_1+x_2-(2n-4)d}} \\
&\leq \sum_{d/2+x_2<y_1<d+x_2}  \frac{2T}{(q-1)^2q^{y_1+d-1+x_2-(2n-4)d}}\\
&\leq \frac{2T}{(q-1)^2} \hspace{2cm}, 
\end{align*}
where $T$ is a constant that is only depended on $C$.
\item[Case 7:] If $X_i$ and $Y_j$ are all line bundles, $x_i=2m-2i+2$,$y_i=2m-2i$ for all $1 \leq i \leq m$. Then $deg(X_2\otimes Y_1^*\otimes \mathcal{L}) = 0$, and therefore it will has no non-trivial global sections if it is a non-trivial line bundle. And if $H^0(X_2\otimes Y_1^*\otimes \mathcal{L}) = 0$, we can see that $\Delta(H_\alpha)=0$ for all $\alpha \in H^0\big((\mathcal{E}\times_G W)\otimes \mathcal{L}\big)$. Hence we can assume that $X_2 \otimes Y_1^* = \mathcal{L}^*$. On the other hand, to make sure that $det(A_\alpha) \neq 0$, we need to have that $X_1^* \otimes Y_1 = \mathcal{L}^*$. Similarly, we will obtain the following necessary conditions: $X_i \cong \mathcal{L}^{2m-2i+2}, Y_i \cong \mathcal{L}^{2m-2i}.$ We now will show that any regular sections will factor through the first Kostant section. Firstly, let recall the form of $A_\alpha$ for any $\alpha \in H^0((\mathcal{E}\times^G W)\otimes \mathcal{L})$ satisfying $det(A_\alpha) \neq 0$ and $\Delta(H_\alpha) \neq 0$:
\begin{equation}
A_\alpha = \begin{pmatrix}
  * & \cdots & * & * & * & * & \cdots & * \\
  * & \cdots & * & * & * &  * &\cdots & c_1 \\
  \vdots & \vdots & \vdots & \vdots & \vdots & \vdots & \reflectbox{$\ddots$} & \vdots \\
  * & \cdots & * & * & * & c_{m-1} & \cdots & 0 \\
  * & \cdots & * & * & c_m & 0 & \cdots & 0 \\
  * & \cdots & c_{m}' & 0 & 0 & 0 & \cdots & 0 \\
   \vdots & \reflectbox{$\ddots$} & \vdots & \vdots & \vdots & \vdots & \vdots & \vdots \\
  c_1' & \cdots & 0 & 0 & 0 & 0 & \cdots & 0 \\  
 \end{pmatrix}
\end{equation}where $c_i$and $c_i'$ are non-zero constants. Notice that the action of $G$ on $W$ is
$$\begin{pmatrix}
B & 0_n \\
0_n & C
\end{pmatrix}. \begin{pmatrix}
0_n & A_\alpha \\
-A_\alpha^* & 0_n
\end{pmatrix}.\begin{pmatrix}
B^* & 0_n \\
0_n & C^*
\end{pmatrix}= \begin{pmatrix}
0_n & BA_\alpha C^* \\
-CA_{\alpha}^*B^* & 0_n
\end{pmatrix}.
$$
Since $c_i$ and $c_i'$ are non-zero, we can make them to be $1$ as follows: take $B$ to be a diagonal matrix with the diagonal entries $b_{ii}= $, then we will have
\footnotesize
\begin{align*}
A_\alpha ' & =\begin{pmatrix}
  \prod_{i=1}^m(c_ic_i') &  &  &  &  &  &   \\
   & \ddots &  &  &  &  & \\
   &  &   c_m'c_m & &  & &    \\
   &  &    & 1 &  & &    \\
   &  &    &  & (c_m'c_m)^{-1} &  &    \\
   &  &    &  &  & \ddots & \\
   &  &    &  &&&  \big(\prod_{i=1}^m(c_ic_i')\big)^{-1} 
 \end{pmatrix}.A_\alpha \\
 &= \begin{pmatrix}
  * & \cdots & * & * & * & * & \cdots & * \\
  * & \cdots & * & * & * &  * &\cdots & b_1 \\
  \vdots & \vdots & \vdots & \vdots & \vdots & \vdots & \reflectbox{$\ddots$} & \vdots \\
  * & \cdots & * & * & * & b_{m-1} & \cdots & 0 \\
  * & \cdots & * & * & b_m & 0 & \cdots & 0 \\
  * & \cdots & b_{m} & 0 & 0 & 0 & \cdots & 0 \\
   \vdots & \reflectbox{$\ddots$} & \vdots & \vdots & \vdots & \vdots & \vdots & \vdots \\
  b_1^{-1} & \cdots & 0 & 0 & 0 & 0 & \cdots & 0 \\  
 \end{pmatrix}.
\end{align*} 
\normalsize
After that we multiply the right hand side of $A_\alpha$ with $$C=diag(b_1,b_2, \cdots, b_m,1,b_m^{-1},\cdots, b_1^{-1}) \in SO(V_2),$$then the resulting matrix will have the property we mentioned before. Now we can assume that in our matrix $A_\alpha$, the entries $c_i$ and $c_i'$ are all equal to $1$. We continue to multiply the left and the right of $A_\alpha$ by some orthogonal matrices to transform it into Kostant form as follows:
\begin{itemize}
\item[Step 1:] We firstly transfer the entries $a_{1,i}$ for $m+2\leq i \leq n$ into zero by multiplying on the left of $A_\alpha$ by the following special orthogonal matrix (upper triangular matrix)
\begin{equation}
B =\begin{pmatrix}
  1 & -a_{1,n} & \cdots & -a_{1,m+3} & -a_{1,m+2} & 0 & \cdots & 0 & -a_{1,m+2}^2/2 \\
0 & 1 & \cdots & 0 &0 & 0 & \cdots & 0 & 0 \\  
 \vdots  & \vdots & \vdots & \vdots & \vdots & \vdots & \vdots & \vdots & \vdots\\
0  & 0 & \cdots &1& 0 & 0 &\cdots & 0 & 0   \\ 
 0  & 0 & \cdots & 0 & 1 & 0 &\cdots & 0 & a_{1,m+2}   \\
 \vdots  & \vdots & \vdots & \vdots & \vdots & \vdots & \vdots & \vdots & \vdots\\
 0  & 0 & \cdots   & 0 & 0 & 0 & \cdots & 1 & a_{1,n}  \\
 0  & 0 & \cdots   & 0 &0 & 0 & \cdots & 0 & 1 \\
  \end{pmatrix}.
\end{equation}
\item[Step 2:] Transfer the entries $a_{j,1}$, where $m+2 \leq j \leq n-1$, into zero. In this step we need to multiply the right hand side of $A_\alpha$ by the following lower triangular matrix:
\begin{equation}
C =\begin{pmatrix}
  1 & 0 & \cdots & 0 & 0 & 0 & \cdots & 0 & 0 \\
-a_{n-1,1} & 1 & \cdots & 0 &0 & 0 & \cdots & 0 & 0 \\  
 \vdots  & \vdots & \vdots & \vdots & \vdots & \vdots & \vdots & \vdots & \vdots\\
-a_{m+2,1}  & 0 & \cdots &1& 0 & 0 &\cdots & 0 & 0   \\ 
 0  & 0 & \cdots & 0 & 1 & 0 &\cdots & 0 & 0   \\
 \vdots  & \vdots & \vdots & \vdots & \vdots & \vdots & \vdots & \vdots & \vdots\\
 0  & 0 & \cdots   & 0 & 0 & 0 & \cdots & 1 & 0  \\
 0  & 0 & \cdots   & 0 &0 & a_{m+2,1} & \cdots & a_{n-1,1} & 1 \\
  \end{pmatrix}.
\end{equation}
\item[Step 3:] If $n=3$ then we can skip this step and go directly to the step $4$. Hence by using an inductive argument, we can assume that our statement is true for $n$ small. More precisely, if we consider the submatrix $A_\alpha'$ obtained by removing the first column, the first row, the last row, and the last column, then we can transfer it into the Kostant form by multiplying (canonically) on the left by some (upper triangular) special orthogonal matrices, and on the right by some (lower triangular) special orthogonal matrices. From this induction, by multiplying $A_\alpha$ by 
$$ B=\begin{pmatrix}
1 && \\
&B'& \\
&&1
\end{pmatrix} \text{on the left, and}$$
$$C=\begin{pmatrix}
1 && \\
&C'& \\
&&1
\end{pmatrix} \text{on the right,}$$where $B'$ and $C'$ are appropriate upper and lower triangular matrices, respectively, $A_\alpha$ will have the following form:
\begin{equation}
A_\alpha= \begin{pmatrix}
  * & * & \cdots & * & * & 0 & 0 & \cdots &0& 0 \\
  * & * & \cdots & * & * & 0 & 0 &\cdots & 0&1 \\
  * & 0 & \cdots & 0 & 0 & 0 & 0 & \cdots & 1&0 \\
  \vdots & \vdots & \vdots & \vdots & \vdots & \vdots & \vdots & \reflectbox{$\ddots$} &\vdots & \vdots \\
  * & 0 & \cdots & 0 & 0 & 0 & 1 & \cdots & 0&0 \\
  * & *& \cdots & * & * & 1 & 0 & \cdots & 0&0 \\
  0 & 0 & \cdots & 1 & 0 & 0 & 0 & \cdots & 0&0 \\
   \vdots & \vdots&  \vdots & \vdots & \vdots & \vdots & \vdots & \vdots & \vdots &\vdots \\  
  1 & 0 & \cdots & 0 & 0 & 0 & 0 & \cdots & 0 &0 \\  
 \end{pmatrix}.
\end{equation}
 We need to emphasise the upper and lower triangular properties here because they help us to keep the entries of $A_\alpha$ considered in the previous steps to be equal to zero. 
\item[Step 4:] Multiply on the right of $A_\alpha$ by  
\begin{equation}
\begin{pmatrix}
  1 & 0 & \cdots & 0 & 0 & 0 & \cdots & 0 & 0 \\
0 & 1 & \cdots & 0 &0 & 0 & \cdots & 0 & 0 \\  
 \vdots  & \vdots & \vdots & \vdots & \vdots & \vdots & \vdots & \vdots & \vdots\\
0  & 0 & \cdots &1& 0 & 0 &\cdots & 0 & 0   \\ 
 a_{2,m+1}  & 0 & \cdots & 0 & 1 & 0 &\cdots & 0 & 0   \\
 \vdots  & \vdots & \vdots & \vdots & \vdots & \vdots & \vdots & \vdots & \vdots\\
 a_{2,2}  & 0 & \cdots   & 0 & 0 & 0 & \cdots & 1 & 0  \\
 -a_{2,m+1}^2/2  & -a_{2,2} & \cdots   & -a_{2,m} &-a_{2,m+1} & 0 & \cdots & 0 & 1 \\
  \end{pmatrix}
\end{equation}
will make the entries $a_{2,i}$ for $2 \leq i \leq m+1$ to be zero.
\item[Step 5:] Finally, multiply on the left of $A_\alpha$ by
\begin{equation}
\begin{pmatrix}
  1 & 0 & \cdots & 0  & a_{m,1} & \cdots & a_{2,1} & 0 \\
0 & 1 & \cdots & 0 &0 & \cdots & 0 & -a_{2,1} \\  
 \vdots  & \vdots  & \vdots & \vdots & \vdots & \vdots & \vdots & \vdots\\
0  & 0 & \cdots &1& 0 &\cdots & 0 & -a_{m,1}   \\ 
 0  & 0 & \cdots & 0  & 1 &\cdots & 0 & 0   \\
 \vdots  & \vdots  & \vdots & \vdots & \vdots & \vdots & \vdots & \vdots\\
 0  & 0 & \cdots   & 0 & 0 & \cdots & 1 & 0  \\
 0  & 0 & \cdots   & 0 & 0 & \cdots & 0 & 1 \\
  \end{pmatrix},
\end{equation}
then the entries $a_{j,1}$ for $2 \leq j \leq m$ will be zero. Our matrix $A_\alpha$ is of the Kostant form now.
\end{itemize}
\item[Case 8:] Similar to the case $7$ but we switch $X_i$ to $Y_i$ and $Y_i$ to $X_i$. Then we can prove that any regular sections of $(\mathcal{E}\times_GW)\otimes \mathcal{L}$ will factor through the second Kostant section. Hence the contribution of this case to the average is $1$.
\item[Case 9:] If $ d-2g+2 \leq x_1+y_1 \leq d$. By Proposition \ref{global sections of semistable} we are able to bound the dimension of $H^0((\mathcal{E}\times^GV)\otimes \mathcal{L})$ as follows:
\begin{eqnarray*}
h^0((\mathcal{E}\times^GV)\otimes \mathcal{L})  \hspace{8cm}\\
 \leq \sum_{i=1}^h \sum_{j=1}^l \big( h^0(X_i \otimes Y_j \otimes \mathcal{L})+h^0(X_i^* \otimes Y_j \otimes \mathcal{L})+h^0(X_i \otimes Y_j^* \otimes \mathcal{L})\\ \hspace{0.5cm}+ h^0(X_i^* \otimes Y_j^* \otimes \mathcal{L})\big) + \sum_{j=1}^l \big(h^0(X_0\otimes Y_i \otimes \mathcal{L} +h^0(X_0\otimes Y_i^* \otimes \mathcal{L}) \big)\\ \hspace{0.5cm}+ \sum_{i=1}^h\big(h^0(X_i \otimes Y_0 \otimes \mathcal{L}) + h^0(X_i^*\otimes Y_0\otimes \mathcal{L})\big) +h^0(X_0\otimes Y_0 \otimes \mathcal{L}) \hspace{0.5cm}\\
\leq n^2+d.n^2. \hspace{9.5cm}
\end{eqnarray*}
Furthermore, if we fix the rank $r_1$ of the semi-stable vector bundle $X_1$, there exists a constant $A_1$ such that for any integer $d_1$ we have that $|Bun^{semi-stable}_{r_1,d_1}(\mathbb{F}_q)| \leq A_1$. In fact, set $d_1=a.r_1+d_2$ for some $0\leq d_2 <r_1$, then $|Bun^{semi-stable}_{r_1,d_1}(\mathbb{F}_q)| = |Bun^{semi-stable}_{r_1,d_2}(\mathbb{F}_q)|$ because of the assumption that our curve $C$ has an $\mathbb{F}_q-$ rational point. Notice that $|Bun^{semi-stable}_{r_1,d_2}(\mathbb{F}_q)|$
is finite for any $d_2$, hence we can choose $A_1$ to be the maximal number among $|Bun^{semi-stable}_{r_1,d_2}(\mathbb{F}_q)|$ for $0\leq d_2 <r_1$. We also can choose the common bound $A_1$ for $|Bun^{semi-stable}_{r_1,d_1}(\mathbb{F}_q)|$ when $r_1$ varies in the period $[1,m]$.

Now if we fix a parabolic subgroup $P$ of $G$, recall that $Bun^P$ is the notation of the set of $G-$bundles whose canonical reductions are reductions to $P$. Then the contribution of $Bun^P$ to the average in this case is:
\begin{eqnarray*}
& \mathlarger{\int}\limits_{\substack{\mathcal{E} \in Bun^P \\ d-2g+2 \leq x_1+y_1 \leq d}} \frac{|H^0\big((\mathcal{E}\times^G V)\otimes \mathcal{L}\big)|}{|\mathcal{A}_{\mathcal{L}}(\mathbb{F}_q)|} d\mathcal{E} \hspace{3cm} \\ & \leq c. \sum\limits_{d-2g+2 \leq x_1+y_1 \leq d} \frac{A_1. q^{n^2+d.n^2}}{q^{r_1x_1(r_1-1+ \cdots +2r_h+r_0)+t_1y_1(t_1-1+\cdots +2t_l+t_0)}.q^{n^2d+n(1-g)}} \\ &\text{(c is a constant that only depends on $P,n$ and the genus $g$ of $C$)} \\
& \leq b.\sum\limits_{d-2g+2 \leq x_1+y_1 \leq d} \frac{1}{q^{r_1x_1+t_1y_1}} \leq b.\sum_{t=d-2g+2}^d \frac{t-1}{q^t}.\hspace{2.7cm}\\ &\text{(b is a constant that only depends on $P,n$ and the genus $g$ of $C$)} 
\end{eqnarray*}  
By taking limit $d \rightarrow  \infty$, the above upper bound implies that the contribution of this case to the average equals zero.

\item[Case 10:] The last case we need consider is $x_1+y_1 <d-2g+2$, i.e. slopes of any consecutive semistable quotients in the "filtration" of the vector bundle $(\mathcal{E}\times^G V) \otimes \mathcal{L}$ are strictly bigger than $2g-2$. Consequently, we obtain the following equality:
\begin{eqnarray*}
&h^0((\mathcal{E}\times^GV)\otimes \mathcal{L}) \hspace{8cm}\\ & = \sum_{i=1}^h \sum_{j=1}^l \big( h^0(X_i \otimes Y_j \otimes \mathcal{L})+h^0(X_i^* \otimes Y_j \otimes \mathcal{L})+h^0(X_i \otimes Y_j^* \otimes \mathcal{L})\\
&+ h^0(X_i^* \otimes Y_j^* \otimes \mathcal{L})\big) + \sum_{j=1}^l \big(h^0(X_0\otimes Y_i \otimes \mathcal{L} +h^0(X_0\otimes Y_i^* \otimes \mathcal{L}) \big) +\\
&+ \sum_{i=1}^h\big(h^0(X_i \otimes Y_0 \otimes \mathcal{L}) + h^0(X_i^*\otimes Y_0\otimes \mathcal{L})\big) +h^0(X_0\otimes Y_0 \otimes \mathcal{L})\\
&= n^2(1-g)+d.n^2. \hspace{9cm}
\end{eqnarray*}
Notice that in case the $G-$bundle $\mathcal{E}$ is semistable, we also have the above equality since $deg(\mathcal{E}\times^G V) =0$. Since the Tamagawa number of $G$ is $4$, by considering the counting measure weighted by the size of automorphism groups on $Bun_G(\mathbb{F}_q)$, we have that 
\begin{eqnarray*}
|Bun_G(\mathbb{F}_q)|& = \int\limits_{Bun_G(\mathbb{F}_q)} 1 d\mu = 4.q^{(4m^2+2m)(g-1)}.\prod\limits_{x \in |C|}\frac{|\kappa(x)|^{dim(G)}}{|G(\kappa(x))|} \\
& = 4.q^{(4m^2+2m)(g-1)}.\zeta_C(2).\zeta_C(4)\dots \zeta_C(2m).
\end{eqnarray*}
Now we can compute the average number in this case as follows:
\begin{eqnarray*}
&&\lim_{d \rightarrow \infty} \dfrac{\mathlarger{\int}\limits_{\substack{Bun_G(\mathbb{F}_q) \\ x_1+y_1 <d-2g+2}}|\mathcal{M}_{L,E}(k)|dE}{|\mathcal{A}_L(k)|} \\
&=& \lim_{d \rightarrow \infty} \dfrac{\mathlarger{\int}\limits_{\substack{Bun_G(\mathbb{F}_q) \\ x_1+y_1 <d-2g+2}}|H^0(C, V(E, L)^{reg}|dE}{q^{n^2d+n(1-g)}} \\
&=& \lim_{d \rightarrow \infty} \dfrac{|H^0(C, V(E,L))|\mathlarger{\int}\limits_{\substack{Bun_G(\mathbb{F}_q) \\ x_1+y_1 <d-2g+2}}\dfrac{|H^0(C, V(E, L)^{reg}|}{|H^0(C, V(E,L))|}dE}{q^{n^2d+n(1-g)}} \\
&\leq & \dfrac{q^{n^2d+n^2(1-g)}\mathlarger{\int}\limits_{\substack{Bun_G(\mathbb{F}_q)}}\mathsmaller{\zeta_C(2)^{-2}\dots \zeta_C(2m)^{-2}.\prod_{v \in |C|}\big(1+c_{2m-1}|k(v)|^{-2}+\dots+c_1|k(v)|^{-2m}\big)}dE}{q^{n^2d+n(1-g)}} \\
&=& \prod_{v \in |C|}\big(1+c_{2m-1}|k(v)|^{-2}+\dots+c_1|k(v)|^{-2m}\big)|Bun_G(\mathbb{F}_q)| \times \\
&& \hspace{4cm} \times q^{(4m^2+2m)(1-g)}\zeta_C(2)^{-2}\dots \zeta_C(2m)^{-2} \\
&=&4.\prod_{v \in |C|}\big(1+c_{2m-1}|k(v)|^{-2}+\dots+c_1|k(v)|^{-2m}\big)
\end{eqnarray*}
\end{description}
\subsubsection{The transversal case}
In this subsection, we consider a special family of hyperelliptic curves, the transversal family. The purpose of this subsection is to show that in the transversal case, we can ignore the case $6$ above. Then, the average in this case will not contain the rational function of $q$. In fact, it is just an easy consequence of Proposition (?) as we can see now:
\begin{proposition}
Assume that $\mathcal{E}$ is a $G-$bundle satisfying the conditions in the case $6$ above, i.e. $X_i$ and $Y_j$ are line bundles for all $i$,$j$, and also $(4m-3)d<x_1+x_2<(4m-2)d$. Suppose that $s$ is a global section of $(\mathcal{E}\times^G V) \otimes \mathcal{L}$. Then, for $d$ large enough, the discriminant section $\Delta(s)$ is not square-free. 
\end{proposition}
\begin{proof}By definition, if $s$ is of the form $\begin{pmatrix}
0_n & A \\ -A^* & 0_n
\end{pmatrix}$, then $\Delta(s)$ is the discriminant of $A.A^*$. It is easy to see that $A.A^*$ is of the following matrix form:
\[ \left( \begin{array}{cccccc}
* & * & \cdots & * & * & *  \\
x_1 & * & \cdots & * & * & *  \\
0 & x_2 & \cdots & * & * & * \\
\vdots & \vdots & \ddots & \vdots & \vdots & \vdots  \\
0 & 0 & \cdots & x_2 & * & *  \\
0 & 0 & \cdots & 0 & x_1 & * 
 \end{array} \right), \]
 where $x_i \in H^0(C, (X_i^*\otimes X_{i+1})\otimes \mathcal{L}^{\otimes 2})$ and here $X_{n+1}:= X_0$. The necessary conditions of $det(s) \neq 0$, combine with the hypothesis, imply that $d<x_i-x_{i+1} \leq 2d$ for all $i$, and there is at least one index $i$ such that $d<x_i-x_{i+1} < 2d$. Thus, the Proposition is a consequence of Proposition ?. 
\end{proof}
Now we will consider the case $10$ in the transversal case. The only difference here is the density of the regular locus as we can see as follows:
\begin{eqnarray*}
&&\lim_{d \rightarrow \infty} \dfrac{\mathlarger{\int}\limits_{\substack{Bun_G(\mathbb{F}_q) \\ x_1+y_1 <d-2g+2}}|\mathcal{M}^{trans}_{L,E}(k)|dE}{|\mathcal{A}^{trans}_L(k)|} = \lim_{d \rightarrow \infty} \dfrac{\mathlarger{\int}\limits_{\substack{Bun_G(\mathbb{F}_q) \\ x_1+y_1 <d-2g+2}}H^0(C, V^{reg}(E, L))^{sf} dE}{\frac{|\mathcal{A}^{trans}_L(k)|}{|\mathcal{A}_L(k)|}.|\mathcal{A}_L(k)|}\\
&=& \lim_{d \rightarrow \infty} \dfrac{|H^0(C,V(E,L))|.\mathlarger{\int}\limits_{\substack{Bun_G(\mathbb{F}_q) \\ x_1+y_1 <d-2g+2}} \frac{|H^0(C, V^{reg}(E, L))^{sf}}{|H^0(C,V(E,L))|}| dE}{q^{n^2d+n(1-g)}.\frac{|\mathcal{A}^{trans}_L(k)|}{|\mathcal{A}_L(k)|}} \\
&=& \lim_{d \rightarrow \infty} \dfrac{q^{n^2d+n^2(1-g)}.\mathlarger{\int}\limits_{\substack{Bun_G(\mathbb{F}_q) \\ x_1+y_1 <d-2g+2}} \frac{|H^0(C, V^{reg}(E, L))^{sf}}{|H^0(C,V(E,L))|}| dE}{q^{n^2d+n(1-g)}.\frac{|\mathcal{A}^{trans}_L(k)|}{|\mathcal{A}_L(k)|}} \\
&=& \lim_{d \rightarrow \infty} q^{(n^2-n)(1-g)}\mathlarger{\int}\limits_{\substack{Bun_G(\mathbb{F}_q) \\ x_1+y_1 <d-2g+2}}\dfrac{\dfrac{|H^0(C, V^{reg}(E, L))^{sf}|}{|H^0(C, V(E,L))|}dE}{\frac{|\mathcal{A}^{trans}_L(k)|}{|\mathcal{A}_L(k)|}} \\
&= & q^{(n^2-n)(1-g)}\mathlarger{\int}\limits_{\substack{Bun_G(\mathbb{F}_q)}}\zeta_C(2)^{-2}\dots \zeta_C(2m)^{-2} dE \\
&=&q^{(4m^2+2m)(1-g)}\zeta_C(2)^{-2}\dots \zeta_C(2m)^{-2}.|Bun_G(\mathbb{F}_q)| \\
&=& 4
\end{eqnarray*}
\subsection{Proof of main theorems}
In the transversal case, we have already shown that $|Sel_2(J)| = |H^1(C, \mathcal{J}2])|$, and also by 3.33 the error term, that comes the case $6$ of the counting section, will not appear in the transversal case. Hence the proof of the theorem 3.14 can be obtained by the above observations and the computation in the previous section. 

The theorem 3.13 for the general case also can be proved by the same manner as the transversal case. Notice that in this theorem we need to put an extra condition on $q$, that is  $q > 16^{\frac{m^2(2m+1)}{2m-1}}$, because we only have the inequalities between  
 $|Sel_2(J)|$ and $|H^1(C, \mathcal{J}2])|$ as in Proposition 3.11 but not the equality as in the transversal case.
 
Finally, to obtain the average size of 2-Selmer groups of hyperelliptic curves, we need to take care of the minimal locus. By looking at the counting section, we can see that if we restrict to the minimal locus, then we only have some changes as follows: in the case $6$ the fractional function of $q$ will has an extra factor $\zeta_C((2m+1)^2)$. And in the case $10$, we use Proposition 3.29 instead of 3.27 to obtain
$$4. \zeta_C((2m+1)^2). \prod_{v \in |C|}\big(1+c_{2m-1}|k(v)|^{-2}+\dots+c_1|k(v)|^{-2m}-2|k(v)|^{(2m+1)^2}\big).$$
To sum up, we have just proved the following theorem:
\begin{theorem}
Suppose that $q > 16^{\frac{m^2(2m+1)}{2m-1}}$ and $char(q) > 3$. Then we have that
\begin{eqnarray*}
&\limsup\limits_{deg(\mathcal{L}) \rightarrow \infty} \frac{\sum \limits_{\substack{\text{H is hyperelliptic} \\ \mathcal{L}(H) \cong \mathcal{L}}} \frac{|Sel_2(H)|}{|Aut(H)|}}{\sum \limits_{\substack{\text{H is hyperelliptic} \\ \mathcal{L}(H) \cong \mathcal{L}}} \frac{1}{|Aut(H)|}} \\
&\leq 4. \zeta_C((2m+1)^2). \prod\limits_{v \in |C|}\big(1+c_{2m-1}|k(v)|^{-2}+\dots+c_1|k(v)|^{-2m}-2|k(v)|^{(2m+1)^2}\big) \\ & \hspace{10cm}+2 + f(q),
\end{eqnarray*}where $\lim_{q \rightarrow \infty} f(q) =0$, and $c_i$ are constants which are only depended on $m$ and $p$. If $p>2m+1$ then $c_i$ is only depended on $m$.

\end{theorem}
\newpage

\bibliographystyle{alpha}
\bibliography{counting}
\addcontentsline{toc}{section}{References}

\end{document}